\documentclass[10pt,reqno]{amsart}

\usepackage[margin=1in]{geometry}
\usepackage{amsfonts,amsmath,amssymb,enumerate, amscd}    
\usepackage{amsthm}    
\usepackage{mathrsfs}  
\usepackage{setspace}
\usepackage{stmaryrd}
\usepackage{tikz}
\usepackage{mathtools}
\usepackage[shortlabels]{enumitem}
\usepackage{tikz-cd}

\onehalfspace
\numberwithin{equation}{section}
\theoremstyle{plain}
\newtheorem{thm}{Theorem}[section]
\newtheorem{lem}[thm]{Lemma}
\newtheorem{prop}[thm]{Proposition}
\newtheorem{cor}[thm]{Corollary}
\theoremstyle{definition}
\newtheorem{defn}[thm]{Definition}
\newtheorem{defprop}[thm]{Definition-Proposition}

\theoremstyle{remark}
\newtheorem{rem}[thm]{Remark}

\newtheorem*{rem*}{Remark}

\newcommand{\be}{\begin{equation}}    
\newcommand{\ee}{\end{equation}}    
\newcommand{\beu}{\begin{equation*}}    
\newcommand{\eeu}{\end{equation*}}    
\newcommand{\bea}{\begin{eqnarray}}    
\newcommand{\eea}{\end{eqnarray}}    
\newcommand{\beaa}{\begin{eqnarray*}}    
\newcommand{\eeaa}{\end{eqnarray*}}    
\newcommand{\bmx}{\begin{pmatrix}}    
\newcommand{\emx}{\end{pmatrix}}    
\newcommand{\btz}{\begin{tikzpicture}}
\newcommand{\etz}{\end{tikzpicture}}
    
\newcommand{\Aut}[1]{\mathrm{Aut}(#1)}

\newcommand{\nn}{\nonumber}    
\newcommand{\sign}{{\rm sign}}

\newcommand{\Z}{{\mathbb Z}}
\newcommand{\N}{{\mathbb N}}
\newcommand{\C}{{\mathbb C}}

\newcommand{\Q}{{\mathbb Q}}

\newcommand{\id}{{\mathrm{id}}}

\newcommand{\qaff}{\dot{\mathrm{U}}_q}
\newcommand{\qdaff}{\ddot{\mathrm{U}}_q}
\newcommand{\qdloop}{\ddot{\mathrm{L}}_q}

\newcommand{\btp}{\begin{tikzpicture}[baseline=0pt,scale=0.9,line width=0.25pt]}    
\newcommand{\etp}{\end{tikzpicture}}

\newcommand{\range}[2]{\llbracket #1,#2 \rrbracket}

\newcommand{\atp}[1]{}

\newcommand{\ie}{i.e. }
\newcommand{\eg}{e.g. }

\newcommand{\Sp}{\mathrm{Sp}}

\newcommand{\ev}{\mathrm{ev}}

\newcommand{\uqslthh}{\dot{\mathrm{U}}_q  ( \dot{\mathfrak{a}}_1 ) }

\newcommand{\uqhhzeroslt}{\ddot{\mathrm{U}}_q^0  ({\mathfrak{a}}_1 ) }
\newcommand{\uqmslthh}{\dot{\mathrm{U}}_q^-  ( \dot{\mathfrak{a}}_1 ) }
\newcommand{\uqpslthh}{\dot{\mathrm{U}}_q^+  ( \dot{\mathfrak{a}}_1 ) }
\newcommand{\uqpmslthh}{\dot{\mathrm{U}}_q^\pm  ( \dot{\mathfrak{a}}_1 ) }
\newcommand{\uqzeroslthh}{\dot{\mathrm{U}}_q^0  ( \dot{\mathfrak{a}}_1 ) }

\newcommand{\xp}[1]{{\bf x}_{#1}^+}
\newcommand{\xm}[1]{{\bf x}_{#1}^-}

\newcommand{\km}[1]{{\bf k}_{#1}^-}

\newcommand{\x}[2]{{\bf x}_{#2}^{#1}}
\newcommand{\kk}[2]{{\bf k}_{#2}^{#1}}
\newcommand{\Xp}[1]{{\bf X}_{#1}^+}

\newcommand{\KK}[2]{{\bf K}_{#2}^{#1}}

\newcommand{\X}[2]{{\bf X}_{#2}^{#1}}
\newcommand{\Xsf}[2]{{\mathsf X}_{#2}^{#1}}
\newcommand{\Xbsf}[2]{{\textbf{\textsf{X}}}_{#2}^{#1}}
\newcommand{\Lbf}[2]{{\bf{L}}_{#2}^{#1}}
\newcommand{\Rbf}[2]{{\bf{R}}_{#2}^{#1}}
\newcommand{\Hbf}[2]{{\bf{H}}_{#2}^{#1}}

\newcommand{\tsf}[2]{{\mathsf t}^{#1}_{#2}}
\newcommand{\psf}[2]{{\mathsf p}^{#1}_{#2}}

\newcommand{\Ksf}[2]{{\mathsf K}^{#1}_{#2}}
\newcommand{\csf}[2]{{\mathsf c}^{#1}_{#2}}
\newcommand{\tbsf}[2]{{\textbf{\textsf{t}}}^{#1}_{#2}}
\newcommand{\pbsf}[1]{{\textbf{\textsf{p}}}^{#1}}

\newcommand{\Kbsf}[2]{{\textbf{\textsf{K}}}^{#1}_{#2}}
\newcommand{\cbsf}[1]{{\textbf{\textsf{c}}}^{#1}}
\newcommand{\eb}[1]{{\textbf{{e}}}^{#1}}
\newcommand{\psib}[1]{{\boldsymbol \psi}^{#1}}
\newcommand{\Csf}{{{\textsf{C}}}}
\newcommand{\Dsf}{{{\textsf{D}}}}

\newcommand{\bwp}[1]{{\boldsymbol \wp}^{#1}}

\newcommand{\bpsi}[2]{{\boldsymbol \psi}^{#1}_{#2}}
\newcommand{\bkap}{{\boldsymbol \kappa}}

\newcommand{\K}{\mathbb K}
\newcommand{\F}{\mathbb F}

\newcommand{\noi}{\noindent}
\DeclareMathOperator*{\res}{res}

\newcommand{\supp}{\mathrm{supp}}
\newcommand{\E}[3]{\mathcal E_{#1,#2,#3}}

\newcommand{\rran}[1]{{\llbracket #1 \rrbracket}}

\newcommand{\zebf}{{\boldsymbol\zeta}}
\newcommand{\zbf}{{\boldsymbol z}}

\author{E. Mounzer}
\author{R. Zegers}
\address{\vspace{-.30cm} \hspace{-.22cm}
Universit\'e Paris-Saclay,
}
\address{\vspace{-.3cm}\hspace{-.1cm} CNRS, IJCLab,}
\address{\vspace{-.08cm} \hspace{-.2cm}
91405, Orsay, France \vspace{.3cm}}
\email{robin.zegers@th.u-psud.fr}

\begin{document} 
\title[Weight-finite modules over $\qaff(\mathfrak a_1)$ and $\qdaff(\mathfrak a_1)$]{Weight-finite modules over the quantum affine and double quantum affine algebras of type $\mathfrak a_1$}

\begin{abstract}
We define the categories of weight-finite modules over the type $\mathfrak a_1$ quantum affine algebra $\qaff(\mathfrak a_1)$ and over the type $\mathfrak a_1$ double quantum affine algebra $\qdaff(\mathfrak a_1)$ that we introduced in \cite{MZ}.  In both cases, we classify the simple objects in those categories. In the quantum affine case, we prove that they coincide with the simple finite-dimensional $\qaff(\mathfrak a_1)$-modules which were classified by Chari and Pressley in terms of their highest (rational and $\ell$-dominant) $\ell$-weights or, equivalently, by their Drinfel'd polynomials. In the double quantum affine case, we show that simple weight-finite modules are classified by their ($t$-dominant) highest $t$-weight spaces, a family of simple modules over the subalgebra $\qdaff^0(\mathfrak a_1)$ of $\qdaff(\mathfrak a_1)$ which is conjecturally isomorphic to a split extension of the elliptic Hall algebra. The proof of the classification, in the double quantum affine case, relies on the construction of a double quantum affine analogue of the evaluation modules that appear in the quantum affine setting.
\end{abstract}

\maketitle

\section{Introduction}
The representation theory of quantum affine algebras is a vast and extremely rich theory which is still the subject of an intense research activity after more than three decades. The recent discovery of its relevance to the monoidal categorification of cluster algebras provides one of the latest and most striking illustrations of it -- see \cite{Hernandez1} for a review on that subject. Probably standing as one of the most significant breakthroughs in the early days of this research area, the classification of the simple finite-dimensional modules over the quantum affine algebra of type $\mathfrak a_1$, $\mathrm U_q(\dot{\mathfrak a}_1)$, is due to Chari and Pressley \cite{CP}. It relies, on one hand, on a careful analysis of the $\ell$-weight structure of those modules made possible by the existence of Drinfel'd's presentation $\qaff(\mathfrak a_1)$ of $\mathrm U_q(\dot{\mathfrak a}_1)$ -- see \cite{Damiani} for the proof that $\qaff(\mathfrak a_1)\cong \mathrm U_q(\dot{\mathfrak a}_1)$ -- and, on the other hand, on the existence of evaluation modules, proven earlier by Jimbo, \cite{Jimbo}. This seminal work paved the way for a more systematic study of the representation theory of quantum affine algebras of all Cartan types, leading to the development of powerful tools such as $q$-characters, $(q,t)$-characters and, consequently, to a much better understanding of the categories $\mathrm{\bf FinMod}$ of their finite-dimensional modules that recently culminated with the realization that the Grothendieck rings of certain subcategories of the categories $\mathrm{\bf FinMod}$ actually have the structure of a cluster algebra, \cite{Hernandez2}.

By contrast, it is fair to say that the representation theory of quantum toroidal algebras, which were initially introduced in type $\mathfrak a_n$ by Ginzburg, Kapranov and Vasserot \cite{GKV} and later generalized to higher rank types, is significantly less well understood and remains, to this date, much more mysterious -- although see \cite{Hernandez} for a review and references therein. In our previous work, \cite{MZ}, we constructed a new (topological) Hopf algebra $\qdaff(\mathfrak a_1)$, called double quantum affinization of type $\mathfrak a_1$, and proved that its completion (in an appropriate topology) is bicontinuously isomorphic to (a corresponding completion) of the quantum toroidal algebra $\qaff(\dot{\mathfrak a}_1)$. Whereas $\qaff(\dot{\mathfrak a}_1)$ is naturally graded over $\Z\times \dot Q$, where $\dot Q$ stands for the root lattice of the untwisted affine root system $\dot{\mathfrak a}_1$ of type $A_1^{(1)}$, $\qaff(\dot{\mathfrak a}_1)$ is naturally graded over $\Z^2\times Q$, where $Q$ stands for the root lattice of the finite root system $\mathfrak a_1$ of type $A_1$. Thus $\qdaff(\mathfrak a_1)$ turns out to be to $\qaff(\dot{\mathfrak a}_1)$ what $\qaff(\mathfrak a_1)$ is to $\mathrm U_q(\dot{\mathfrak a}_1)$, \ie its Drinfel'd presentation. The latter, in the quantum affine case, has a natural triangular decomposition which allows one to define an adapted class of highest weight modules, namely highest $\ell$-weight modules, in which finite-dimensional modules are singled out by the particular form of their highest $\ell$-weights. Therefore, it is only natural to ask the question of whether $\qdaff(\mathfrak a_1)$ plays a similar role for the representation theory of $\qaff(\dot{\mathfrak a}_1)$, leading, in particular, to a new notion of highest weight modules. We answer positively that question and introduce the corresponding notion of highest $t$-weight modules. Schematically, whereas the transition from the classical Lie theoretic weights to $\ell$-weights can be regarded as trading numbers for (rational) functions, the transition from $\ell$-weights to $t$-weights can be regarded as trading (rational) functions for entire modules over the non-commutative $\qdaff^0(\mathfrak a_1)$-subalgebra of $\qdaff(\mathfrak a_1)$. That substitution can be interpreted from the perspective of a conjecture in \cite{MZ}, stating that $\qdaff^0(\mathfrak a_1)$ is isomorphic to a split extension of the elliptic Hall algebra $\E{q^{-4}}{q^2}{q^2}$ which was initially defined by Miki, in \cite{Miki}, as a $(q,\gamma)$-analogue of the $W_{1+\infty}$ algebra and reappeared later on in different guises;  the quantum continuous  $\mathfrak{gl}_\infty$ algebra in \cite{feigin2011}, the Hall algebra of the category of coherent sheaves on some elliptic curve in \cite{Schiffmann}, or the quantum toroidal algebra associated with $\mathfrak{gl}_1$ in \cite{feigin2012} and in subsequent works by Feigin et al. Our conjecture is actually supported by the existence of an algebra homomorphism between $\E{q^{-4}}{q^2}{q^2}$ and $\qdaff^0(\mathfrak a_1)$ which we promote, in the present paper, to a (continuous) homomorphism of (topological) Hopf algebras. Intuitively, the weights adapted to our new triangular decomposition can therefore be regarded as representations of a quantized algebra of functions on a non-commutative 2-torus. 

On the other hand, unless the value of some scalar depending on the deformation parameter is taken to be a root of unity, the question of the existence of finite-dimensional modules over quantum toroidal algebras of type $\mathfrak a_{n\geq2}$ was already answered negatively by Varagnolo and Vasserot in \cite{Varagnolo}. However, it is possible to push further the analogy with the quantum affine situation by defining another type of finiteness condition, namely weight-finiteness. It turns out that, in type $1$, \ie when the central charges act trivially, $\qdaff^0(\mathfrak a_1)$ admits an infinite dimensional abelian subalgebra that, itself, admits as a subalgebra the Cartan subalgebra $\mathrm{U}_q^0(\mathfrak a_1)$ of the Drinfel'd-Jimbo quantum algebra $\mathrm{U}_q(\mathfrak a_1)$ of type $\mathfrak a_1$. Hence, we can assign classical Lie theoretic weights to the $t$-weight spaces of our modules and declare that a $\qdaff'(\mathfrak a_1)$-module is weight-finite whenever it has only finitely many classical weights. The same notion is readily defined for modules over $\qaff(\mathfrak a_1)$ and we then focus on $\mathrm{\bf WFinMod\,\dot{}}$ (resp. $\mathrm{\bf WFinMod}$), \ie the full subcategory of the category $\mathrm{\bf Mod\,\dot{}}$ (resp. $\mathrm{\bf Mod}$) of $\qdaff(\mathfrak a_1)$-modules (resp. $\qaff(\mathfrak a_1)$-modules) whose modules are weight-finite. Of course, the widely studied category $\mathrm{\bf FinMod}$ of finite-dimensional $\qaff(\mathfrak a_1)$-modules is a full subcategory of $\mathrm{\bf WFinMod}$. The main results of the present paper consist in showing that, on one hand, the simple objects in 
$\mathrm{\bf WFinMod}$ are all finite-dimensional and therefore coincide with the simple finite-dimensional $\qaff(\mathfrak a_1)$-modules classified by Chari and Pressley, and, on the other hand, in classifying the simple objects in $\mathrm{\bf WFinMod\,\dot{}}$ in terms of their highest $t$-weight spaces. These results clearly establish $\mathrm{\bf WFinMod\,\dot{}}$ as the natural quantum toroidal analogue of $\mathrm{\bf FinMod}$ and suggest studying further its structure and, in particular, the structure of its Grothendieck ring. Another natural development at this point would be to generalize to the quantum toroidal setting the interesting classes of $\qaff(\mathfrak a_1)$-modules outside of $\mathrm{\bf FinMod}$, for example by constructing a quantum toroidal analogue of category $\mathcal O$. We leave these questions for future work.

The present paper is organized as follows. In section \ref{Sec:qaff}, we briefly review classic results about the quantum affine algebra $\qaff(\mathfrak a_1)$ and its finite-dimensional modules. Then, we prove that simple objects in $\mathrm{\bf WFinMod}$ are actually finite-dimensional. In section \ref{Sec:qdaff}, we review the main relevant results of \cite{MZ} and establish a few new results, as relevant for the subsequent sections. We define highest $t$-weight modules in section \ref{Sec:tweight} and, by thoroughly  analyzing their structure, we establish one implication in our classification theorem, namely theorem \ref{thm:classification}. The opposite implication is established in section \ref{Sec:ev} by explicitly constructing a quantum toroidal analogue of the quantum affine evaluation modules. That construction is obtained after proving the existence of an evaluation homomorphism between $\qdaff(\mathfrak a_1)$ and an evaluation algebra built as a double semi-direct product of $\qaff(\mathfrak a_1)$ with the completions of some Heisenberg algebras. The evaluation modules are then obtained by pulling back induced modules over the evaluation algebra along the evaluation homomorphism.

\subsection*{Notations and conventions}
We let $\N = \{0, 1, \dots\}$ be the set of natural integers including $0$. We denote by $\N^\times$ the set $\N-\{0\}$. For every $m\leq n \in\N$, we denote by $\range{m}{n} = \{m, m+1, \dots, n\}$. We also let $\rran{n}=\range{1}{n}$ for every $n\in\N$. For every $m, n \in \N^\times$, we let
$$C_m(n):=\left\{\lambda=(\lambda_1, \dots, \lambda_m) \in \left (\N^\times\right )^{m} : \lambda_1+ \dots +\lambda_m =n \right \}\,,$$
denote the set of $m$-compositions of $n$, \ie of compositions of $n$ having $m$ summands.

We let $\sign:\Z\to\{-1, 0, 1\}$ be defined by setting, for any $n\in\Z$,
$$\sign(n) = \begin{cases}
-1 & \mbox{if $n<0$;}\\
0 & \mbox{if $n=0$;}\\
1 & \mbox{if $n>0$.}
\end{cases}$$

We assume throughout that $\K$ is an algebraically closed field of characteristic $0$ and we let $\F:=\K(q)$ denote the field of rational functions over $\K$ in the formal variable $q$. As usual, we let $\K^\times=\K-\{0\}$ and $\F^\times=\F-\{0\}$. Whenever we wish to evaluate $q$ to some element of $\K^\times$, we shall always do so under the restriction that $1 \notin q^{\Z^\times}$. For every $m, n \in \N$, we define the following elements of $\F$
\be [n]_q := \frac{q^n-q^{-n}}{q-q^{-1}}\,, \qquad [n]_q^! := \begin{cases} [n]_q [n-1]_q \cdots [1]_q & \mbox{if $n\in\N^\times$;}\\
1&\mbox{if $n=0$;}
\end{cases}  \qquad {n \choose m}_q := \frac{[n]_q^!}{[m]_q^! [n-m]_q^!}\, .\ee

We shall say that a polynomial $P(z)\in \F[z]$ is \emph{monic} if $P(0)=1$. For every rational function $P(z)/Q(z)$, where $P(z)$ and $Q(z)$ are relatively prime polynomials, we denote by
$$\left (\frac{P(z)}{Q(z)}\right )_{|z|\ll 1} \qquad (\mbox{resp.} \left (\frac{P(z)}{Q(z)}\right )_{|z|^{-1}\ll 1} )$$
the Laurent series of $P(z)/Q(z)$ at $0$ (resp. at $\infty$).

We shall let
$${}_a\left[A, B \right ]_b = a AB - b BA\,,$$
for any symbols $a$, $b$, $A$ and $B$ provided the r.h.s of the above equations makes sense. 

The Dynkin diagrams and correponding Cartan matrices of the root systems $\mathfrak a_1$ and $\dot{\mathfrak a}_1$ are reminded in the following table.
\begin{center}
\begin{tabular}{|c|c|c|c|}
\hline
Type & Dynkin diagram & Simple roots & Cartan matrix\\
\hline
		$\mathfrak a_1$&
		\btz
		\tikzstyle{vert}=[circle,thick,fill=black]
		\draw (0,.5) node {$1$};		
		\draw (0,0) node[vert]{};
		\etz
		& $\Phi=\{\alpha_1\}$ & $(2)$\\
		&&&\\
		\hline
		$\dot{\mathfrak a}_1$&
		\btz
\tikzset{doublearrow/.style={draw, color=black, draw=black, double distance=2pt, ->}}
\tikzset{doubleline/.style={draw, color=black, draw=black, double distance=2pt}}
\tikzstyle{vert}=[circle,thick,fill=black]
\draw[doubleline] (0:2) node[vert] {}-- ++(0:1.1) node[vert] {}; 
\draw[doublearrow] (0:2.9)-- ++(180:.77); 
\draw[doublearrow] (0:2.3)-- ++(0:.65); 
\draw (2,.5) node {$0$};
\draw (3.1,.5) node {$1$};
\etz
		& $\dot\Phi=\{\alpha_0, \alpha_1\}$ & $\begin{pmatrix}
		2&-2\\-2&2
		\end{pmatrix}$\\
		&&&\\
		\hline
		\end{tabular}
\end{center}		

\section{Weight-finite modules over the quantum affine algebra $\qaff(\mathfrak a_1)$}
\label{Sec:qaff}
\subsection{The quantum  affine algebra $\qaff(\mathfrak a_1)$}
\begin{defn}
\label{def:defqaffdota1}
The \emph{quantum affine algebra} $\qaff(\mathfrak a_1)$ is the associative $\K(q)$-algebra generated by 
$$\left \{D, D^{-1}, C^{1/2}, C^{-1/2}, k_{1,n}^+, k_{1, -n}^-,  x_{1,m}^+, x_{1,m}^- : m \in \Z, n \in \N\right\}$$ 
subject to the following relations
\be\label{eq:relccentral} \mbox{$C^{\pm1/2}$ is central} \qquad C^{\pm 1/2} C^{\mp 1/2} = 1 \qquad D^{\pm 1} D^{\mp 1} =1\ee
\be D\kk\pm 1(z)D^{-1} = \kk\pm 1(zq^{-1}) \qquad D\x\pm 1(z)D^{-1} = \x\pm 1(zq^{-1}) \ee
\be\label{eq:relkpmkpm} \kk \pm 1(z_1) \kk \pm 1(z_2) =  \kk \pm 1(z_2) \kk \pm 1(z_1)\ee
\be\label{eq:relkpkm} \kk - 1(z_1) \kk + 1(z_2) = G^-(C^{-1}z_1/z_2) G^+(Cz_1/z_2)  \kk + 1(z_2)  \kk - 1(z_1) = 1 \mod z_1/z_2\ee
\be\label{eq:relkpxpm} G^\mp(C^{\mp 1/2} z_2/z_1)  \kk +1(z_1) \x \pm 1(z_2) = \x \pm 1(z_2)  \kk +1(z_1)\ee
\be\label{eq:relkmxpm} \kk -1(z_1) \x \pm 1(z_2) = G^\mp(C^{\mp 1/2} z_1/z_2) \x \pm 1(z_2)  \kk -1(z_1)\ee
\be\label{eq:relxpmxpm} (z_1-q^{\pm 2} z_2) \x \pm 1(z_1) \x \pm 1(z_2) = (z_1q^{\pm 2}-z_2) \x \pm 1(z_2) \x \pm 1(z_1)\ee
\be\label{eq:relxpxm} [\x + 1(z_1), \x - 1(z_2)] = \frac{1}{q-q^{-1}} \left [ \delta \left ( \frac{z_1}{Cz_2} \right ) \kk + 1(z_1C^{-1/2}) - \delta \left( \frac{z_1C}{z_2}\right ) \kk - 1(z_2C^{-1/2}) \right ]  \ee
where we define the following $\uqslthh$-valued formal distributions
\be \x\pm{1}(z) := \sum_{m \in \Z} x^\pm_{1,m} z^{-m} \in \uqslthh[[z, z^{-1}]]\,;\ee
\be \kk \pm 1(z) := \sum_{n \in \N} k_{1, \pm n}^\pm z^{\mp n} \in \uqslthh[[z^{\mp 1}]]\, ,\ee
for every $i,j\in\dot I$, we define the following $\F$-valued formal power series
\be G^\pm(z):= q^{\pm c_{ij}} + (q-q^{-1}) [\pm c_{ij} ]_q \sum_{m \in \N^\times} q^{\pm m c_{ij}} z^m \in \F[[z]] \ee
and
\be \delta(z) := \sum_{m \in \Z} z^{m} \in \F[[z, z^{-1}]]\ee
is an $\F$-valued formal distribution.  We denote by $\qaff'({\mathfrak a}_1)$ the subalgebra of $\qaff({\mathfrak a}_1)$ generated by
$$\left \{C^{1/2}, C^{-1/2}, k_{1,n}^+, k_{1, -n}^-,  x_{1,m}^+, x_{1,m}^- : m \in \Z, n \in \N\right\}\,.$$
We denote by $\qaff^0({\mathfrak a}_1)$ the subalgebra of $\qaff'({\mathfrak a}_1)$ generated by
$$\left\{C^{1/2}, C^{-1/2}, k_{1,n}^+, k_{1, -n}^-:  n \in \N\right \}\,.$$
We let $\qaff^\geq(\mathfrak a_1)$ (resp. $\qaff^\leq(\mathfrak a_1)$)  denote the subalgebra of $\qaff'(\mathfrak a_1)$ generated by
$$\left \{C^{1/2}, C^{-1/2}, k_{1,n}^+, k_{1, -n}^-,  x_{1,m}^+ : m \in \Z, n \in \N\right\}$$
(resp.
$$\left \{C^{1/2}, C^{-1/2}, k_{1,n}^+, k_{1, -n}^-,  x_{1,m}^- : m \in \Z, n \in \N\right\}$$
). We let $\qaff(\mathfrak a_1)\,\breve{}$ denote the $\F$-algebra generated by the same generators as $\qaff'({\mathfrak a}_1)$, subject to the relations (\ref{eq:relkpmkpm} - \ref{eq:relxpmxpm}) -- \ie we omit relation (\ref{eq:relxpxm}). We define the type $\mathfrak a_1$ \emph{quantum loop algebra} $\mathrm U_q(\mathrm L \mathfrak a_1)$ as the quotient of $\qaff'({\mathfrak a}_1)$ by its two-sided ideal $(C^{1/2}-1)$ generated by $\left\{C^{1/2}-1, C^{-1/2}-1\right \}$. Similarly, we let $\mathrm U_q^{\geq}(\mathrm L \mathfrak a_1) = \qaff^\geq(\mathfrak a_1)/(C^{1/2}-1)$ and $\mathrm U_q^{\leq}(\mathrm L \mathfrak a_1) = \qaff^\leq(\mathfrak a_1)/(C^{1/2}-1)$. We eventually set $\breve{\mathrm U}_q(\mathrm L\mathfrak a_1)= \qaff(\mathfrak a_1)\,\breve{}/(C^{1/2}-1)$.
\end{defn}
Obviously,
\begin{prop}
\label{prop:Ubreve}
There exists a surjective $\F$-algebra homomorphism $\breve{\mathrm U}_q(\mathrm L\mathfrak a_1) \to \mathrm U_q(\mathrm L\mathfrak a_1)$.
\end{prop}

\subsection{Finite dimensional $\qaff'({\mathfrak a}_1)$-modules}
Let $\mathrm{\bf Mod}$ be the category of $\qaff'(\mathfrak a_1)$-modules. We denote by $\mathrm{\bf FinMod}$ the full subcategory of $\mathrm{\bf Mod}$ whose objects are finite-dimensional. Following \cite{CP}, we make the following 
\begin{defn}
We shall say that a $\qaff'(\mathfrak a_1)$-module $M$ is:
\begin{itemize}
\item a \emph{weight module} if $k_{1,0}^+$ acts semisimply on $M$;
\item of \emph{type $1$} if it is a weight module and $C^{1/2}$ acts on $M$ as $\id$;
\item \emph{highest $\ell$-weight} if it is of type $1$ and there exists $v\in M-\{0\}$ such that
$$\x+1(z).v=0\,,\qquad \kk\pm1(z).v= \kappa^\pm(z)v$$
for some $\kappa^\pm(z)\in \F[[z^{\mp 1}]]$ and $M=\qaff'({\mathfrak a}_1) . v$. We shall refer to any such $v$ as a \emph{highest $\ell$-weight vector} and to $\bkap=(\kappa^+(z), \kappa^-(z))$ as the corresponding highest $\ell$-weight.
\end{itemize}
\end{defn}
Clearly, type $1$ $\qaff'({\mathfrak a}_1)$-modules coincide with $\mathrm U_q(\mathrm L \mathfrak a_1)$-modules.
\begin{defn}
For every $\bkap \in \F[[z^{-1}]] \times \F[[z]]$, we construct a one-dimensional $\mathrm U_q^\geq(\mathrm L \mathfrak a_1)$-module $\F_\bkap\cong \F$ by setting
$$\x+1(z).1 = 0 \,,\qquad \mbox{and}\qquad \kk\pm1(z).1 = \kappa^\pm(z)\,.$$
We then define the \emph{universal} highest $\ell$-weight $\qaff'({\mathfrak a}_1)$-module with highest $\ell$-weight $\bkap$ by setting
$$M(\bkap) := \mathrm U_q(\mathrm L \mathfrak a_1) \underset{\mathrm U_q^\geq(\mathrm L \mathfrak a_1)}{\otimes} \F_\bkap$$
as $\mathrm U_q(\mathrm L \mathfrak a_1)$-modules. Let $N(\bkap)$ be the maximal $\mathrm U_q(\mathrm L \mathfrak a_1)$-submodule of $M(\bkap)$ such that $N(\bkap)\cap \F_\bkap = \{0\}$ and set
$$L(\bkap):= M(\bkap)/N(\bkap)\,.$$
By construction, $L(\bkap)$ is a simple highest $\ell$-weight $\mathrm U_q(\mathrm L \mathfrak a_1)$-module with highest $\ell$-weight $\bkap$. It is unique up to isomorphisms.
\end{defn}
The simple objects in $\mathrm{\bf FinMod}$ were classified by Chari and Pressley in \cite{CP}. The main result is the following
\begin{thm}[Chari-Pressley]\label{thm:CP}
The following hold:
\begin{enumerate}[i.]
\item any simple finite-dimensional $\qaff'(\mathfrak a_1)$-module $M$ can be obtained by twisting a simple finite-dimensional $\qaff'(\mathfrak a_1)$-module of type $1$ with an algebra automorphism of $\mathrm{Aut}(\qaff'(\mathfrak a_1))$;
\item\label{it:FinHighest} every simple finite dimensional $\qaff'(\mathfrak a_1)$-module of type $1$ is highest $\ell$-weight;
\item\label{it:univsimple} the simple highest $\ell$-weight module $L(\bkap)$ is finite-dimensional if and only if
$$\kappa^\pm(z) = q^{\deg(P)} \left (\frac{P(q^{-2}/z)}{P(1/z)}\right )_{|z|^{\mp1}\ll 1}\,,$$
for some monic polynomial $P(1/z)\in \F[z^{-1}]$ called \emph{Drinfel'd polynomial} of $L(\bkap)$.
\end{enumerate}
\end{thm}
\begin{proof}
The proof can be found in \cite{CP}.
\end{proof}
Up to isomorphisms, the simple objects in $\mathrm{\bf FinMod}$ are uniquely parametrized by their Drinfel'd polynomials and we shall therefore denote by $L(P)$ the (isomorphism class of the) simple $\qaff'({\mathfrak a}_1)$-module with Drinfel'd polynomial $P$. Note that the roles of $\qaff^\geq(\mathfrak a_1)$ and $\qaff^\leq(\mathfrak a_1)$ in the above constructions are clearly symmetrical and we could have equivalently considered lowest $\ell$-weight modules. In particular, point \ref{it:univsimple} of the above theorem immediately translates into
\begin{prop}
\label{prop:lowestellweight}
The simple lowest $\ell$-weight module with lowest $\ell$-weight $\bkap=(\kappa^+(z),\kappa^-(z))\in \F[[z^{-1}]]\times \F[[z]]$ is finite-dimensional if and only if
$$\kappa^\pm(z) = q^{-\deg(P)}\left(\frac{P(1/z)}{P(q^{-2}/z)}\right)_{|z|^{\mp 1}\ll 1}\,,$$
for some monic polynomial $P(1/z)\in \F[z^{-1}]$. In the latter case, we denote it by $\bar L(P)$.
\end{prop}

\subsection{Weight-finite simple $\mathrm U_q(\mathrm L\mathfrak a_1)$-modules}
\label{sec:wfqaff}
We now wish to consider a slightly broader family of modules over $\qaff'({\mathfrak a}_1)$. In particular, we want to allow these modules to be infinite-dimensional, while retaining some of the nice features of finite dimensional $\qaff'({\mathfrak a}_1)$-modules such as the fact that they decompose into $\ell$-weight spaces. This is achieved by introducing the following notion.
\begin{defn}\label{def:ellweight}
We shall say that a (not necessarily finite-dimensional) $\qaff'({\mathfrak a}_1)$-module $M$ is \emph{$\ell$-weight} if there exists a countable set $\left \{ M_\alpha: \alpha\in A\right \}$ of indecomposable locally finite-dimensional $\qaff^0({\mathfrak a}_1)$-modules, called the \emph{$\ell$-weight spaces} of $M$, such that, as $\qaff^0({\mathfrak a}_1)$-modules, 
$$M\cong \bigoplus_{\alpha\in A} M_\alpha\,.$$
We shall say that $M$ is of type $1$ if $C^{1/2}$ acts on $M$ by $\id$.
\end{defn}
\begin{defprop}\label{defprop:ellweight}
Let $M$ be an $\ell$-weight $\qaff'({\mathfrak a}_1)$-module. Then:
\begin{enumerate}[i.]
\item\label{qaffellwi} $C^2$ acts as $\id$ over $M$;
\item\label{qaffellwii} for every $\ell$-weight space $M_\alpha$, $\alpha\in A$, of $M$, there exists $\kappa_{\alpha,0}\in\F^\times$ and $(\kappa^\pm_{\alpha, \pm m})_{m\in\N^\times} \in\F^{\N^\times}$ such that
$$M_\alpha \subseteq \left\{ v\in M : \exists n\in\N^\times\,,\forall m\in\N \quad \left (k^\pm_{1, \pm m} - \kappa^\pm_{\alpha, \pm m} \id \right ) ^n.v=0\right \}\,,$$
where we have set $\kappa^\pm_{\alpha,0} = \kappa_{\alpha, 0}^{\pm 1}$. 
\end{enumerate}
We let $\Sp(M) = \left\{\kappa_{\alpha, 0} : \alpha\in A\right \}$ and refer to the formal power series
$$\kappa^\pm_\alpha(z) = \sum_{m\in\N} \kappa^\pm_{\alpha, \pm m} z^{\mp m}$$
as the \emph{$\ell$-weight} of the $\ell$-weight space $M_\alpha$.
\end{defprop}
\begin{proof}
Let $M_\alpha$ be an $\ell$-weight space of $M$ and let $v\in M_\alpha-\{0\}$. By definition, there exists a finite dimensional $\qaff^0(\mathfrak a_1)$-submodule $\tilde M_\alpha$ of $M_\alpha$ such that $v\in \tilde M_\alpha$. Over $\tilde M_\alpha$, $C$ must admit an eigenvector and, since $C$ is central, it follows that $C$ acts over $\tilde M_\alpha$ by a scalar mutliple of $\id$. Assume for a contradiction that $C-C^{-1}$ does not act by multiplication by zero. Then, it is possible to pull back $\tilde M_\alpha$ into a finite-dimensional module over the Weyl algebra $\mathcal A_1(\K) = \K\langle x, y\rangle /(xy-yx-1)$ by the obvious algebra homomorphism $\mathcal A_1(\K) \hookrightarrow \qaff^0(\mathfrak a_1)$. But the Weyl algebra is known to admit no finite-dimensional modules. A contradiction. It follows that $C^2$ acts as $\id$ over $\tilde M_\alpha$. But this could be repeated for any non-zero vector in any $\ell$-weight space of $M$. \ref{qaffellwi} follows. As for \ref{qaffellwii}, observe that, as a consequence of \ref{qaffellwi} and of the defining relations (\ref{eq:relkpmkpm}) and (\ref{eq:relkpkm}), $\left\{k^+_{1,m}, k^-_{1,-m} : m\in\N \right \}$ acts by a family of commuting linear operators over $M$. Thus \ref{qaffellwii} follows from the decomposition of locally finite-dimensional vector spaces into the generalized eigenspaces of a commuting family of linear operators; the indecomposability of $M_\alpha$ further imposing that it coincides with a single block in a single generalized eigenspace.
\end{proof}

\begin{rem}\label{rem:ellweight}
It is worth emphasizing that definition \ref{def:ellweight} and definition-proposition \ref{defprop:ellweight} straightforwardly generalize to (topological) modules over any (topological) algebra $\mathcal A$ containing $\qaff^0(\mathfrak a_1)$ as a (closed) subalgebra.
\end{rem}

\begin{defn}
We shall say that an $\ell$-weight $\qaff'({\mathfrak a}_1)$-module $M$ is \emph{weight-finite} if $Sp(M)$ is a finite set. We let  $\mathrm{\bf WFinMod}$ denote the full subcategory of the category  $\mathrm{\bf Mod}$ of $\qaff'({\mathfrak a}_1)$-modules whose objects are weight-finite.
\end{defn}
Clearly, finite dimensional $\qaff'({\mathfrak a}_1)$-modules are objects in $\mathrm{\bf WFinMod}$, but not every object in $\mathrm{\bf WFinMod}$ is in $\mathrm{\bf FinMod}$. However we have
\begin{thm}
\label{thm:wffdqaffsl2} The following hold:
\begin{enumerate}[i.]
\item\label{it:type1} every simple $\ell$-weight $\qaff'(\mathfrak a_1)$-module can be obtained by twisting a simple $\ell$-weight $\qaff'(\mathfrak a_1)$-module of type $1$ with an algebra automorphism of $\mathrm{Aut}(\qaff(\mathfrak a_1))$;
\item\label{it:qaffWF} every weight-finite simple $\mathrm U_q(\mathrm L\mathfrak a_1)$-module is highest $\ell$-weight;
\item\label{it:findim} every weight-finite simple $\mathrm U_q(\mathrm L\mathfrak a_1)$-module is finite dimensional.
\end{enumerate}
\end{thm}
\begin{proof}
In view of definition-proposition \ref{defprop:ellweight}, $C^2$ acts as $\id$ over $M$. Since the latter is simple and since $C^{1/2}$ is central, it is clear that $C$ acts over $M$ either as $\id$ or as $-\id$. In the former case, there is nothing to do; whereas in the latter, upon twisting as in the finite-dimensional case -- see \cite{CP} --, we can ensure that $C^{1/2}$ acts as $\id$. This proves \ref{it:type1}. As for \ref{it:qaffWF}, the same proof as for part \ref{it:FinHighest} of theorem \ref{thm:CP} can be used. So, we eventually prove \ref{it:findim}. Let $M$ be a weight-finite simple $\mathrm U_q(\mathrm L\mathfrak a_1)$-module. By \ref{it:qaffWF} it is highest $\ell$-weight. Hence, there exists $v\in M-\{0\}$ such that $M\cong \mathrm U_q(\mathrm L\mathfrak a_1).v$, $\xp{1}(z).v=0\,$ and $\kk\pm{1}(z).v= \kappa^\pm(z) v$, for some $\kappa^\pm(z)\in \F[[z^{\mp 1}]]$ with
$\res_{z_1,z_2} z_1^{-1} z_2^{-1} \kappa^+(z_1)\kappa^-(z_2) = 1\,.$
The triangular decomposition of $\mathrm U_q(\mathrm L\mathfrak a_1)$ implies that $M = \mathrm U_q^-(\mathrm L\mathfrak a_1).v$ and, setting for every $n\in\N$
\be\label{eq:vzdef} v(z_1, \dots, z_n) = \xm{1}(z_1) \cdots \xm{1}(z_n). v\,,\ee
it is clear that
\be\label{eq:spanset}\left\{v_{m_1, \dots, m_n}=\res_{z_1, \dots, z_n} z_1^{-1-m_1} \cdots z_n^{-1-m_n} v(z_1, \dots, z_n) : n\in\N, m_1,\dots, m_n \in\Z \right\}\ee
is a spanning set of $M$. The defining relations (\ref{eq:relkpxpm}) and (\ref{eq:relkmxpm}) of $\mathrm U_q(\mathrm L\mathfrak a_1)$ easily imply that, for every $n\in\N$,
\be\label{eq:kvzns}\kk\pm{1}(z) . v(z_1,\dots, z_n)= \kappa^\pm(z) \prod_{p=1}^n G^\mp_{}\left (\left (z_p/z\right )^{\mp 1}\right  ) v(z_1,\dots, z_n) \ee
and, in particular,
$$k^\pm_0 .v(z_1, \dots, z_n)  = (\kappa^+_0)^{\pm 1} q^{-2n} v(z_1, \dots, z_n)\,.$$
Therefore, $M$ being weight-finite, there must exist an $N\in\N$ such that
\be\label{eq:weightfinite}\xm{1}(z). v(z_1, \dots, z_N) = 0\,.\ee
Making use of (\ref{eq:relxpxm}), one easily proves that, for every $n\in\rran{N}$
\bea\label{eq:xpvz} \xp{1}(z).v(z_0,\dots, z_n) &=& \frac{1}{q-q^{-1}} \sum_{p=0}^n \delta\left(\frac{z_p}{z}\right ) \left [\kappa^+(z) \prod_{r=p+1}^n G^-_{}(z_r/z) \right. \\
&&\qquad\qquad\qquad\qquad\qquad\qquad\qquad \left .- \kappa^-(z)\prod_{r=p+1}^n G^+_{} (z/z_r)  \right ] v(z_0, \dots, \widehat{z_p}, \dots, z_n)\,,\nn\eea
where a hat over a variable indicates that that variable should be omitted. Combining (\ref{eq:weightfinite}) and (\ref{eq:relxpxm}), we get
\bea - \xm{1}(z_0) \xp{1}(z).v(z_1, \dots, z_N) &=& [\xp{1}(z), \xm{1}(z_0)].v(z_1, \dots, z_N) \nn\\
&=& \frac{1}{q-q^{-1}} \delta\left(\frac{z_0}{z}\right ) \left [\kappa^+(z) \prod_{p=1}^N G^-_{} (z_p/z) - \kappa^-(z) \prod_{p=1}^N G^+(z/z_p)\right ] v(z_1,\dots, z_N)\,.\nn\eea
Making use of (\ref{eq:xpvz}) and (\ref{eq:vzdef}), the above equation eventually yields
$$\sum_{p=0}^N \delta\left (\frac{z_p}{z}\right ) \left [\kappa^+(z_p) \prod_{r=p+1}^N G^-_{}(z_r/z_p) - \kappa^-(z_p) \prod_{r=p+1}^N G^+_{}(z_p/z_r) \right ] v(z_0, \dots, \widehat{z_p},\dots, z_N)=0\,.$$
Acting on the l.h.s of the above equation with $\xp{1} (\zeta_N) \cdots \xp{1}(\zeta_1)$ and making repeated use of (\ref{eq:xpvz}), one easily shows that
\be\label{eq:constr}\sum_{\sigma\in S_{N+1}} \prod_{i=0}^N \delta \left (\frac{z_i}{\zeta_{\sigma(i)}}\right ) \left [\kappa^+(z_i) \prod_{\substack{r=i+1\\ \sigma(r) >\sigma(i)}}^N G^-_{}(z_r/z_i) -\kappa^-(z_i) \prod_{\substack{r=i+1\\ \sigma(r) >\sigma(i)}}^N G^+_{}(z_i/z_r)  \right ] v= 0\,.\ee
Since $v\neq0$, its prefactor in the above equation must vanish. Now, it is clear that multiplication of the latter by $\prod_{j=0}^{N-1} (z_0 - \zeta_j)$ annihilates all the summands with $\sigma$ such that $\sigma(0)\neq N$. Similarly, multiplication by $\prod_{i=0}^1 \prod_{j=0}^{N-2} (z_i - \zeta_j)$ annihilates all the summands with $\sigma$ such that $\sigma(0) \neq N$ and $\sigma(1)\neq N-1$. Repeating the argument finitely many times, we arrive at the fact that multiplication by $\prod_{i=0}^N  \prod_{j=0}^{N-i-1} (z_i -\zeta_j)$ annihilates all the summands with $\sigma\neq(N, N-1, \dots, 0)$, so that, eventually,
$$0= \prod_{i=0}^N \delta \left (\frac{z_i}{\zeta_{N-i}}\right ) \prod_{j=0}^{N-i-1} (z_i -\zeta_j)\left [\kappa^+(z_i)  -\kappa^-(z_i) \right ] = \prod_{i=0}^N \delta \left (\frac{z_i}{\zeta_{N-i}}\right ) \prod_{j=0}^{N-i-1} (z_i -z_{N-j})\left [\kappa^+(z_i)  -\kappa^-(z_i) \right ]\,.$$
Taking the zeroth order term in $\zeta_j$ for $j=0,\dots, N$, we get
\bea 0&=& \prod_{i=0}^N \prod_{j=i+1}^{N} (z_i -z_{j})\left [\kappa^+(z_i)  -\kappa^-(z_i) \right ] \nn\\
&=& \begin{vmatrix}
 \left [\kappa^+(z_0)  -\kappa^-(z_0) \right ]&\left [\kappa^+(z_1)  -\kappa^-(z_1) \right ] & \dots & \left [\kappa^+(z_N)  -\kappa^-(z_N) \right ]\\
z_0 \left [\kappa^+(z_0)  -\kappa^-(z_0) \right ]& z_1 \left [\kappa^+(z_1)  -\kappa^-(z_1) \right ] & \dots & z_N \left [\kappa^+(z_N)  -\kappa^-(z_N) \right ]\\
\vdots & \vdots & \dots &\vdots \\
z_0^{N-1} \left [\kappa^+(z_0)  -\kappa^-(z_0) \right ]& z_1^{N-1}\left [\kappa^+(z_1)  -\kappa^-(z_1) \right ] & \dots & z_N^{N-1}\left [\kappa^+(z_N)  -\kappa^-(z_N) \right ]
\end{vmatrix}\,.\nn\eea
Hence, the rows of the matrix on the r.h.s. of the above equation are linearly dependent and it follows that there exists a $P(z) \in \F[z]-\{0\}$ of degree at most $N-1$, such that
\be\label{eq:Pkappa} P(z) \left [\kappa^+(z)  -\kappa^-(z) \right ] = 0\,.\ee
As a consequence, there clearly exists $Q(z) \in \F[z]$ such that $\deg Q = \deg P$ and
$$\kappa^\pm(z) = \left (\frac{ Q(z)}{P(z)}\right )_{|z|^{\mp 1}\ll 1}\,.$$
Now considering (\ref{eq:xpvz}) with $n=0$ and multiplying it by $P(z_0)$ obviously yields
 \be\label{eq:whighest} \xp{1}(z). P(z_0) v(z_0) = 0\,.\ee
 Set for very $m\in\Z$, $w_m = \res_{z_0} z_0^{-1-m} P(z_0) v(z_0)$. 
 Then, (\ref{eq:whighest}), together with (\ref{eq:kvzns}) for $n=1$, implies that
 $$\bigoplus_{m\in \Z} \qaff(\mathfrak a_1) . w_m$$
 is a strict submodule of the simple $\mathrm U_q(\mathrm L\mathfrak a_1)$-module $M$ and it follows that $w_m=0$ for every $m\in\Z$. All the vectors in $\left\{v_m : m\in\Z\right \}$ -- see (\ref{eq:spanset}) -- can therefore be expressed as linear combinations of the vectors in, say, $\left \{v_1, \dots, v_{\deg(P)} \right \}$ and the linear span of $\left\{v_m : m\in\Z\right \}$ turns out to be finite dimensional. Repeating that argument finitely many times for the linear spans of $\left\{v_{m_1, \dots, m_p}: m_1, \dots, m_p\in\Z\right \}$ with $p=1,\dots, N$ eventually concludes the proof.
\end{proof}

\begin{cor}
\label{cor:simplewffinite}
Let $M$ be a weight-finite simple highest (resp. lowest) $\ell$-weight $\mathrm U_q(\mathrm L\mathfrak a_1)$-module. Then $M\cong L(P)$ (resp. $M\cong \bar L(P)$), for some monic polynomial $P$. 
\end{cor}
\begin{proof}
In the highest $\ell$-weight case, this follows directly by the previous theorem and the classification of the simple finite dimensional $\mathrm U_q(\mathrm L\mathfrak a_1)$-modules, theorem \ref{thm:CP}. In the lowest $\ell$-weight case, see proposition \ref{prop:lowestellweight}.
\end{proof}

\section{Double quantum affinization of type $\mathfrak a_1$}
\label{Sec:qdaff}
\subsection{Definition of $\qdaff(\mathfrak a_1)$}
\begin{defn}
\label{defn:qdaff}
The \emph{double quantum affinization} $\qdaff(\mathfrak{a}_1)$ 
of type $\mathfrak{a}_1$ is defined as the $\F$-algebra generated by $$\{\Dsf_1, \Dsf_1^{-1}, \Dsf_2, \Dsf_2^{-1},\Csf^{1/2}, \Csf^{-1/2},   \csf+{m},\csf-{-m}, \Ksf+{1,0,m}, \Ksf-{1,0,-m}, \Ksf+{1,n,r}, \Ksf-{1,-n,r}, \Xsf+{1,r,s}, \Xsf-{1,r,s}:m\in\N, n\in \N^\times, r,s\in\Z\}$$
subject to the relations
\be\label{eq:Csfcentral} \mbox{$\Csf^{\pm 1/2}$ and $\cbsf\pm(z)$ are central} \ee
\be\label{eq:csbf} \res_{v,w} \frac{1}{vw}\cbsf\pm(v) \cbsf\mp(w) =  1 \,,  \ee
\be \Dsf^{\pm 1}_1 \Dsf^{\mp 1}_1 = 1 \qquad \Dsf^{\pm 1}_2 \Dsf^{\mp 1}_2 = 1 \qquad \Dsf_1\Dsf_2= \Dsf_2\Dsf_1\ee
\be \Dsf_1 \Kbsf\pm{1,\pm m}(z) \Dsf_1^{-1} = q^{\pm m} \Kbsf\pm{1,\pm m}(z) \qquad \Dsf_1 \Xbsf\pm{1,r}(z) \Dsf_1^{-1} = q^{r} \Xbsf\pm{1,r}(z)\,, \ee
\be \Dsf_2 \Kbsf\pm{1,\pm m}(z) \Dsf_2^{-1} = \Kbsf\pm{1,\pm m}(zq^{-1}) \qquad \Dsf_2 \Xbsf\pm{1,r}(z) \Dsf_2^{-1} =  \Xbsf\pm{1,r}(zq^{-1})\,, \ee
\be \res_{v,w} \frac{1}{vw} \Kbsf\pm{1,0}(v) \Kbsf\mp{1,0}(w) 
= 1 \,,  \ee
\be\label{eq:K+K+} (v-q^{\pm 2}z)(v-q^{2(m-n \mp 1)}z) \Kbsf\pm{1,\pm m}(v) \Kbsf\pm{1,\pm n}(z) = (vq^{\pm 2}-z)(vq^{\mp2}-q^{2(m-n)}z) \Kbsf\pm{1,\pm n}(z) \Kbsf\pm{1,\pm m}(v)\,,\ee
\be\label{eq:K+K-} (\Csf q^{2(1-m)}v-w)(q^{2(n-1)}v-\Csf w) \Kbsf+{1, m}(v) \Kbsf-{1,-n}(w) = (\Csf q^{-2m}v - q^{2}w)(q^{2n}v-\Csf q^{-2}w) \Kbsf-{1,-n}(w)\Kbsf+{1, m}(v) \,,\ee
\be\label{eqbf:K+X+} (v-q^{\pm 2}z) \Kbsf\pm{1,\pm m}(v) \Xbsf\pm{1,r}(z) = (q^{\pm 2}v-z)\Xbsf\pm{1,r}(z)\Kbsf\pm{1,\pm m}(v)\,,\ee
\be\label{eq:K+X-} (\Csf v - q^{2(m\mp 1)}z) \Kbsf\pm{1,\pm m}(v) \Xbsf\mp{1,r}(z) = (\Csf q^{\mp 2} v - q^{2m}z) \Xbsf\mp{1,r}(z) \Kbsf\pm{1,\pm m}(v)\,,\ee
\be
(v-q^{\pm 2} w) \Xbsf\pm{1,r}(v) \Xbsf\pm{1,s}(w) =(vq^{\pm 2}- w) \Xbsf\pm{1,s}(w) \Xbsf\pm{1,r}(v)  \,,\label{eq:X+rX+s}\ee
\bea[\Xbsf+{1,r}(v),\Xbsf-{1,s}(z)] &=& \frac{1}{q-q^{-1}} \left \{\delta\left ( \frac{\Csf v}{q^{2(r+s)}z}  \right ) \prod_{p=1}^{|s|} \cbsf-\left(\Csf^{-1/2}q^{\left (2p-1\right ) \sign(s)-1}
z\right )^{-\sign(s)} \Kbsf+{1,r+s}(v) \right .\nn\\
&& \left . -\delta\left ( \frac{\Csf^{-1} v}{q^{2(r+s)}z}  \right ) \prod_{p=1}^{|r|} \cbsf+\left(\Csf^{-1/2}q^{\left (1-2p\right)\sign(r)-1}
v\right )^{\sign(r)} \Kbsf-{1,r+s}(z)\right \}\,,
\label{eq:X+X-}
\label{eq:X+X-KK}
\eea
where $m,n \in \N$, $r,s\in\Z$ and we have set
\be \cbsf\pm(z) = \sum_{m\in\N} \csf\pm{\pm m} z^{\mp m}\,,\ee
\be \Kbsf\pm{1,0}(z) = \sum_{m\in\N} \Ksf\pm{1,0,\pm m} z^{\pm m}\,,\ee
and, for every $m\in\N^\times$ and $r\in\Z$,
\be \Kbsf\pm{1,\pm m} (z) = \sum_{s\in\Z} \Ksf\pm{1,\pm m, s} z^{-s}\,,\ee
\be \Xbsf\pm{1,r}(z) = \sum_{s\in\Z} \Xsf\pm{1,r,s} z^{-s}\,.\ee
In (\ref{eq:X+X-KK}), we further assume that $\Kbsf\pm{1,\mp m}(z) =0$ for every $m\in\N^\times$.
\end{defn}
\begin{defn}
We denote by $\qdaff'(\mathfrak a_1)$ the subalgebra of $\qdaff(\mathfrak a_1)$ generated by
$$\{\Dsf_2, \Dsf_2^{-1},\Csf^{1/2}, \Csf^{-1/2},   \csf+{m},\csf-{-m}, \Ksf+{1,0,m}, \Ksf-{1,0,-m}, \Ksf+{1,n,r}, \Ksf-{1,-n,r}, \Xsf+{1,r,s}, \Xsf-{1,r,s}:m\in\N, n\in \N^\times, r,s\in\Z\}\,,$$
i.e. the subalgebra generated by all the generators of $\qdaff(\mathfrak a_1)$ except $\Dsf_1$ and $\Dsf_1^{-1}$. We shall denote by
$$\jmath:\qdaff'(\mathfrak a_1) \hookrightarrow \qdaff(\mathfrak a_1)$$
the natural injective algebra homomorphism.
\end{defn}

\begin{defn}
\label{def:qdaff0}
We denote by $\qdaff^0(\mathfrak a_1)$ the subalgebra of $\qdaff(\mathfrak{a}_1)$ generated by 
$$\left\{ \Csf^{1/2}, \Csf^{-1/2}, \csf+{m},\csf-{-m}, \Ksf+{1,0,m} , \Ksf-{1,0,-m}, \Ksf+{1,n,r} , \Ksf-{1,-n,r} : m\in\N, n\in\N^\times,r\in\Z \right \}$$
and by $\qdaff^{0,0}(\mathfrak a_1)$ the subalgebra of $\qdaff^0(\mathfrak{a}_1)$ generated by 
$$\left\{ \Csf^{1/2}, \Csf^{-1/2}, \csf+{m},\csf-{-m}, \Ksf+{1,0,m} , \Ksf-{1,0,-m} : m\in\N \right \}\,.$$
Similarly, we denote by $\qdaff^\pm(\mathfrak{a}_1)$ the subalgebra of $\qdaff(\mathfrak{a}_1)$ generated by $\left \{\Xsf\pm{1,r,s} : r,s\in \Z\right \}$. We eventually denote by $\qdaff^\geq(\mathfrak{a}_1)$ (resp. $\qdaff^\leq(\mathfrak{a}_1)$) the subalgebra of $\qdaff(\mathfrak{a}_1)$ generated by
$$ \left \{\Csf^{1/2}, \Csf^{-1/2}, \csf+{m},\csf-{-m}, \Ksf+{1,0,m} , \Ksf-{1,0,-m}, \Ksf+{1,n,r} , \Ksf-{1,-n,r},\Xsf+{1,r,s}  : m\in\N, n\in\N^\times,r, s\in\Z\right \} $$
(resp. $$\left \{\Csf^{1/2}, \Csf^{-1/2}, \csf+{m},\csf-{-m}, \Ksf+{1,0,m} , \Ksf-{1,0,-m}, \Ksf+{1,n,r} , \Ksf-{1,-n,r},\Xsf-{1,r,s}  : m\in\N, n\in\N^\times,r,s \in\Z \right \} $$)
\end{defn}
\begin{rem}
Obviously, $\qdaff^\pm(\mathfrak{a}_1)$ is graded over $Q^\pm$ whereas $\qdaff(\mathfrak a_1)$ is graded over the root lattice $Q$ of $\mathfrak a_1$. $\qdaff(\mathfrak a_1)$ is also graded over $\Z^2 = \Z_{(1)}\times \Z_{(2)}$;
$$\qdaff(\mathfrak a_1) = \bigoplus_{(n_1, n_2)\in\Z^2} \qdaff(\mathfrak a_1)_{(n_1, n_2)}\,,$$
where, for every $(n_1, n_2)\in\Z^2$, we let
$$\qdaff(\mathfrak a_1)_{(n_1, n_2)} = \left\{x\in\qdaff(\mathfrak a_1) : \Dsf_1 x\Dsf_1^{-1} = q^{n_1} x, \quad \Dsf_2 x\Dsf_2^{-1} = q^{n_2} x \right \}\,.$$
\end{rem}

\begin{prop}\label{prop:uq0subalg}
The set $\left\{ \Csf^{1/2}, \Csf^{-1/2}, \Ksf+{1,0,m} , \Ksf-{1,0,-m} : m\in\N \right \}$ generates a subalgebra of $\qdaff^{0,0}(\mathfrak a_1)$ that is isomorphic to $\qaff^0(\mathfrak a_1)$.
\end{prop}
\begin{proof}
This can be directly checked from the defining relations. Otherwise, it suffices to observe that the algebra isomorphism $\widehat\Psi : \widehat\qaff(\dot{\mathfrak a}_1) \to \widehat{\qdaff'(\mathfrak a_1)}$ -- see theorem \ref{thm:main} -- restricts on that set to
$$\widehat\Psi (C^{\pm 1/2})=\Csf^{\pm 1/2} \qquad \mbox{and}\qquad  \widehat\Psi(\kk\pm1(z))= - \Kbsf\mp{1,0}(\Csf^{-1/2} z)\, .$$
\end{proof}

\subsection{$\qdaff(\mathfrak a_1)$ as a topological algebra}
\label{sec:topology}
Because of relation (\ref{eq:X+X-}), the definition of $\qdaff(\mathfrak{a}_1)$ is not purely algebraic. Indeed, the r.h.s. of (\ref{eq:X+X-}) involves two infinite series. One way to make sense of that relation is to equip $\qdaff(\mathfrak a_1)$ -- and, for later use, its tensor powers -- with a topology, such that both series be convergent in the corresponding completion $\widehat{\qdaff(\mathfrak a_1)}$ of $\qdaff(\mathfrak a_1)$. Making use of the natural $\Z_{(2)}$-grading of the tensor algebras $\qdaff(\mathfrak a_1)^{\otimes m}$, $m\in\N^\times$, we let, for every $n\in\N$, 
$$\dot\Omega_{n}^{(m)} := \bigoplus_{\substack{r\geq n\\s \geq n}} \qdaff(\mathfrak a_1)^{\otimes m} \cdot \left (\qdaff(\mathfrak a_1)^{\otimes m}\right )_{-r} \cdot \qdaff(\mathfrak a_1)^{\otimes m}\cdot \left (\qdaff(\mathfrak a_1)^{\otimes m}\right )_{s} \cdot \qdaff(\mathfrak a_1)^{\otimes m}  \,.$$
One easily checks that
\begin{prop}
\label{prop:Omegan}
The following hold true for every $m\in\N^\times$:
\begin{enumerate}
\item [i.] For every $n\in\N$, $\dot\Omega_{n}^{(m)}$ is a two-sided ideal of $\qdaff({\mathfrak a}_1)$;
\item[ii.] For every $n\in\N$, $\dot\Omega_{n}^{(m)}\supseteq \dot\Omega_{n+1}^{(m)}$;
\item[iii.] $\dot\Omega_{0}^{(m)} := \bigcup_{n\in\N}\dot\Omega_{n}^{(m)} = \qdaff({\mathfrak a}_1)$;
\item [iv.] $\bigcap_{n\in\N} \dot\Omega_{n}^{(m)} =\{0\}$;
\item[v.] For every $n,p\in\N$, $\dot\Omega_{n}^{(m)} + \dot\Omega_{p}^{(m)} \subseteq \dot\Omega_{\min(n,p)}$;
\item[vi.] For every $n,p\in\N$, $\dot\Omega_{n}^{(m)} \cdot \dot\Omega_{p}^{(m)} \subseteq \dot\Omega_{\max(n,p)}$.
\end{enumerate}
\end{prop}
\begin{proof}
See \cite{MZ} for a proof in the $\qaff(\dot{\mathfrak a}_1)$ case that can be transposed to the present situation.
\end{proof}

\begin{defprop}
\label{defprop:topol}
We endow $\qdaff({\mathfrak a}_1)$ with the topology $\tau$ whose open sets are either $\emptyset$ or nonempty subsets $\mathcal O\subseteq \qdaff({\mathfrak a}_1)$ such that for every $x\in \mathcal O$, $x+\dot\Omega_{n}\subseteq \mathcal O$ for some $n\in\N$. Similarly, we endow each tensor power $\qdaff({\mathfrak a}_1)^{\otimes m\geq 2}$ with the topology induced by $\{\dot\Omega_n^{(m)}:n\in\N\}$. These turn $\qdaff({\mathfrak a}_1)$ into a (separated) topological algebra. We then let $\widehat{\qdaff({\mathfrak a}_1)}$ denote its completion and we extend by continuity to $\widehat{\qdaff({\mathfrak a}_1)}$ all the (anti)-automorphisms defined over $\qdaff({\mathfrak a}_1)$ and its subalgebras in the previous section
In particular, we extend  $\jmath : {\qdaff'(\mathfrak a_1)} \hookrightarrow \qdaff(\mathfrak a_1)$ into
$$\widehat \jmath : \widehat{\qdaff'(\mathfrak a_1)} \hookrightarrow \widehat{\qdaff(\mathfrak a_1)}\,.$$
Similarly, we denote with a hat the completion of any subalgebra of $\widehat{\qdaff(\mathfrak a_1)}$, like \eg $\widehat{\qdaff^-(\mathfrak a_1)}$, $\widehat{\qdaff^0(\mathfrak a_1)}$ and $\widehat{\qdaff^+(\mathfrak a_1)}$.  We eventually denote by $\qdaff({\mathfrak a}_1)^{\widehat\otimes m\geq 2}$ the corresponding completions of $\qdaff({\mathfrak a}_1)^{\otimes m\geq 2}$.
\end{defprop}
\begin{proof}
This was proven in \cite{MZ}.
\end{proof}
\begin{rem}
As was noted in \cite{MZ}, the above defined topology is actually ultrametrizable. 
\end{rem}

\subsection{The double quantum loop algebra}
\label{rem:qdafftopology}
An alternative way to make sense of relations (\ref{eq:X+X-}) consists in observing that $\qdaff(\mathfrak a_1)$ is \emph{proalgebraic}. Indeed, for every $N\in\N$, let $\qdaff(\mathfrak a_1)^{(N)}$ be the $\F$-algebra generated by 
$$\{\Csf^{1/2}, \Csf^{-1/2},   \csf+{n},\csf-{-n}, \Ksf+{1,0,m}, \Ksf-{1,0,-m}, \Ksf+{1,p,r}, \Ksf-{1,-p,r}, \Xsf+{1,r,s}, \Xsf-{1,r,s}: m\in\N, n\in\range{0}{N}, p\in\N^\times, r,s\in\Z\}$$ 
subject to relations ((\ref{eq:Csfcentral}) -- (\ref{eq:X+X-})), where, this time,
\be \cbsf\pm(z) = \sum_{m=0}^N \csf\pm{\pm m} z^{\mp m}\,.\ee
Now clearly, each $\qdaff(\mathfrak a_1)^{(N)}$ is algebraic since the sums on the r.h.s. of (\ref{eq:X+X-}) are both finite -- whenever $\cbsf\pm(z)^{-1}$ is involved, just multiply through by $\cbsf\pm(z)$ to get an equivalent algebraic relation. Moreover, letting $\mathcal I_N$ be the two-sided ideal of $\qdaff(\mathfrak a_1)^{(N)}$ generated by $\{\csf+N, \csf-{-N}\}$ (resp. $\{\csf+0-1, \csf-0-1\}$) for every $N>1$ (resp. for $N=0$), we obviously have a surjective algebra homomorphism
\be \qdaff(\mathfrak a_1)^{(N)} \longrightarrow \qdaff(\mathfrak a_1)^{(N-1)} \cong \frac{\qdaff(\mathfrak a_1)^{(N)}}{\mathcal I_N}\ee
and we can define $\qdaff(\mathfrak a_1)$ as the inverse limit 
$$\qdaff(\mathfrak a_1) = \lim_{\longleftarrow} \qdaff(\mathfrak a_1)^{(N)}$$
of the system of algebras
$$\cdots \longrightarrow \qdaff(\mathfrak a_1)^{(N)} \longrightarrow \qdaff(\mathfrak a_1)^{(N-1)} \longrightarrow \cdots \longrightarrow \qdaff(\mathfrak a_1)^{(0)}\longrightarrow \qdaff(\mathfrak a_1)^{(-1)}\,.$$

\begin{defn}
We shall refer to the quotient of $\qdaff(\mathfrak a_1)^{(-1)}$ by the two-sided ideal generated by $\left\{\Csf^{1/2}-1 \right\}$ as the \emph{double quantum loop algebra} of type $\mathfrak a_1$ and denote it by $\qdloop(\mathfrak a_1)$. Correspondingly, we denote by $\qdloop^\pm(\mathfrak a_1)$ and $\qdloop^0(\mathfrak a_1)$, the subalgebras of $\qdloop(\mathfrak a_1)$ respectively generated by $\left\{\Xsf\pm{1,r,s} : r,s\in \Z\right\}$
and 
$$\left\{\Ksf+{1,0,m} , \Ksf-{1,0,-m}, \Ksf+{1,n,r} , \Ksf-{1,-n,r} : m\in\N, n\in\N^\times,r\in\Z \right\}\,.$$
We denote by $\qdloop^{0,0}(\mathfrak a_1)$ the subalgebra of $\qdloop^0(\mathfrak a_1)$ generated by
$$\left\{\Ksf+{1,0,m} , \Ksf-{1,0,-m}: m\in\N\right\}\,.$$
It is worth emphasizing that $\qdloop^{0,0}(\mathfrak a_1)$ is abelian.
\end{defn}

\subsection{Triangular decomposition of $\widehat{\qdaff'(\mathfrak a_1)}$}
In \cite{MZ}, we proved that $\widehat{\qdaff'(\mathfrak a_1)}$ has a triangular decomposition in the following sense.
\begin{defn}
Let $A$ be a complete  topological algebra with closed subalgebras $A^\pm$ and $A^0$. We shall say that $(A^-,A^0,A^+)$ is a \emph{triangular decomposition} of $A$ if the multiplication induces a bicontinuous isomorphism of vector spaces $A^-\widehat\otimes A^0 \widehat\otimes A^+ \stackrel{\sim}{\rightarrow} A$.  
\end{defn}
\noi Recalling the definitions of $\qdaff^\pm(\mathfrak a_1)$ and $\qdaff^0(\mathfrak a_1)$ from definition \ref{defn:qdaff}, we have 
\begin{prop}
\label{prop:triang}
$(\qdaff^-(\mathfrak a_1), \qdaff^0(\mathfrak a_1), \qdaff^+(\mathfrak a_1))$ is a triangular decomposition of $\widehat{\qdaff'(\mathfrak a_1)}$ and $\qdaff^\pm(\mathfrak a_1)$ is bicontinuously isomorphic to the algebra generated by 
$\{\Xsf\pm{1,r, s} : r,s\in\Z\}$ subject to relation (\ref{eq:X+rX+s}).
\end{prop}
\begin{proof}
See \cite{MZ}.\end{proof}

\subsection{The closed subalgebra $\widehat{\uqhhzeroslt}$ as a topological Hopf algebra}
\begin{defn}
\label{def:uq0+}
In $\widehat{\qdaff^0(\mathfrak a_1)}$, we define
\be\label{eq:defppm}\pbsf\pm (z) = \sum_{m\in\N} \psf\pm{\pm m} z^{\mp m} = \cbsf\pm(z) \Kbsf\mp{1,0}(\Csf^{-1/2}z)^{-1} \Kbsf\mp{1,0}(\Csf^{-1/2}zq^2)\ee
and for every $m\in\N^\times$,
\be\label{eq:deftp}\tbsf+{1,m}(z) =  \sum_{n\in\N} \tsf+{1,m,n} z^{-n} =-\frac{1}{q-q^{-1}}\Kbsf+{1,0}(zq^{-2m})^{-1}\Kbsf+{1,m}(z)\,,\ee
\be\label{eq:deftm}\tbsf-{1,-m}(z)= \sum_{n\in\N} \tsf-{1,- m,n} z^{n} = \frac{1}{q-q^{-1}} \Kbsf-{1,-m}(z) \Kbsf-{1,0}(zq^{-2m})^{-1}\,.\ee
Then, we let ${\qdaff^{0^+}(\mathfrak a_1)}$ be the subalgebra of $\widehat{\qdaff^0(\mathfrak a_1)}$ generated by 
$$\{\Csf^{1/2},\Csf^{-1/2}, \psf+{m},\psf-{-m}, \tsf +{1,p,n}, \tsf -{1,-p,n}: m\in\N, n\in\Z, p \in \N^\times\}\,.$$
and we let $\widehat{\qdaff^{0^+}(\mathfrak a_1)}$ be its completion in the inherited topology.
\end{defn}
Clearly, the closed subalgebra $\widehat{\qdaff^0(\mathfrak a_1)}$ can be presented as in definition \ref{def:qdaff0} or, equivalently, in terms of the generators in
$$\{\Csf^{1/2},\Csf^{-1/2},  \csf+{m},\csf-{-m}, \psf+{m},\psf-{-m}, \tsf +{1,p,n}, \tsf -{1,-p,n}: m\in\N, n\in\Z, p \in \N^\times\}\,.$$
In section \ref{sec:Hopfalg}, we will endow $\widehat{\qdaff'(\mathfrak a_1)}$ with a topological Hopf algebraic structure. It turns out that, for that structure, the closed subalgebra $\widehat{\uqhhzeroslt}$ is not a closed Hopf subalgebra of $\widehat{\qdaff(\mathfrak a_1)}$. However, it is possible to endow $\widehat{\qdaff^0(\mathfrak a_1)}$ with its own topological Hopf algebraic structure as follows.
\begin{defprop}
We endow $\widehat{\uqhhzeroslt}$ with:
\begin{enumerate}[i.]
\item the comultiplication $\Delta^0:\widehat\uqhhzeroslt \to \uqhhzeroslt\widehat\otimes\uqhhzeroslt$ defined by
\be
\label{eq:coprodC}
\Delta^0(\Csf^{\pm 1/2}) = \Csf^{\pm1/2} \otimes \Csf^{\pm 1/2}
\ee
\be
\label{eq:coprodcm}
\Delta^0(\cbsf\pm(z)) = \cbsf\pm(z\Csf_{(2)}^{\pm1/2} ) \otimes \cbsf\pm(z\Csf_{(1)}^{\mp1/2})\,, 
\ee
\be
\label{eq:coprodppm}
\Delta^0(\pbsf\pm(z)) = \pbsf\pm(z\Csf_{(2)}^{\pm1/2}) \otimes \pbsf\pm(z\Csf_{(1)}^{\mp1/2})\,,
\ee
\bea
\Delta^0(\tbsf+{1,m}(z))&=& \tbsf+{1,m}(z)\otimes 1 + \prod_{k=1}^m \pbsf-(zq^{-2k}\Csf_{(1)}^{1/2}) \widehat\otimes\tbsf+{1,m}(z\Csf_{(1)})\nn\\
&&-(q-q^{-1}) \sum_{k=1}^{m-1} \prod_{l=k+1}^m \pbsf-(zq^{-2l}\Csf_{(1)}^{1/2}) \tbsf+{1,k}(z) \widehat\otimes\tbsf+{1,m-k} (zq^{-2k} \Csf_{(1)})\,, \label{eq:coprodtp}
\eea
\bea
\Delta^0(\tbsf-{1,-m}(z))&=& \tbsf-{1,-m}(z \Csf_{(2)}) \widehat\otimes  \prod_{k=1}^m \pbsf+(zq^{-2k}\Csf_{(2)}^{1/2})+1\otimes \tbsf-{1,-m}(z)\nn\\
&&+(q-q^{-1}) \sum_{k=1}^{m-1} \tbsf-{1,-(m-k)}(zq^{-2k}\Csf_{(2)}) \widehat\otimes\tbsf-{1,-m}(z) \prod_{l=+1}^m\pbsf+(zq^{-2l} \Csf_{(2)}^{1/2}) \,,\label{eq:coprodtm}
\eea
for every $m\in\N$, where $\Csf_{(1)}^{\pm1/2} = \Csf^{\pm1/2} \otimes 1$ and $\Csf_{(2)}^{\pm1/2} = 1\otimes \Csf^{\pm 1/2}$,
\item the counit $\varepsilon(\Csf) = \varepsilon^0(\cbsf\pm(z))=\varepsilon^0(\pbsf\pm(z))= 1$, $\varepsilon^0(\tbsf\pm{1,\pm m}(z))=0$, for every $m\in\N$,
\item and the antipode defined by
\be S^0(\Csf^{\pm 1/2}) = \Csf^{\mp1/2}\,, \ee
\be S^0(\cbsf\pm(z)) = \cbsf\pm(z)^{-1}\,,\ee
\be\label{eq:antipppm}
S^0(\pbsf\pm(z))  = \pbsf\pm(z) ^{-1}\,,
\ee
\be\label{eq:antiptp}
S^0(\tbsf+{1,m} (z)) = - \prod_{k=1}^m \pbsf-(zq^{-2k} \Csf^{-1/2})^{-1} \sum_{n=1}^m \sum_{\lambda\in C_n(m)} (-1)^{n-1} c_{m,\lambda} \tbsf+{1,\lambda} (z\Csf^{-1})\,,
\ee
\be\label{eq:antiptm}
S^0(\tbsf-{1,-m}(z)) = -\sum_{n=1}^m \sum_{\lambda\in C_n(m)} c_{m,\lambda} \tbsf-{1,-\lambda}(z\Csf^{-1}) \prod_{k=1}^m \pbsf+(zq^{-2k})^{-1}\,,
\ee
where we have set, for every $m\in\N^\times$ and every $\lambda\in C_n(m)$,
$$c_{m,\lambda} = (q-q^{-1})^{n-1} \frac{[m+1]_q}{[m-1]_q} \prod_{i=1}^n \frac{[\lambda_i-1]_q}{[\lambda_i+1]_q}$$
and
$$\tbsf+{1,\lambda}(z\Csf^{-1}) = \overleftarrow{\prod_{i\in\rran{n}}} \tbsf+{1,\lambda_i} (zq^{-2\sum_{k=i+1}^n\lambda_k}\Csf^{-1})\,,$$
$$\tbsf-{1,-\lambda}(z\Csf^{-1}) = \overrightarrow{\prod_{i\in\rran{n}}} \tbsf-{1,-\lambda_i} (zq^{-2\sum_{k=i+1}^n\lambda_k}\Csf^{-1})\,.$$
for every $m\in\N$.
\end{enumerate}
With these operations, $\widehat\uqhhzeroslt$ is a topological Hopf algebra.
\end{defprop}
\begin{proof}
One easily checks that $\Delta^0$ as defined by (\ref{eq:coprodC} -- \ref{eq:coprodtm}) is compatible with the defining relations of $\widehat\uqhhzeroslt$ and that $S^0$ is compatible with both the multiplication and the comultiplication.
\end{proof}
In that presentation, one readily checks that
\begin{prop}
$\widehat{\qdaff^{0^+}(\mathfrak a_1)}$ is a closed Hopf subalgebra of $\widehat\uqhhzeroslt$.
\end{prop}
\begin{proof}
$\widehat{\qdaff^{0^+}(\mathfrak a_1)}$ is a closed subalgebra of $\widehat\uqhhzeroslt$ and it is clearly stable under $\Delta^0$ and $S^0$.
\end{proof}

\subsection{The closed subalgebra $\widehat{\uqhhzeroslt}$ and the elliptic Hall algebra}
As emphasized in \cite{MZ}, another remarkable feature of $\widehat{\qdaff'(\mathfrak a_1)}$ and, more particularly of its closed subalgebra $\widehat{\qdaff^0(\mathfrak a_1)}$, is the existence of an algebra homomorphism onto it, from the elliptic Hall algebra that we now define.
\begin{defn}
Let $q_1, q_2, q_3$ be three (dependent) formal variables such that $q_1q_2q_3=1$. The \emph{elliptic Hall algebra} $\mathcal E_{q_1,q_2,q_3}$ is the $\Q(q_1,q_2,q_3)$-algebra generated by $\left \{ C^{1/2}, \C^{-1/2}, \psi^+_m, \psi^-_{-m}, e^+_n, e^-_n : m\in\N, n\in\Z\right \}$, with $\psi^\pm_0$ invertible, subject to the relations
\be \mbox{$C^{\pm 1/2}$ is central\,,}\ee
\be\label{eq:psi+psi+} \psib\pm (z) \psib\pm(w) = \psib\pm(w)\psib\pm(z)\,,\ee
\be\label{eq:psi+psi-} g(Cz,w)g(Cw,z)\psib+(z)\psib-(w)=g(z,Cw)g(w,Cz) \psib-(w)\psib+(z)\,,\ee
\be\label{eq:psie+} g(C^{\frac{1\pm 1}{2}}z,w)\psib\pm(z) \eb+(w) = - g(w, C^{\frac{1\pm 1}{2}}z) \eb+(w) \psib\pm(z)  \,,\ee
\be\label{eq:psie-} g(w, C^{\frac{1\mp 1}{2}}z)  \psib\pm(z) \eb-(w)=  -g(C^{\frac{1\mp 1}{2}}z,w) \eb-(w) \psib\pm(z)\,, \ee
\be\label{eq:e+e-} [\eb+(z),\eb-(w) ] = \frac{1}{g(1,1)} \left [\delta\left (\frac{Cw}{z} \right )\psib+(w) -\delta\left (\frac{w}{Cz} \right )\psib-(z)\right ]\,,\ee
\be\label{eq:e+e+} g(z,w)\eb+(z)\eb+(w)= -g(w,z)\eb+(w)\eb+(z)\,,\ee
\be\label{eq:e-e-} g(w,z)\eb-(z)\eb-(w)= -g(z,w)\eb-(w)\eb-(z)\,,\ee
\be\label{eq:e+e+e+} \res_{v,w,z} (vwz)^m(v+z)(w^2-vz) \eb\pm(v) \eb\pm(w) \eb\pm(z)=0\,,\ee
where $m\in\Z$ and we have introduced
\be g(z,w) =(z-q_1w)(z-q_2w)(z-q_3w)\,, \ee
\be \psib\pm(z) = \sum_{m\in\N} \psi^\pm_{\pm m} z^{\mp m}\,,\ee
\be \eb\pm(z) = \sum_{m\in\Z} e^\pm_m z^{-m}\,.\ee
\end{defn}
\begin{rem}
The elliptic Hall algebra $\E{q_1}{q_2}{q_3}$ is $\Z$-graded and can be equipped with a natural topology along the lines of what we did for $\qdaff({\mathfrak a}_1)$ in section \ref{sec:topology}. It then becomes a topological algebra and we denote by $\widehat{\E{q_1}{q_2}{q_3}}$ its completion. Similar topologies can be constructed on its tensor powers.
\end{rem}
\begin{defprop}
We endow $\widehat{\E{q_1}{q_2}{q_3}}$ with:
\begin{enumerate}[i.]
\item the comultiplication $\Delta_{\mathcal E} :\widehat{\E{q_1}{q_2}{q_3}} \to \E{q_1}{q_2}{q_3}\widehat\otimes \E{q_1}{q_2}{q_3}$ defined by
\be \Delta_{\mathcal E}(\psib\pm(z)) = \psib\pm(zC^{\frac{1\pm 1}{2}}_{(2)}) \otimes \psib\pm(zC^{\frac{1\mp 1}{2}}_{(1)}) \,,\ee
\be
\Delta_{\mathcal E} (\eb+(z)) = \eb+(z)\otimes 1 + \psib-(z) \widehat\otimes\eb+(zC_{(1)})\,,
\ee
\be
\Delta_{\mathcal E}(\eb-(z)) = \eb-(zC_{(2)}) \widehat\otimes \psib+(z) + 1\otimes \eb-(z)\,,
\ee
\item the counit $\varepsilon_{\mathcal E} : \widehat{\E{q_1}{q_2}{q_3}}\to \F$ defined by $\varepsilon_{\mathcal E}(C^{\pm1/2}) = \varepsilon_{\mathcal E}(\psib\pm(z))=1$, $\varepsilon_{\mathcal E}(\eb\pm(z)) = 0$,
\item the antipode $S_{\mathcal E} : \widehat{\E{q_1}{q_2}{q_3}}\to\widehat{\E{q_1}{q_2}{q_3}}$ defined by
\be S_{\mathcal E} (\psib\pm(z)) = \psib\pm(zC^{-1})^{-1} \,,\ee
\be S_{\mathcal E} (\eb+(z))= -\psib-(zC^{-1})^{-1} \eb+(zC^{-1})\,,\ee
\be S_{\mathcal E} (\eb-(z))= - \eb-(zC^{-1})\psib+(zC^{-1})^{-1}\,.\ee
\end{enumerate}
As usual, we have set $C_{(1)}^{\pm 1/2} = C^{\pm 1/2} \otimes 1$ and $C_{(2)}^{\pm1/2} = 1\otimes C^{\pm1/2}$. With the above defined operations, $\widehat{\E{q_1}{q_2}{q_3}}$ is a topological Hopf algebra.
\end{defprop}

\begin{prop}
\label{prop:Hall}
There exists a unique continuous Hopf algebra homomorphism $f:\widehat{\mathcal E_{q^{-4}, q^2,q^2}} \to \widehat{\qdaff^{0^+}(\mathfrak a_1)}$ such that
\be\label{eq:f(C)} f(C^{1/2}) = \Csf^{1/2}\,,\ee
\be f(\psib\pm(z)) =   \pbsf\pm (\Csf^{1/2} zq^{-2}) \,,\ee
\be\label{eq:f(ebp)} f(\eb+(z)) =  \tbsf+{1,1}\,,(z)\ee
\be\label{eq:f(ebm)} f(\eb-(z)) =  \frac{\tbsf-{1,-1}(z)}{(q^2-q^{-2})^2}\,.\ee
\end{prop}
\begin{proof}
In \cite{MZ}, we proved that the assignment 
$$C^{1/2}\mapsto \Csf^{1/2}\, \quad \psib\pm(z)\mapsto (q^2-q^{-2})^2 \, \pbsf\pm (\Csf^{1/2} zq^{-2}) \,,\quad \eb\pm(z)\mapsto  \tbsf\pm{1,\pm 1}(z)$$
defined an $\F$-algebra homomorphism. Hence, $f$, which is obtained from the above assignment by  rescaling the images of $\pbsf\pm(z)$ and $\eb-(z)$, is obviously an $\F$-algebra homomorphism. Moreover, it suffices to write (\ref{eq:coprodppm}),  (\ref{eq:coprodtp}) and (\ref{eq:coprodtm}) with $m=1$, to get
$$\Delta^0(\pbsf\pm(z)) = \pbsf\pm(z\Csf_{(2)}^{\pm1/2}) \otimes \pbsf\pm(z\Csf_{(1)}^{\mp1/2})\,,
$$
$$\Delta^0(\tbsf+{1,1}(z))= \tbsf+{1,1}(z)\otimes 1 + \pbsf-(zq^{-2}\Csf_{(1)}^{1/2}) \widehat\otimes\tbsf+{1,1}(z\Csf_{(1)})\,,$$
$$\Delta^0(\tbsf-{1,-1}(z))= \tbsf-{1,-1}(z \Csf_{(2)}) \widehat\otimes \pbsf+(zq^{-2}\Csf_{(2)}^{1/2})+1\otimes \tbsf-{1,-1}(z) \,,$$
as well as (\ref{eq:antipppm}), (\ref{eq:antiptp}) and (\ref{eq:antiptm}), with $m=1$, to get
$$S^0(\pbsf\pm(z))  = \pbsf\pm(z) ^{-1}\,,$$
$$ S^0(\tbsf+{1,1} (z)) = - \pbsf-(zq^{-2} \Csf^{-1/2})^{-1}  \tbsf+{1,1} (z\Csf^{-1})\,,$$
$$S^0(\tbsf-{1,-1}(z)) = -  \tbsf-{1,-1}(z\Csf^{-1}) \pbsf+(zq^{-2})^{-1}\,,
$$
and thus to prove that $(f\widehat\otimes f) \circ \Delta^0 = \Delta_{\mathcal E} \circ f$ and $f\circ S^0 = S_{\mathcal E} \circ f$ as claimed. 
\end{proof}
\begin{rem}
Note that we have $f(\psi^+_0)f(\psi^-_0)= f(\psi^-_0)f(\psi^+_0) = 1$, meaning that $f$ descends to the quotient of $\mathcal E_{q^{-4}, q^2, q^2}$ by the two-sided ideal generated by $\{\psi^+_0\psi^-_0-1, \psi^-_0\psi^+_0-1\}$. That quotient is actually Miki's $(q, \gamma)$-analogue of the $W_{1+\infty}$ algebra \cite{Miki}.
\end{rem}

\subsection{The quantum toroidal algebra $\qaff(\dot{\mathfrak a}_1)$}
\label{sec:qaff}
Let $\dot{I}=\{0,1\}$ be a labeling of the nodes of the Dynkin diagram of type $\dot{\mathfrak a}_1$ and let $\dot\Phi = \left\{\alpha_0, \alpha_1\right \}$ be a choice of simple roots for the corresponding root system. Let $\dot Q^\pm = \Z^\pm \alpha_0 \oplus \Z^\pm  \alpha_1$ and let $\dot Q = \Z\alpha_0 \oplus \Z\alpha_1$ be the type $\dot{\mathfrak a}_1$ root lattice.
\begin{defn}
\label{def:defqaffdota1}
The \emph{quantum toroidal algebra} $\uqslthh$ is the associative $\F$-algebra generated by the generators 
$$\left \{D, D^{-1}, C^{1/2}, C^{-1/2}, k_{i, n}^+, k_{i, -n}^-,  x_{i,m}^+, x_{i,m}^- : i \in \dot{I}, m \in Z, n \in \N\right\}$$ 
subject to the following relations
\be\label{eq:ccentral} \mbox{$C^{\pm1/2}$ is central} \qquad C^{\pm 1/2} C^{\mp 1/2} = 1 \qquad D^{\pm 1} D^{\mp 1} =1\ee
\be D\kk\pm i(z)D^{-1} = \kk\pm i(zq^{-1}) \qquad D\x\pm i(z)D^{-1} = \x\pm i(zq^{-1}) \ee
\be \kk \pm i(z_1) \kk \pm j(z_2) =  \kk \pm j(z_2) \kk \pm i(z_1)\ee
\be \kk - i(z_1) \kk + j(z_2) = G^-_{ij}(C^{-1}z_1/z_2) G^+_{ij}(Cz_1/z_2)  \kk + j(z_2)  \kk - i(z_1) = 1 \mod z_1/z_2 \label{eq:kpkm}\ee
\be\label{eq:kpxpm} G_{ij}^\mp(C^{\mp 1/2} z_2/z_1)  \kk +i(z_1) \x \pm j(z_2) = \x \pm j(z_2)  \kk +i(z_1)\ee
\be\label{eq:kmxpm} \kk -i(z_1) \x \pm j(z_2) = G_{ij}^\mp(C^{\mp 1/2} z_1/z_2) \x \pm j(z_2)  \kk -i(z_1)\ee
\be\label{eq:xpmxpm} (z_1-q^{\pm c_{ij}} z_2) \x \pm i(z_1) \x \pm j(z_2) = (z_1q^{\pm c_{ij}}-z_2) \x \pm j(z_2) \x \pm i(z_1)\ee
\be\label{eq:relx+x-} [\x + i(z_1), \x - j(z_2)] = \frac{\delta_{ij}}{q-q^{-1}} \left [ \delta \left ( \frac{z_1}{Cz_2} \right ) \kk + i(z_1C^{-1/2}) - \delta \left( \frac{z_1C}{z_2}\right ) \kk - i(z_2C^{-1/2}) \right ]  \ee
\be\label{eq:qaffserre} \sum_{\sigma \in S_{1-c_{ij}}} \sum_{k=0}^{1-c_{ij}} (-1)^k {{1-c_{ij}}\choose{k}}_q \x \pm i(z_{\sigma(1)})  \cdots \x \pm i(z_{\sigma(k)}) \x \pm j(z)  \x \pm i(z_{\sigma(k+1)}) \cdots \x\pm i(z_{\sigma(1-c_{ij})}) =0  \ee
where, for every $i \in \dot I$, we define the following $\uqslthh$-valued formal distributions
\be \x \pm i(z) := \sum_{m \in \Z} x^\pm_{i, m} z^{-m} \in \uqslthh[[z, z^{-1}]]\,;\ee
\be \kk \pm i(z) := \sum_{n \in \N} k_{i, \pm n}^\pm z^{\mp n} \in \uqslthh[[z^{\mp 1}]]\, ,\ee
for every $i,j\in\dot I$, we define the following $\F$-valued formal power series
\be G_{ij}^\pm(z):= q^{\pm c_{ij}} + (q-q^{-1}) [\pm c_{ij} ]_q \sum_{m \in \N^\times} q^{\pm m c_{ij}} z^m \in \F[[z]] \ee
is an $\F$-valued formal distribution,
\end{defn}
Note that $G_{ij}^\pm(z)$ is invertible in $\F[[z]]$ with inverse $G^\mp_{ij}(z)$, \ie
\be G_{ij}^\pm(z) G_{ij}^{\mp}(z) = 1\, ,\ee
and that it can be viewed as the power series expansion of a rational function of $(z_1, z_2) \in \C^2$ as $|z_2| \gg |z_1|$, which we shall denote as follows
\be G_{ij}^\pm(z_1/z_2) = \left ( \frac{z_1q^{\mp c_{ij}} - z_2}{z_1-q^{\mp c_{ij}} z_2} \right )_{|z_2| \gg |z_1|} \, .\ee
Observe furthermore that we have the following useful identity in $\F[[z, z^{-1}]]$
\be\label{eq:G+G-} \frac{G_{ij}^\pm(z_1/z_2)- G_{ij}^\mp(z_2/z_1) }{q-q^{-1}} = [\pm c_{ij}]_q \delta \left( \frac{z_1 q^{\pm c_{ij}}}{z_2} \right )\, . \ee
\begin{rem}
\label{rem:Gij}
In type ${\mathfrak a}_1$, $\dot{I}=\{0,1\}$, $c_{ij} = 4\delta_{ij} -2$ and we have an additional identity, namely $G_{10}^\pm(z) = G_{11}^\mp(z)$. 
\end{rem}
\noi $\qaff(\dot{\mathfrak a}_1)$ is obviously a $\Z$-graded algebra, \ie we have
\be \uqslthh = \bigoplus_{n \in \Z} \uqslthh_n \, , \qquad \mbox{where for all $n \in \Z$} \qquad \uqslthh_n := \{x \in \uqslthh: DxD^{-1}=q^n x \}\, .\label{Zgrad}\ee
It was proven in \cite{Hernandez05} to admit a triangular decomposition $(\uqmslthh, \uqzeroslthh, \uqpslthh)$, where $\uqpmslthh$ and $\uqzeroslthh$ are the subalgebras of $\qaff(\dot{\mathfrak a}_1)$ respectively generated by $\left\{x_{i, m}^\pm : i \in \dot I, m \in \Z\right \}$ and 
$$\left\{C^{1/2}, C^{-1/2}, D, D^{-1}, k_{i, m}^+, k_{i, m}^-: i \in \dot I, m \in \Z\right\}\,.$$
Observe that $\uqpmslthh$ admits a natural gradation over $\dot Q^\pm$ that we shall denote by
\be \uqpmslthh = \bigoplus_{\alpha \in \dot Q^\pm} \uqpmslthh_\alpha\, .\ee
Of course $\qaff(\dot{\mathfrak a}_1)$ is graded over the root lattice $\dot Q$. We finally remark that the two Dynkin diagram subalgebras $\qaff(\mathfrak a_1)^{(0)}$ and $\qaff(\mathfrak a_1)^{(1)}$ of $\qaff(\dot{\mathfrak a}_1)$ generated by 
$$\left \{D, D^{-1}, C^{1/2}, C^{-1/2}, k_{i, n}^+, k_{i, -n}^-,  x_{i,m}^+, x_{i,m}^- :  m \in Z, n \in \N\right\}\,,$$ 
with $i=0$ and $i=1$ respectively, are both isomorphic to $\qaff(\mathfrak{a}_1)$, thus yielding two injective algebra homomorphisms $\iota^{(i)}: \qaff(\mathfrak a_1) \hookrightarrow \qaff(\dot{\mathfrak a}_1)$.
\noi In \cite {MZ}, making use of their natural $\Z$-grading, $\qaff(\dot{\mathfrak a}_1)$ and all its tensor powers were endowed with a topology along the lines of what we did in section \ref{sec:topology} for $\qdaff(\mathfrak a_1)$ and its tensor powers, and subsequently completed into $\widehat{\qaff(\dot{\mathfrak a}_1)}$ and $\qaff(\dot{\mathfrak a}_1)^{\widehat\otimes r}$. The main result in \cite{MZ} is the following
\begin{thm}
\label{thm:main}
There exists a unique bicontinuous $\F$-algebra isomorphism $\widehat\Psi:\widehat{\qaff(\dot{\mathfrak a}_1)} \stackrel{\sim}{\longrightarrow} \widehat{\qdaff'(\mathfrak a_1)}$ such that
$$\widehat\Psi(D^{\pm 1}) = \Dsf_2^{\pm1} \qquad \widehat\Psi (C^{\pm 1/2})=\Csf^{\pm 1/2}\,,$$
$$\widehat\Psi(\kk\pm0(z))= -\cbsf\pm(z)\Kbsf\mp{1,0}(\Csf^{-1/2} z)^{-1} \qquad \widehat\Psi(\kk\pm1(z))= - \Kbsf\mp{1,0}(\Csf^{-1/2} z)$$
$$\widehat\Psi(\x+{0}(z)) =-\cbsf-(\Csf^{1/2} z) \Kbsf+{1,0}(z)^{-1} \Xbsf-{1,1}(\Csf z) \qquad \widehat\Psi(\x-{0}(z)) =-\Xbsf+{1,-1}(\Csf z)\cbsf+(\Csf^{1/2} z) \Kbsf-{1,0}(z)^{-1} $$
$$\widehat\Psi(\x\pm{1}(z)) = \Xbsf\pm{1,0}(z)\,.$$
\end{thm}
\begin{proof} See \cite{MZ} for a proof.
\end{proof}

\subsection{$\qaff(\mathfrak{a}_1)$ subalgebras of $\qdaff(\mathfrak{a}_1)$}
Interestingly, $\qdaff(\mathfrak{a}_1)$ admits countably many embeddings of the quantum affine algebra $\qaff(\mathfrak{a}_1)$. This is the content of the following
\begin{prop}
For every $m\in\Z$, there exists a unique injective algebra homomorphism $\iota_m : \qaff(\mathfrak{a}_1)\hookrightarrow\widehat{\qdaff'(\mathfrak{a}_1)}$ such that
\be \iota_m(C^{\pm 1/2}) = \Csf^{\pm 1/2} \qquad \iota_m (D^{\pm 1}) = \Dsf_2^{\pm1}\label{eq:iotamC}\ee
\be \iota_m(\kk\pm 1(z)) =  -\prod_{p=1}^{|m|} \cbsf\pm\left(q^{(1-2p)\sign(m)-1}z\right )^{\sign(m)} \Kbsf\mp{1,0}(\Csf^{-1/2}z)\,,\ee
\be \iota_m(\x\pm1(z)) = \Xbsf\pm{1,\pm m}(z)\,.\label{eq:iotamx}\ee
\end{prop}
\begin{proof}
See \cite{MZ}.
\end{proof}
We also have
\begin{prop}
For every $i\in \dot I=\{0, 1\}$, $\widehat\Psi \circ \iota^{(i)}$ is an injective algebra homomorphism.
\end{prop}
\begin{proof}
This is obvious since $\widehat\Psi$ is an isomorphism and $\iota^{(i)}$ is an injective algebra homomorphism.
\end{proof}

\subsection{(Anti-)Automorphisms of $\widehat{\qdaff'(\mathfrak a_1)}$}
$\widehat{\qdaff'(\mathfrak a_1)}$ naturally inherits, through $\widehat \Psi$, all the continuous (anti-)automorphisms defined over $\widehat{\qaff(\dot{\mathfrak a}_1)}$.
\begin{prop}
Conjugation by $\widehat \Psi$ clearly provides a group isomorphism $\Aut{\widehat{\qaff(\dot{\mathfrak a}_1)}} \cong \Aut{\widehat{\qdaff'(\mathfrak a_1)}}$. In particular, for every $f\in \Aut{\widehat{\qaff(\dot{\mathfrak a}_1)}}$, we let $\dot f = \widehat \Psi \circ f \circ \widehat \Psi^{-1} \in \Aut{\widehat{\qdaff'(\mathfrak a_1)}}$.
\end{prop}
As an example, consider the Cartan anti-involution $\varphi$ of $\qaff(\dot{\mathfrak a}_1)$ defined in \cite{MZ}. It extends by continuity into an anti-involution $\widehat\varphi$ over $\widehat{\qaff(\dot{\mathfrak a}_1)}$ which eventually yields, upon conjugation by $\widehat\Psi$, an anti-involution $\dot\varphi$ over $\widehat{\qdaff'(\mathfrak a_1)}$. One can easily check -- or take as a definition of $\dot\varphi$ the fact -- that,
$$\dot\varphi(q) =q^{-1}\,, \qquad \dot\varphi(\Dsf_2^{\pm 1}) = \Dsf_2^{\mp1}\,, \qquad \dot\varphi(\Csf^{\pm1/2})=\Csf^{\mp1/2}\,,\qquad \dot\varphi(\cbsf\pm(z)) = \cbsf\mp(1/z) \,, $$
$$\dot\varphi(\Kbsf\pm{1,\pm m}(z))= \Kbsf\mp{1,\mp m}(1/z)\,, \qquad \dot\varphi(\Xbsf\pm{1,r}(z))= \Xbsf\mp{1,-r}(1/z)\,.$$
for every $m\in \N$ and every $r\in \Z$.

In addition to the above, $\widehat{\qdaff'(\mathfrak a_1)}$ also admits the following automorphisms that will prove useful in the study of its representation theory.
\begin{prop}\label{prop:tausigma}
\begin{enumerate}
\item[i.] There exists a unique $\F$-algebra automorphism $\tau$ of $\widehat{\qdaff'(\mathfrak a_1)}$ such that, for every $m\in\N$ and every $n\in\Z$,
\be\nn \tau(\Csf) = -\Csf\,,\quad \tau(\cbsf\pm(\Csf^{-1/2} z))= \cbsf\pm(\mp\Csf^{-1/2} z)\,, \quad \tau(\Kbsf\pm{1,\pm m}(z))=  \Kbsf\pm{1,\pm m}(\mp z) \,,\quad \tau(\Xbsf\pm{1,n}(z))= \Xbsf\pm{1,n}(\mp z)\,. \ee
\item[ii.] There exists a unique $\F$-algebra automorphism $\sigma$ of $\widehat{\qdaff'(\mathfrak a_1)}$ such that
\be\nn \sigma(\Csf^{1/2}) = -\Csf^{1/2}\,,\quad \sigma(\cbsf\pm(z))= \cbsf\pm(z)\,, \quad \tau(\Kbsf\pm{1,\pm m}(z))=  \Kbsf\pm{1,\pm m}(-z) \,,\quad \tau(\Xbsf\pm{1,n}(z))= \Xbsf\pm{1,n}(-z)\,. \ee
\end{enumerate} 
\end{prop}
\begin{proof}
It suffices to check the defining relations of $\widehat{\qdaff(\mathfrak a_1)}$.
\end{proof}

\subsection{Topological Hopf algebra structure on $\widehat{\qdaff'(\mathfrak a_1)}$}
\label{sec:Hopfalg}
\begin{defn}
\label{eq:Hopfalgqaffdot}
We endow the topological $\F$-algebra $\widehat\uqslthh$ with: 
\begin{enumerate}
\item[i.] the comultiplication $\Delta:\widehat\uqslthh\to \uqslthh\widehat{\otimes} \uqslthh$ defined by
\be
\Delta(C^{\pm1/2}) = C^{\pm 1/2} \otimes C^{\pm 1/2}\,,\qquad \Delta(D^{\pm 1}) = D^{\pm 1}\otimes D^{\pm 1}\,,\ee
\be\Delta(\kk\pm i(z)) =\kk\pm i(z C^{\pm 1/2}_{(2)})\otimes \kk\pm i(z C^{\mp 1/2}_{(1)})\,,\ee
\be\label{eq:coprodxip} \Delta(\x+i(z)) =\x+i(z)\otimes 1 + \kk-i(z C^{1/2}_{(1)})\widehat{\otimes} \x+i(z C^{}_{(1)})\,, \ee
\be \Delta(\x-i(z)) =\x-i(z C^{}_{(2)})\widehat{\otimes} \kk+i(z C^{1/2}_{(2)})+ 1  \otimes \x-i(z)\,, \ee
where $C^{\pm 1/2}_{(1)} = C^{\pm 1/2} \otimes 1$ and $C^{\pm 1/2}_{(2)} = 1\otimes C^{\pm 1/2}$;
\item[ii.] the counit $\varepsilon : \widehat\uqslthh\to \F$, defined by $\varepsilon(D^{\pm1}) = \varepsilon(C^{\pm 1/2}) = \varepsilon (\kk\pm i(z))=1$, $\varepsilon(\x\pm i(z))=0$ and; 
\item[iii.] the antipode $S:\widehat\uqslthh\to \widehat\uqslthh$, defined by $S(D^{\pm 1}) = D^{\mp1}$, $S(C^{\pm 1/2}) = C^{\mp 1/2}$ and
$$S(\kk\pm i(z)) = \kk\pm i(z)^{-1}\,, \quad S(\x+i(z)) = - \kk-i(z C^{-1/2})^{-1} \x+i(z C^{-1})\,, \quad S(\x-i(z))=- \x-i(z C^{-1}) \kk+i(zC^{-1/2})^{-1} \,.$$
\end{enumerate}
With these operations so defined and the topologies defined in section \ref{sec:qaff}, $\widehat\uqslthh$ is a topological Hopf algebra.
\end{defn}
In view of theorem \ref{thm:main}, it is clear that $\widehat{\qdaff(\mathfrak a_1)}$ inherits that topological Hopf algebraic structure.
\begin{defprop}
We define
\be \dot\Delta= \left (\widehat\Psi \widehat{\otimes} \widehat\Psi\right ) \circ \Delta\circ \widehat \Psi^{-1} \,,\ee
\be \dot S = \widehat\Psi\circ S\circ\widehat\Psi^{-1}\,,\ee
\be \dot\varepsilon = \varepsilon\circ\widehat\Psi^{-1}\,.\ee
Equipped with the above comultiplication, antipode and counit, $\widehat{\qdaff'(\mathfrak a_1)}$ is a topological Hopf algebra. 
\end{defprop}
Before we move on to introducing $t$-weight $\qdaff'(\mathfrak a_1)$-modules, we give the following
\begin{lem}
\label{lem:DeltaPsiXp}
For every $m\in\N^\times$ and every $r\in\Z$, we have
\begin{enumerate}[i.]
\item\label{it:DeltaPsi} $\dot{\Delta}(\Kbsf\pm{1,\pm m}(z)) =\Delta^0(\Kbsf\pm{1,\pm m}(z)) \mod \qdaff^<(\mathfrak a_1)\widehat\otimes \qdaff^>(\mathfrak a_1)[[z, z^{-1}]]$;
\item\label{it:DeltaXp} $\dot{\Delta}(\Xbsf+{1,r}(z)) \in \left (\qdaff^>(\mathfrak a_1)\widehat\otimes \qdaff^0(\mathfrak a_1)\oplus \qdaff(\mathfrak a_1)\widehat\otimes \qdaff^>(\mathfrak a_1)\right ) [[z,z^{-1}]]$;
\end{enumerate}
where we have set $\qdaff^>(\mathfrak a_1)= \qdaff^\geq(\mathfrak a_1) - \qdaff^\geq(\mathfrak a_1) \cap \qdaff^0(\mathfrak a_1)$ and $\qdaff^<(\mathfrak a_1)=\qdaff^\leq(\mathfrak a_1)-\qdaff^\leq(\mathfrak a_1)\cap \qdaff^0(\mathfrak a_1)$.
\end{lem}
\begin{proof}
We first prove \ref{it:DeltaPsi} for upper choices of signs. Observe that (\ref{eq:deftp}) equivalently reads
$$\Kbsf+{1,m}(z) = -(q-q^{-1}) \Kbsf+{1,0}(zq^{-2m}) \tbsf+{1,m}(z)\,,$$
for every $m\in\N^\times$. For every $m\in \N^\times$, let 
$$\KK\pm{1,\pm m}(z) = \widehat\Psi^{-1}(\Kbsf\pm{1,\pm m}(z))\in\widehat{\qaff(\dot{\mathfrak a}_1)}[[z,z^{-1}]] \,.$$
In \cite{MZ} -- see proposition-definition 4.9, definition 4.25 and eq. (4.66) --, we proved that $\KK{+}{1,0}(z) = - \km{1}(C^{1/2}z)$ and that, for every $m\in\N^\times$, $$\KK{+}{1,m}(z)=(q-q^{-1})\km{1}(C^{1/2}zq^{-2m})\bpsi{+}{1,m}(z)\,,$$
where $\bpsi{+}{1,m}(z)$ can be recursively defined by setting
\be\label{eq:init}\left [\xp0(w) , \xp1(z) \right ]_{G_{10}^-(w/z)} = \delta \left ( \frac{q^2 w}{z} \right ) \bpsi +{1,1}(z) \ee
and
\bea\label{eq:psipsimm}{}_{G^-_{01}(q^{-2m}v/w)G^-_{11}(q^{2(1-m)}v/w)} \left [ \bpsi +{1,1}(w), \bpsi +{1,m}(v) \right ]_{G^-_{01}(w/vq^2)G^-_{11}(w/v)} &=&
	[2]_q\delta\left(\frac{w}{vq^2}\right) \bpsi+{1,m+1}(q^2v) \nn\\&& -[2]_q\delta\left(\frac{q^{2m}w}{v}\right) \bpsi+{1,m+1}(v)
	\,.\eea
Hence, \ref{it:DeltaPsi} for $m=0$ is clear. From (\ref{eq:init}) and definition \ref{eq:Hopfalgqaffdot}, and making use of relations (\ref{eq:kpxpm}) and (\ref{eq:kmxpm}) as well as of the identity (\ref{eq:G+G-}), we deduce that
$$\Delta(\bpsi +{1,1}(z)) = \bpsi +{1,1}(z)\otimes 1 + \bwp-(zq^{-2}C_{(1)}^{1/2}) \widehat\otimes \bpsi +{1,1}(zC_{(1)})- [2]_q (q-q^{-1}) \kk-1(zC_{(1)}^{1/2}) \x+0(zq^{-2}) \widehat\otimes \x+1(zC_{(1)})\,,$$
where $\bwp-(v) = \kk-0(v)\kk-1(vq^2)$. Applying $\widehat\Psi\widehat\otimes\widehat\Psi$ to the first two terms obviously yields $\Delta^0(\tbsf+{1,1}(z))$. Since, on the other hand, $\widehat\Psi(\x+0(z))\in \qdaff^<(\mathfrak a_1)[[z,z^{-1}]]$ -- see theorem \ref{thm:main} --, applying $\widehat\Psi\widehat\otimes\widehat\Psi$ to the third term yields an element of $\qdaff^<(\mathfrak a_1)\widehat\otimes \qdaff^>(\mathfrak a_1)[[z,z^{-1}]]$ and it follows that \ref{it:DeltaPsi} holds for $m=1$ and for upper choices of signs. Suppose it holds for upper choices of signs and for some $m\in\N^\times$. Then, making use of (\ref{eq:psipsimm}), one easily checks that \ref{it:DeltaPsi} holds for $m+1$ and for upper choices of signs, which completes the proof of \ref{it:DeltaPsi}  for upper choices of signs. Now, \ref{it:DeltaPsi} for lower choices of signs follows after applying $\dot\varphi$ and observing that, indeed,
 $$\dot\Delta\circ\dot\varphi  = \left (\dot\varphi\widehat\otimes\dot\varphi\right )\circ \dot\Delta^{\mathrm{cop}} \,,\qquad \mbox{and}\qquad \Delta^0\circ\dot\varphi_{|\qdaff^0(\mathfrak a_1)}  = \left (\dot\varphi\widehat\otimes\dot\varphi\right )_{|\qdaff^0(\mathfrak a_1)}\circ \Delta^{0,\mathrm{cop}}\,.$$
As for \ref{it:DeltaXp}, we let, for every $r\in\Z$,
$$\X+{1,r}(z) = \widehat\Psi^{-1}(\Xbsf+{1,r}(z))\,.$$
In \cite{MZ} -- see definition 4.1 and proposition 4.8 --, we proved that $\X+{1,r}(z)$ could be defined recursively by setting $\X+{1,0}(z) = \x+1(z)$ and letting, for every $r\in\N$,
\be\label{eq:indXpup}\left [ \bpsi +{1,1}(z), \Xp{1,r}(v)  \right ]_{G_{10}^-(z/vq^2)G_{11}^-( z/v)} = [2]_q \delta \left ( \frac{z}{vq^2} \right ) \Xp{1,r+1}(z)\ee
and
\be\label{eq:indXpdown}\left [ \bpsi-{1,-1}(z), \Xp{1,-r}(v)  \right ] =  [2]_q \delta \left ( \frac{Cz}{v} \right )  \Xp{1,-(r+1)}(Cq^{-2}z) \bwp +(C^{1/2}q^{-2}z)\,,\ee
where $\bpsi-{1,-1}(z) = \varphi(\bpsi+{1,1}(1/z))$ -- see proposition 4.3 in \cite{MZ}. Observing that $\left (\varphi\widehat\otimes\varphi\right ) \circ \Delta^{\mathrm{cop}} = \Delta\circ \varphi$, we clearly get
$$\left (\widehat\Psi\widehat\otimes\widehat\Psi\right)\circ \Delta(\bpsi-{1,-1}(z)) =\Delta^0(\tbsf-{1,-1}(z)) \mod \qdaff^<(\mathfrak a_1)\widehat\otimes \qdaff^>(\mathfrak a_1)[[z, z^{-1}]]\,.$$
Now, applying $\widehat\Psi\widehat\otimes\widehat\Psi$ to (\ref{eq:coprodxip}) in definition \ref{eq:Hopfalgqaffdot} clearly proves \ref{it:DeltaXp} in the case $r=0$. Assuming it holds for $r\in \N$, it suffices to apply $\left (\widehat\Psi\widehat\otimes\widehat\Psi\right )\circ \Delta$ to (\ref{eq:indXpup}) above to prove that it also holds for $r+1$. Similarly, if \ref{it:DeltaXp} holds for some $r\in-\N$, applying $ \left (\widehat\Psi\widehat\otimes\widehat\Psi\right )\circ \Delta$ to (\ref{eq:indXpdown})  to prove that it also holds for $r-1$. This concludes the proof.
\end{proof}

\section{$t$-weight $\qdaff(\mathfrak a_1)$-modules}
\label{Sec:tweight}
\subsection{$\ell$-weight modules over $\qdaff^{0}(\mathfrak a_1)$}
Remember that $\qdaff^{0,0}(\mathfrak a_1)$ contains a subalgebra that is isomorphic to $\qaff^0(\mathfrak a_1)$ -- see proposition \ref{prop:uq0subalg}. Hence, in view of remark \ref{rem:ellweight}, we can repeat for modules over $\qdaff^0(\mathfrak a_1)$ what we did in section \ref{sec:wfqaff} for modules over $\qaff(\mathfrak a_1)$. We thus make the following
\begin{defn}\label{def:qdaff0ellweight}
We shall say that a (topological) $\qdaff^0(\mathfrak a_1)$-module $M$ is \emph{$\ell$-weight} if there exists a countable set $\left\{ M_\alpha: \alpha\in A\right \}$ of indecomposable locally finite-dimensional $\qdaff^{0,0}(\mathfrak a_1)$-modules called \emph{$\ell$-weight spaces} of $M$, such that, as $\qdaff^{0,0}(\mathfrak a_1)$-modules,
$$M\cong \bigoplus_{\alpha\in A} M_\alpha\,.$$
\end{defn}
As in section \ref{sec:wfqaff}, it follows that
\begin{defprop}
\label{defprop:CKeigen}
Let $M$ be an $\ell$-weight $\qdaff^0(\mathfrak a_1)$-module. Then:
\begin{enumerate}[i.]
\item $\Csf^2$ acts on $M$ by $\id$;
\item for every $\ell$-weight space $M_\alpha$, $\alpha\in A$, of $M$, there exist $\kappa_{\alpha, 0}\in\F^\times$ and sequences $(\kappa^\pm_{\alpha,\pm m})_{m\in\N^\times} \in \F^{\N^{\times}}$ such that
\be\label{eq:Jordanlweight}M_\alpha \subseteq \left \{v\in M : \exists n\in\N^\times\,, \forall m\in\N \quad  \left (\Ksf\pm{1,0,\pm m}- \kappa^\pm_{\alpha, \pm m} \id \right )^n .v=0\right \}\,,\ee
where we have set $\kappa_{\alpha,0}^\pm = \kappa_{\alpha,0}^{\pm 1}$. 
\end{enumerate}
We let $\Sp(M) =\{\kappa_{\alpha, 0} : \alpha \in A\}$ and we shall refer to
$$\kappa_\alpha^\pm(z) = \sum_{m\in\N} \kappa_{\alpha, \pm m}^\pm z^{\pm m}$$
as the \emph{$\ell$-weight} of the $\ell$-weight space $M_\alpha$. We shall say that $M$ is
\begin{itemize}
\item of \emph{type $1$} if $\Csf^{1/2}$ acts by $\id$ over $M$;
\item of \emph{type $(1,N)$} for $N\in\N^\times$ if it is of type $1$ and, for every $m\geq N$, $\csf\pm{\pm m}$ acts by multiplication $0$ over $M$;
\item of \emph{type $(1,0)$} if it is of type $(1,1)$ and $\csf\pm0$ acts by $\id$ over $M$.
\end{itemize} 
\end{defprop}
\begin{proof}
The proof follows the same arguments as the proof of definition-proposition \ref{defprop:ellweight}. 
\end{proof}

\begin{prop}
\label{prop:Kellweight}
Let $M$ be a type $1$ $\ell$-weight $\qdaff^0(\mathfrak a_1)$-module and let $M_\alpha$ and $M_\beta$ be two $\ell$-weight spaces of $M$ such that, for some $m\in\N^\times$ and some $n\in\Z$, $M_\alpha \cap \Ksf\pm{1,\pm m,n}. M_\beta \neq \{0\}$. Then, there exists a unique $a\in\F^\times$ such that:
\begin{enumerate}[i.]
\item\label{Kelweighti} the respective $\ell$-weights $\kappa_\alpha^\varepsilon(z)$ and $\kappa_\beta^\varepsilon(z)$ of $M_\alpha$ and $M_\beta$ be related by
$$\kappa_\alpha^\varepsilon(z) = \kappa_\beta^\varepsilon (z) H_{m, a}^\varepsilon(z)^{\pm 1} \,,$$
where $\varepsilon \in \{-,+\}$ and 
\be\label{eq:ellweightsnullroot} H_{m, a}^\pm (z) = \left (\frac{(1-q^{-2}a/z)(1-q^{-2(m-1)}a/z)}{(1-q^2a/z)(1-q^{-2(m+1)} a/z)}\right )_{|z|^{\pm 1}\ll 1}\,;\ee
\item\label{Kelweightii} $(z-a)^N M_\alpha \cap \Kbsf\pm{1,\pm m}(z). M_\beta = \{0\}$ for some $N\in\N^\times$.
\end{enumerate}
\end{prop}

\begin{proof}
There clearly exist two bases $\left \{v_i: i = 1, \dots, \dim M_\alpha \right \}$ and $\left \{w_i: i = 1, \dots, \dim M_\beta\right \}$ of $M_\alpha$ and $M_\beta$ respectively, in which
$$\forall i \in \rran{\dim M_\alpha}\,, \qquad\qquad \Kbsf\pm{1,0}(z)  . v_i = \kappa_\alpha^\pm(z) \sum_{k=i}^{\dim M_\alpha} \eta_{\alpha,i,k}^\pm(z) v_k \,,$$
$$\forall j \in \rran{\dim M_\beta}\,, \qquad\qquad \Kbsf\pm{1,0}(z) . w_j = \kappa_\beta^\pm(z) \sum_{l=j}^{\dim M_\beta} \eta_{\beta,j,l}^\pm(z) w_l \,,$$
for some $\eta_{\alpha,i,k}^\pm(z), \eta_{\beta,j,l}^\pm(z) \in \F[[z^{\pm 1}]]$, with $i,k\in\rran{\dim M_\alpha}$ and $j,l\in\rran{\dim M_\beta}$, such that $\eta_{\alpha,i,i}^\pm(z) =1$ for every $i\in\rran{\dim M_\alpha}$ and $\eta_{\beta,j,j}^\pm(z)=1$ for every $j\in\rran{\dim M_\beta}$.

Now, if $M_\alpha \cap \Ksf\pm{1,\pm m,n}. M_\beta \neq \{0\}$, there must exist a largest nonempty subset $J\subseteq\rran{\dim M_\beta}$ such that, for every $j\in J$, $M_\alpha \cap \Kbsf\pm{1,\pm m}(z).w_{j} \neq\{0\}$. Let $j_* = \max J$.  Obviously, for every $j\in J$, there must exist a largest nonempty subset $I(j)\subseteq \rran{\dim M_\alpha}$ such that, for every $j\in J$ and every $i\in I(j)$, $\F v_{i}\cap  \Kbsf\pm{1,\pm m}(z).w_{j} \neq\{0\}$. Let $i_*(j) = \min I(j)$ and let for simplicity $i_*=i_*(j_*)$. Then, for every $j\in J$,
$$M_\alpha\cap\Kbsf\pm{1,\pm m}(z).w_j = \sum_{i\in I(j)} \xi_{m,j,i}^\pm(z) v_i \,,$$ 
for some $\xi_{m, j, i}^\pm(z)\in \F[[z,z^{-1}]]-\{0\}$. When needed, we shall extend by zero the definition of $\xi_{m, j, i}^\pm(z)$ outside of the set of pairs $\{(j,i): j\in J, i\in I(j)\}$. Making use of the relations in $\qdaff(\mathfrak a_1)$ -- namely (\ref{eq:K+K+}) and (\ref{eq:K+K-}) --, we get, for $\varepsilon \in\{-,+\}$,
$$(z_1-q^{\pm 2}z_2)(z_1- q^{2(m\mp 1)}z_2) \Kbsf\pm{1,\pm m}(z_1) \Kbsf\varepsilon{1, 0}(z_2)  . w_j  = (z_1q^{\pm 2} -z_2)(z_1q^{\mp 2} -q^{2m}z_2)\Kbsf\varepsilon{1, 0}(z_2)  \Kbsf\pm{1,\pm m}(z_1) . w_j \,.$$
The latter easily implies that, for every $j\in J$ and every $i\in I(j)$,
\bea\label{eq:Kxideltaeq}&&(z_1-q^{\pm 2} z_2) (z_1- q^{2(m\mp 1)}z_2) \kappa_\beta^\varepsilon(z_2) \sum_{\substack{l\in J\\l\geq j}} \eta_{\beta,j,l}^\varepsilon (z_2) \xi_{m,l,i}^\pm(z_1) \nn\\
&&\quad\qquad\qquad\qquad\qquad\qquad\qquad = (z_1q^{\pm 2} -z_2)(z_1q^{\mp 2} -q^{2m}z_2) \kappa_\alpha^\pm(z_2) \sum_{\substack{k\in I(j)\\k\leq i}} \eta_{\alpha,k,i}^\varepsilon(z_2) \xi_{m,j,k}^\pm(z_1)\,. \eea
Taking $i=i_*$ and $j=j_*$ in the above equation immediately yields
$$\left [ (z_1-q^{\pm 2}z_2)(z_1- q^{2(m\mp 1)}z_2) \kappa_\beta^\varepsilon(z_2)-  (z_1q^{\pm 2} -z_2)(z_1q^{\mp 2} -q^{2m}z_2)\kappa_\alpha^\varepsilon(z_2)\right ] \xi_{m,j_*,i_*}^\pm(z_1)  =0\,.$$
The latter is equivalent to the fact that, for every $p\in\Z$,
\bea\label{eq:Hma} 
&& \left (\xi_{m,j_*,i_*, p}^\pm q^{2m} z^2+ \xi_{m,j_*,i_*, p+2}^\pm \right ) \left [\kappa_\beta^\varepsilon(z) - \kappa_\alpha^\varepsilon(z)\right ] \nn\\
&& \qquad\qquad\qquad\qquad\qquad\qquad = \xi_{m,j_*,i_*, p+1}^\pm z \left [(q^{2(m\mp 1)}+q^{\pm 2})\kappa_\beta^\varepsilon(z) - (q^{2(m\pm 1)} +q^{\mp 2})\kappa_\alpha^\varepsilon(z) \right ] \,.
\eea
where, as usual, we have set
$$\xi_{m,j_*,i_*, p}^\pm  = \res_z z^{p-1} \xi_{m,j_*,i_*}^\pm(z)\,.$$
Since $\xi_{m,j_*,i_*}^\pm(z)\neq0$, there must exist a $p\in\Z$ such that $\xi_{m,j_*,i_*, p}^\pm\neq 0$. Assuming that $\xi_{m,j_*,i_*, p+1}^\pm =0$, one easily obtains that, on one hand $\kappa_\beta^\varepsilon(z) =\kappa_\alpha^\varepsilon(z)$ and that, on the other hand,
$$\left [(q^{2(m\mp 1)}+q^{\pm 2})\kappa_\beta^\varepsilon(z) - (q^{2(m\pm 1)} +q^{\mp 2})\kappa_\alpha^\varepsilon(z) \right ] =0\,.$$
A contradiction. By similar arguments, one eventually proves that $\xi_{m,j_*,i_*, p}^\pm \neq 0$ for every $p\in\Z$. But then dividing (\ref{eq:Hma}) by $\xi_{m,j_*,i_*, p}^\pm$ we get
$$\left (q^{2m} z^2+ a^2 \right ) \left [\kappa_\beta^\varepsilon(z) - \kappa_\alpha^\varepsilon(z)\right ] - a z \left [(q^{2(m\mp 1)}+q^{\pm 2})\kappa_\beta^\varepsilon(z) - (q^{2(m\pm 1)} +q^{\mp 2})\kappa_\alpha^\varepsilon(z) \right ] =0\,.$$ 
where we have set, for every $p\in\Z$, $\xi_{m,j_*,i_*, p+1}^\pm/\xi_{m,j_*,i_*, p}^\pm=a\in\F^\times$ and, consequently, $\xi_{m,j_*,i_*, p+2}^\pm/\xi_{m,j_*,i_*, p}^\pm = a^2$. \ref{Kelweighti} follows. Moreover, we clearly have
$$\xi_{m,j_*,i_*}^\pm (z) = A_{m,j_*,i_*}^\pm \delta(z/a)\,,$$
for some $A_{m,j_*,i_*}^\pm \in \F^\times$. More generally, we claim that, 
\be\label{eq:Kclaim} \forall j\in J\,,\forall i\in I(j)\,,\qquad\qquad  \xi_{m,j,i}^\pm (z) = \sum_{p=0}^{N(i,j)} A_{m,j,i,p}^\pm \delta^{(p)}(z/a) \,,\ee
for some $A_{m,j,i,p}^\pm \in\F$ and some $N(i,j)\in \N$. This is proven by a finite induction on $j$ and $i$. Indeed, making use of (\ref{eq:ellweightsnullroot}), we can rewrite (\ref{eq:Kxideltaeq}) as
\bea\label{eq:Kxideltasimp}&&(z_1-q^{\pm 2} z_2) (z_1- q^{2(m\mp 1)}z_2)(z_2-q^{\pm 2} a) (z_2 - q^{-2(m\pm 1)}a)  \sum_{\substack{l\in J\\l\geq j}} \eta_{\beta,j,l}^\varepsilon (z_2) \xi_{m,l,i}^\pm(z_1) \nn\\
&&\quad\qquad = (z_1q^{\pm 2} -z_2)(z_1q^{\mp 2} -q^{2m}z_2)(z_2 - q^{\mp 2}a)(z_2- q^{-2(m\mp 1)}a) \sum_{\substack{k\in I(j)\\k\leq i}} \eta_{\alpha,k,i}^\varepsilon(z_2) \xi_{m,j,k}^\pm(z_1)\,, \eea
for every $j\in J$ and every $i\in I(j)$. Now, assume that (\ref{eq:Kclaim}) holds for every pair in 
$$\left \{(j,i): j\in J, \, i\in I(j), \quad j> j_0 \right \} \cup \left \{(j_0,i): i\in I(j_0),\quad  i\leq i_0 \right \}\,,$$
 for some $j_0\in J$ and some $i_0\in I(j_0)$ such that $i_0<\max I(j_0)$. Let $i_0'$ be the smallest element of $I(j_0)$ such that $i_0<i_0'$. It suffices to write (\ref{eq:Kxideltasimp}) for $j=j_0$ and $i=i_0'$, to get
\bea &&(z_1-a)z_2(z_1 a- q^{2m} z_2^2) (q^{\mp 2} + q^{-2(m\mp 1)} -q^{\pm 2} -q^{-2(m\pm 1)}) \xi_{m,j_0,i_0'}^\pm(z_1) \nn\\
&&\qquad = - (z_1-q^{\pm 2} z_2) (z_1- q^{2(m\mp 1)}z_2)(z_2-q^{\pm 2} a) (z_2 - q^{-2(m\pm 1)}a)  \sum_{\substack{l\in J\\l> j_0}} \eta_{\beta,j_0,l}^\varepsilon (z_2) \xi_{m,l,i_0'}^\pm(z_1) \nn\\
&&\qquad + (z_1q^{\pm 2} -z_2)(z_1q^{\mp 2} -q^{2m}z_2)(z_2 - q^{\mp 2}a)(z_2- q^{-2(m\mp 1)}a)  \sum_{\substack{k\in I(j_0)\\k\leq i_0}} \eta_{\alpha,k,i_0'}^\varepsilon(z_2) \xi_{m,j_0,k}^\pm(z_1)\,.\eea
 Combining the recursion hypothesis and lemma \ref{lem:deltaid} from the appendix, one easily concludes that (\ref{eq:Kclaim}) holds for the pair $(j_0, i_0')$. Repeating the argument finitely many times, we get that it actually holds for all the pairs in $\left \{(j,i): j\in J, \, i\in I(j), \quad j\geq j_0 \right \}$. Now, either $j_0 = \min J$ and we are done; or $j_0>\min J$ and there exists a largest $j_0'\in J$ such that $j_0>j_0'$. Writing (\ref{eq:Kxideltasimp}) for $j=j_0'$ and $i=i_*(j_0')$, we get
\bea && (z_1-a)z_2(z_1 a- q^{2m} z_2^2) (q^{\mp 2} + q^{-2(m\mp 1)} -q^{\pm 2} -q^{-2(m\pm 1)}) \xi_{m,j_0',i_*(j_0')}^\pm(z_1)\nn\\
&& \qquad  \qquad = -(z_1-q^{\pm 2} z_2) (z_1- q^{2(m\mp 1)}z_2)(z_2-q^{\pm 2} a) (z_2 - q^{-2(m\pm 1)}a) \sum_{\substack{l\in J\\l\geq j_0}} \eta_{\beta,j_0',l}^\varepsilon (z_2) \xi_{m,l,i_*(j_0')}^\pm(z_1) \,.\nn\eea
Combining again the recursion hypothesis and lemma \ref{lem:deltaid}, we easily get that (\ref{eq:claim}) holds for $(j_0', i_*(j_0'))$. It is now clear that the claim holds for every $j\in J$ and every $i\in I(j)$. Letting $N=\max \{N(i,j): j\in J, \, i\in I(j)\}$, \ref{Kelweightii} follows. Furthermore, for every $b\in\F-\{a\}$ and every $n\in\N$, we obviously have $(z-b)^n M_\alpha \cap \Kbsf\pm{1,\pm m}(z).M_\beta \neq \{0\}$, thus making $a$ the only element of $\F$ satisfying \ref{Kelweightii}. This concludes the proof.
\end{proof}

\noi We let $\omega_1$ denote the fundamental weight of $\mathfrak a_1$ and we let $P=\Z \omega_1$ be the corresponding weight lattice. In view of proposition \ref{prop:Kellweight}, it is natural to make the following
\begin{defn}
\label{def:dominant}
Let $M$ be a type $(1,0)$ $\ell$-weight $\qdaff^0(\mathfrak a_1)$-module and let $\left \{M_\alpha : \alpha\in A\right \}$ be the countable set of its $\ell$-weight spaces. We shall say that $M$ is \emph{rational} if, for every $\alpha \in A$, there exist relatively prime monic polynomials $P_\alpha(1/z), Q_\alpha(1/z) \in \F[z^{-1}]$, called Drinfel'd polynomials of $M$, such that 
the $\ell$-weight $\kappa_\alpha^\pm(z)$ of $M_\alpha$ be given by
$$\kappa_\alpha^\pm(z) = -q^{\deg(P_\alpha) - \deg(Q_\alpha)} \left (\frac{P_\alpha(q^{-2}/z) Q_\alpha(1/z)}{P_\alpha(1/z) Q_\alpha(q^{-2}/z)}\right )_{|z|^{\pm 1} \ll 1}\,.$$
With each rational $\ell$-weight $\kappa_\alpha^\pm(z)$ of a rational $\qdaff^0(\mathfrak a_1)$-module $M$, we associate an integral weight $\lambda_\alpha\in P$, by setting 
$$\lambda_\alpha =\left [ \deg(P_\alpha) - \deg (Q_\alpha) \right ] \omega_1\,.$$ 
We shall say that $M$ is \emph{$\ell$-dominant} (resp. \emph{$\ell$-anti-dominant}) if it is rational and there exists $N\in\N^\times$ such that, for every $\alpha\in A$, $\deg (P_\alpha)= N$ and $\deg (Q_\alpha)=0$ (resp. $\deg (P_\alpha)= 0$ and $\deg (Q_\alpha)=N$).
\end{defn}
\begin{rem}
The classical weight $N\omega_1$ (resp. $-N\omega_1$) associated with any $\ell$-dominant (resp. $\ell$-anti-dominant) type $1$ $\ell$-weight rational $\qdaff^0(\mathfrak a_1)$-module $M$ is a dominant (resp. anti-dominant) integral weight. Note that the converse need not be true.
\end{rem}

\begin{rem}
\label{rem:qchar}
The data of the $\ell$-weights of a rational $\qdaff^0(\mathfrak a_1)$-module is equivalent to the data of its Drinfel'd polynomials $\left \{(P_\alpha, Q_\alpha):\alpha\in A\right \}$ which, in turn, is equivalent to the data of their finite multisets of roots $\left\{ (\nu_\alpha^+, \nu_\alpha^-):\alpha\in A\right \}$. The latter are finitely supported maps $\nu_\alpha^\pm : \F^\times \to \N$ such that, for every $\alpha\in A$,
$$P_\alpha (1/z) = \prod_{x\in\F^\times} \left (1-x/z\right )^{\nu_\alpha^+(x)} \qquad \mbox{and} \qquad Q_\alpha (1/z) = \prod_{x\in \F^\times} (1-x/z)^{\nu_\alpha^-(x)}\,.$$
Note that, in the above formulae, since $\nu_\alpha^\pm$ is finitely supported, the products only run through the finitely many numbers in the support $\supp(\nu_\alpha^\pm)$ of $\nu_\alpha^\pm$. Moreover, since $P_\alpha$ and $Q_\alpha$ are relatively prime for every $\alpha\in A$, we have $\supp(\nu_\alpha^+)\cap \supp(\nu_\alpha^-)=\emptyset$. We denote by $\N_f^{\F^\times}$, the set of finitely supported $\N$-valued maps over $\F^\times$. As is customary in the theory of $q$-characters, we associate with every $\ell$-weight given by a pair $(P_\alpha, Q_\alpha)$ of Drinfel'd polynomials or, equivalently, by a pair $(\nu_\alpha^+, \nu_\alpha^-)$ with $\nu_\alpha^+, \nu_\alpha^-\in\N_f^{\F^\times}$ and $\supp(\nu_\alpha^+)\cap \supp(\nu_\alpha^-)=\emptyset$,  a monomial
$$m_\alpha  = Y^{\nu^+-\nu^-}= \prod_{x\in\F^\times} Y_{x}^{\nu_\alpha^+(x) - \nu_\alpha^-(x)}\in \Z[Y_a, Y_a^{-1}]_{a\in\F^\times}\,.$$
\end{rem}

\begin{defn}\label{def:tdom}
Let $M$ be an $\ell$-dominant $\qdaff^0(\mathfrak a_1)$-module and let $M_\alpha$ and $M_\beta$ be any two $\ell$-weight spaces of $M$ with respective $\ell$-weights 
$$\kappa_\alpha^\pm(z) = -q^{\deg(P_\alpha)} \left (\frac{P_\alpha(q^{-2}/z) }{P_\alpha(1/z) }\right )_{|z|^{\pm 1} \ll 1}\qquad\qquad \mbox{and}\qquad \qquad 
\kappa_\beta^\pm(z) = -q^{\deg(P_\beta)} \left (\frac{P_\beta(q^{-2}/z) }{P_\beta(1/z) }\right )_{|z|^{\pm 1} \ll 1}\,,$$
where $P_\alpha(1/z), P_\beta(1/z) \in \F[z^{-1}]$ are two monic polynomials. By proposition \ref{prop:Kellweight}.\emph{i.}, if $M_\alpha \cap \Ksf\pm{1,\pm m}(z). M_\beta \neq \{0\}$ for some $m\in\N^\times$, then there exists a unique $a\in \F^\times$ such that 
$$\kappa_\alpha^\varepsilon(z) = \kappa_\beta^\varepsilon (z) H_{m, a}^\varepsilon(z)^{\pm 1} \,,$$
where $\varepsilon \in \{-,+\}$. We shall say that $M$ is \emph{$t$-dominant} if, under the same assumptions, we have, in addition, that 
$$P_\beta(1/aq^{-(m\pm m)}) = P_\beta(1/aq^{2-(m\pm m)}) =0\,.$$
\end{defn}

For every $a\in\F^\times$, we let  $\delta_a \in  \N_f^{\F^\times}$ be defined by 
$$\delta_a(x) = \begin{cases}
1 & \mbox{if $x=a$;}\\
0 &\mbox{otherwise.}
\end{cases}$$
For every $a\in\F^\times$, we let $\N_a^{\F^\times} = \left \{\nu \in \N_f^{\F^\times} : \{a, aq^2\} \subseteq \supp(\nu) \right \}$ and we define, for every $m\in\Z$, an operator $\Gamma_{m,a} : \N_{aq^{-2m}}^{\F^\times}\to \N_{a}^{\F^\times}$ by letting~\footnote{Although the definition of $\Gamma_{\pm 1,a}^{\pm 1}$ easily extends to $\left \{ \nu\in\N_{f}^{\F^\times} : aq^{-(1\pm 1)} \in \supp(\nu) \right \}$, we will not make use of that extension and exclusively regard $\Gamma_{\pm 1,a}^{\pm 1}$ as a map  $\N_{aq^{-2(1\pm 1)}}^{\F^\times}\to \N_{aq^{2(1\mp 1)}}^{\F^\times}$.}, for every $\nu\in\N_{aq^{-2m}}^{\F^\times}$,
$$\Gamma_{m,a}(\nu) = \nu - \delta_{aq^{-2m}}- \delta_{aq^{2-2m}} + \delta_{a} + \delta_{aq^2}  \,. $$
$\Gamma_{m,a}$ is obviously invertible, with inverse $\Gamma_{m,a}^{-1}:  \N_{a}^{\F^\times}\to \N_{aq^{-2m}}^{\F^\times}$ given by $\Gamma_{m,a}^{-1} = \Gamma_{-m,aq^{-2m}}$. Note that, for every $a\in\F^\times$, $\Gamma_{0,a} = \id$ over $\N_a^{\F^\times}$. Given two finite multisets $\nu, \nu'\in \N_f^{\F^\times}$, we we shall say that they are \emph{equivalent} and write $\nu\sim\nu'$ iff
\be\label{eq:multisetequiv}\nu = \Gamma_{m_1, a_1} \circ \dots \circ \Gamma_{m_n, a_n} (\nu')\,,\ee
for some $n\in\N$, $m_1, \dots, m_n\in\Z^n$ and some $a_1, \dots, a_n\in\F^\times$. In writing (\ref{eq:multisetequiv}), it is assumed that, for every $p= 2, \dots, n$, $\Gamma_{m_p, a_p} \circ \dots \circ \Gamma_{m_n, a_n} (\nu') \in \N_{a_{p-1}q^{-2m_{p-1}}}^{\F^\times}$. 
It is clear that $\sim$ is an equivalence relation and we denote by $[\nu]\in\N_f^{\F^\times}/\sim$ the equivalence class of $\nu$ in $\N_f^{\F^\times}$. Following remark \ref{rem:qchar}, we naturally extend the action of $\Gamma_{m,a}$ to $\Z[Y_b,Y_b^{-1}]_{b\in\F^\times}$, by setting
$$\Gamma_{m,a} (Y^\nu) = Y^{\Gamma_{m,a}(\nu)}\,.$$
The equivalence relation $\sim$ similarly extends from $\N_f^{\F^\times}$ to $\Z[Y_b,Y_b^{-1}]_{b\in\F^\times}$. Note that, setting 
$$H_{m,a} = Y_{aq^{-2m}}^{-1} Y_{aq^{2-2m}}^{-1} Y_a Y_{aq^2}\in \Z[Y_b,Y_b^{-1}]_{b\in \F^\times}\,,$$
for every $a\in\F^\times$ and every $m\in\Z$, we have, for every $\nu\in\N_a^{\F^\times}$
$$\Gamma_{m,a}(Y^\nu) = H_{m,a} Y^\nu\,.$$

\begin{cor}
\label{cor:Ynu}
Let $M$ be a simple $t$-dominant 
$\qdaff^0(\mathfrak a_1)$-module. Then there exists a multiset $\nu\in\N_f^{\F^\times}$ such that all the monomials associated with the $\ell$-weights of $M$ be in the equivalence class of $Y^\nu$. 
\end{cor}

\begin{proof}
By proposition \ref{prop:Kellweight}, for any two $\ell$-weight spaces, $M_\alpha$ and $M_\beta$, of an $\ell$-dominant $\qdaff^0(\mathfrak a_1)$-module $M$, with respective $\ell$-weights 
$$\kappa_\alpha^\pm(z) = -q^{\deg(P_\alpha)} \left (\frac{P_\alpha(q^{-2}/z) }{P_\alpha(1/z) }\right )_{|z|^{\pm 1} \ll 1}\qquad\qquad \mbox{and}\qquad \qquad \kappa_\beta^\pm(z) = -q^{\deg(P_\beta)} \left (\frac{P_\beta(q^{-2}/z) }{P_\beta(1/z) }\right )_{|z|^{\pm 1} \ll 1}\,,$$
if $M_\alpha \cap \Ksf\pm{1,\pm m ,n}.M_\beta \neq \{0\}$ for some $m\in\N^\times$ and some $n\in\Z$, then we must have
\be\label{eq:PH}\frac{P_\alpha(q^{-2}/z) }{P_\alpha(1/z) }=  \frac{P_\beta(q^{-2}/z) }{P_\beta(1/z) }\left ( \frac{(1-q^{-2}a/z)(1-a/z)}{(1-a/z)(1-q^2a/z)}\frac{(1-q^{-2m}a/z)(1-q^{-2(m-1)}a/z)}{(1-q^{-2(m+1)} a/z) (1-q^{-2m}a/z)}\right )^{\pm 1}\,,\ee
for some $a\in\F^\times$. Now, assuming $m>1$, it is clear that:
\begin{itemize}
\item[-] for the upper choice of sign on the right hand side of the above equation, the last fraction line must completely cancel against factors in the first one, whereas the second one survives, eventually replacing the cancelled factors;
\item[-] for the lower choice of sign, the second fraction line must cancel against factors in the first one, whereas the last one survives, eventually replacing the cancelled factors.
\end{itemize}
If on the other hand $m=1$, since $M$ is $t$-dominant, we have, by definition, that $aq^{\mp 2}$ is a root of $P_\beta(1/z)$. In any case, denoting by $\nu_\alpha$ (resp. $\nu_\beta$) the multiset of roots of $P_\alpha(1/z)$ (resp. $P_\beta(1/z)$), it is clear that $\nu_\alpha \sim\nu_\beta$ and hence $Y^{\nu_\alpha} \sim Y^{\nu_\beta}$. Since $M$ is simple, there can be no non-zero $\ell$-weight space $M_\beta$ of $M$ such that $M_\alpha \cap  \Ksf\pm{1,\pm m, n}.M_\beta =\{0\}$ for every $\ell$-weight space $M_\alpha$ of $M$, every $m\in\N^\times$ and every $n\in\Z$. 
\end{proof}

\noi In view of definition-proposition \ref{defprop:CKeigen}, we can make the following
\begin{defn}
For every monic polynomial $P(1/z)\in \F[z^{-1}]$, denote by $\F_P$ the one-dimensional $\qdaff^{0,0}(\mathfrak a_1)$-module such that
$$\Kbsf\pm{1,0}(z) .v = -q^{\deg(P)} \left (\frac{P(q^{-2}/z)}{P(1/z)}\right )_{|z|^{\pm 1} \ll 1} v\,,$$
for every $v\in \F_P$. There exists a universal $\qdaff^0(\mathfrak a_1)$-module $M^0(P) \cong \qdaff^{0}(\mathfrak a_1) \underset{\qdaff^{0,0}(\mathfrak a_1)}{\widehat\otimes} \F_P$ that admits the $\ell$-weight associated with $P$. Denoting by $N^0(P)$ the maximal $\qdaff^0(\mathfrak a_1)$-submodule of $M^0(P)$ such that $N^0(P)\cap \F_P= \{0\}$, we define the unique -- up to isomorphisms -- simple $\qdaff^0(\mathfrak a_1)$-module $L^0(P)= M^0(P)/N^0(P)$.
\end{defn}
\begin{prop}
\label{prop:simpleelldomuq0mods}
For every simple $\ell$-dominant $\qdaff^0(\mathfrak a_1)$-module $M$, there exists a monic polynomial $P(1/z)\in \F[z^{-1}]$ such that $M\cong L^0(P)$.
\end{prop}
\begin{proof}
Obviously, for every $v\in M-\{0\}$, we have $M\cong \qdaff^0(\mathfrak a_1).v$. Now since $M$ is $\ell$-dominant, $v$ can be chosen as an $\ell$-weight vector, \ie
$$\Kbsf\pm{1,0}(z). v = -q^{\deg(P)} \left (\frac{P(q^{-2}/z)}{P(1/z)}\right)_{|z|^{\pm 1}\ll 1} v$$
for some monic polynomial $P(1/z)\in\F[z^{-1}]$.
\end{proof}
\begin{rem}
The above proof makes it clear that if $\{P_\alpha:\alpha\in A\}$ is the set of Drinfel'd polynomials of a simple $\ell$-dominant $\qdaff^0(\mathfrak a_1)$-module $M$, then, for every $\alpha\in A$, $M\cong L^0(P_\alpha)$.
\end{rem}
\begin{thm}
\label{thm:elldomtdom}
For every monic polynomial $P(1/z)\in \F[z^{-1}]$, $L^0(P)$ is $t$-dominant.
\end{thm}
\begin{proof}
We postpone the proof of this theorem until section \ref{Sec:ev}, where we construct $L^0(P)$ for every $P$ and directly check that it is indeed $t$-dominant.
\end{proof}

\begin{prop}
Any topological $\widehat{\qdaff^0(\mathfrak a_1)}$-module pulls back to a module over the elliptic Hall algebra $\mathcal E_{q^{-4}, q^2,q^2}$.
\end{prop}
\begin{proof}
It suffices to make use of the Hopf algebra homomorphism 
$${\mathcal E_{q^{-4}, q^2,q^2}} \stackrel{f}{\longrightarrow} {\qdaff^{0^+}(\mathfrak a_1)} \hookrightarrow \widehat{\qdaff^0(\mathfrak a_1)}\,,$$
where $f$ is defined in proposition \ref{prop:Hall} and the second arrow is the canonical injection into $ \widehat{\qdaff^0(\mathfrak a_1)}$ of its Hopf subalgebra ${\qdaff^{0^+}(\mathfrak a_1)}$ -- see proposition \ref{def:uq0+}.
\end{proof}
\begin{rem}
It is worth mentioning that, as an example of the above proposition, $\ell$-anti-dominant $\qdaff^0(\mathfrak a_1)$-modules pullback to a family of ${\mathcal E_{q^{-4}, q^2,q^2}}$-modules that were recently introduced in \cite{diFrancescoKedem}. It might be interesting to investigate further the class of ${\mathcal E_{q^{-4}, q^2,q^2}}$-modules obtained by pulling back other (rational) $\qdaff^0(\mathfrak a_1)$-modules.
\end{rem}

We conclude the present subsection by proving the following
\begin{lem}
\label{lem:tdom}
Let $M$ be an $\ell$-dominant $\qdaff^0(\mathfrak a_1)$-module. Suppose that, for any two $\ell$-weight spaces $M_\alpha$ and $M_\beta$ of $M$, with respective $\ell$-weights $\kappa_\alpha^\pm(z)$ and $\kappa_\beta^\pm(z)$, such that $M_\alpha \cap \Kbsf\pm{1,\pm 1}(z).M_\beta \neq \{0\}$, the unique $a\in \F^\times$ such that $\kappa_\alpha^\varepsilon(z) = \kappa_\beta^\varepsilon (z) H_{1, a}^\varepsilon(z)^{\pm 1}$, for every $\varepsilon \in \{-,+\}$, and $(z-a)^N M_\alpha \cap \Kbsf\pm{1,\pm 1}(z). M_\beta = \{0\}$ for some $N\in \N^\times$ -- see proposition \ref{prop:Kellweight} -- also satisfies $P_\beta(1/a)=0$. Then $M$ is $t$-dominant.
\end{lem}
\begin{proof}
Let $M$ be as above and let $M_\alpha$ and $M_\beta$ be two $\ell$-weight spaces of $M$ with respective $\ell$-weights
$$\kappa_\alpha^\pm(z) = -q^{\deg(P_\alpha)} \left (\frac{P_\alpha(q^{-2}/z) }{P_\alpha(1/z) }\right )_{|z|^{\pm 1} \ll 1}\qquad\qquad \mbox{and}\qquad \qquad \kappa_\beta^\pm(z) = -q^{\deg(P_\beta)} \left (\frac{P_\beta(q^{-2}/z) }{P_\beta(1/z) }\right )_{|z|^{\pm 1} \ll 1}\,.$$
Suppose that $M_\alpha \cap \Kbsf\pm{1,\pm m}(z).M_\beta \neq \{0\}$ for some $m\in\N^\times$. If $m>1$, writing down $\kappa_\alpha^\varepsilon(z) = \kappa_\beta^\varepsilon (z) H_{m, a}^\varepsilon(z)^{\pm 1}$, we obtain equation (\ref{eq:PH}) as in the proof of corollary \ref{cor:Ynu}. By the same discussion as the one following equation (\ref{eq:PH}), we conclude that $P_\beta(1/aq^{-(m\pm m)}) = P_\beta(1/aq^{2-(m\pm m)}) =0$, as needed -- see definition \ref{def:tdom}. Finally, if $m=1$, writing down $\kappa_\alpha^\varepsilon(z) = \kappa_\beta^\varepsilon (z) H_{1, a}^\varepsilon(z)^{\pm 1}$, we obtain
$$\frac{P_\alpha(q^{-2}/z) }{P_\alpha(1/z) }=  \frac{P_\beta(q^{-2}/z) }{P_\beta(1/z) }\left( \frac{(1-a/z)}{(1-q^2a/z)}\frac{(1-q^{-2}a/z)}{(1-q^{-4} a/z)}\right )^{\pm 1}\,.$$
Then, it is clear that:
\begin{itemize}
\item for the upper choice of sign on the right hand side of the above equation, the last fraction line must completely cancel against factors in the first one, whereas the second one survives, eventually replacing the cancelled factors;
\item for the lower choice of sign on the right hand side of the above equation, the second fraction line must completely cancel against factors in the first one, whereas the last one survives, eventually replacing the cancelled factors.
\end{itemize}
In any case, it follows that $P_\beta(1/aq^{\mp 2})=0$. But by our assumptions on $M$, we also have that $P_\beta(1/a)=0$ and the $t$-dominance of $M$ follows -- see definition \ref{def:tdom}.
\end{proof}

\subsection{$t$-weight $\qdaff(\mathfrak a_1)$-modules}
\begin{defn}
For every $N\in\N^\times$, we shall say that a (topological) module $M$ over $\qdaff'(\mathfrak a_1)$ is of \emph{type~$(1,N)$} if:
\begin{enumerate}
\item[i.] $\Csf^{\pm 1/2}$ acts as $\id$ on $M$;
\item[ii.] $\csf\pm{\pm m}$ acts by multiplication by $0$ on $M$, for every $m\geq N$.
\end{enumerate}
We shall say that $M$ is of \emph{type~$(1,0)$} if points i. and ii. above hold for every $m>0$ and, in addition, $\csf\pm{0}$ acts as $\id$ on $M$.
\end{defn}
\begin{rem}
Let $N\in\N$. Then the $\qdaff'(\mathfrak a_1)$-modules of type $(1,N)$ are in one-to-one correspondence with the $\qdaff(\mathfrak a_1)^{(N)}/(\Csf^{1/2}-1)$-modules -- see section \ref{rem:qdafftopology} for a definition of $\qdaff(\mathfrak a_1)^{(N)}$. Obviously $\qdaff(\mathfrak a_1)$-modules of type $(1,0)$ descend to modules over the double quantum loop algebra of type $\mathfrak a_1$, $\qdloop(\mathfrak a_1)$.
\end{rem}

\begin{defn}
We shall say that a (topological) $\qdaff(\mathfrak a_1)$-module $M$ is a \emph{$t$-weight} module if there exists a countable set $\left\{M_\alpha : \alpha\in A\right \}$ of indecomposable $\ell$-weight $\uqhhzeroslt$-modules, called \emph{$t$-weight spaces} of $M$, such that, as (topological) $\uqhhzeroslt$-modules,
\be M \cong \bigoplus_{\alpha\in A} M_\alpha\,.\ee
We shall say that $M$ is \emph{weight-finite} if, regarding it as a completely decomposable $\uqhhzeroslt$-module, its $\Sp(M)$ is finite -- see definition-proposition \ref{defprop:CKeigen} for the definition of $\Sp$. 
A vector $v\in M-\{0\}$ is a \emph{highest $t$-weight vector} of $M$ if $v\in M_\alpha$ for some $\alpha\in A$ and, for every $r,s\in\Z$,
\be \Xsf+{1,r,s} . v = 0\,.\ee
We shall say that $M$ is \emph{highest $t$-weight} if $M\cong \qdaff(\mathfrak a_1) . v$ for some highest $t$-weight vector $v\in M-\{0\}$.
\end{defn}
\begin{defprop}
\label{defprop:htwtsp}
Let $M$ be a $t$-weight $\qdaff(\mathfrak a_1)$-module that admits a highest $t$-weight vector $v\in M-\{0\}$. Denote by $M_0$ the $t$-weight space of $M$ containing $v$. Then $M_0=\qdaff^0(\mathfrak a_1).v$ and, for every $r,s\in\Z$,
\be \Xsf+{1,r,s} .M_0 = \{0\}\,.\ee
We shall say that $M_0$ is a \emph{highest $t$-weight space} of $M$. If in addition $M$ is simple, then it admits a unique -- up to isomorphisms of $\qdaff^0(\mathfrak a_1)$-modules -- highest $t$-weight space $M_0$.
\end{defprop}
\begin{proof}
It is an easy consequence of the triangular decomposition of $\qdaff(\mathfrak a_1)$ -- see proposition \ref{prop:triang} -- and of the root grading of $\qdaff(\mathfrak a_1)$ that, indeed, $\Xsf+{1,r,s}.\left (\qdaff^0(\mathfrak a_1).v\right ) =\{0\}$, for every $r,s\in\Z$. Now since $M$ is highest $t$-weight, we have $M\cong\qdaff(\mathfrak a_1).v$. By proposition \ref{prop:triang}, $M_0 \subset M \cong \qdaff^-(\mathfrak a_1) \qdaff^0(\mathfrak a_1).v$ and it follows that $M_0 \cong \qdaff^0(\mathfrak a_1).v$. Now, assuming that $M$ is simple and that it admits highest $t$-weight spaces $M_0$ and $M_0'$, we have that $\qdaff^-(\mathfrak a_1).M_0 \cong M\cong \qdaff^-(\mathfrak a_1).M_0'$ as $\qdaff(\mathfrak a_1)$-modules. In particular, $M_0\cong M_0'$ as $\qdaff^0(\mathfrak a_1)$-modules. 
\end{proof}
In view of the triangular decomposition of $\qdaff(\mathfrak a_1)$ -- see proposition \ref{prop:triang} --, the above proposition implies that any highest $t$-weight $\qdaff(\mathfrak a_1)$-modules $M$ is entirely determined as $M\cong \qdaff^-(\mathfrak a_1). M_0$, by the data of its highest $t$-weight space $M_0$, a cyclic $\ell$-weight $\qdaff^0(\mathfrak a_1)$-module. Now for any $v\in M_0-\{0\}$ such that $M_0 \cong \qdaff^0(\mathfrak a_1).v$, let $N_0$ be the maximal $\qdaff^0(\mathfrak a_1)$-submodule of $M_0$ not containing $v$ and set $L_0=M_0/N_0$~\footnote{$N_0$ clearly does not depend on the chosen generator $v$. Indeed, if $N_0$ contained a generator $v'$ of $M_0$, it would contain all the others, including $v$. It follows that $N_0$ and hence $L_0$ are both independent of $v$.}. Then, by construction, $L_0$ is a simple $\qdaff^0(\mathfrak a_1)$-module such that, as $\qdaff(\mathfrak a_1)$-modules, $M\cong \qdaff^-(\mathfrak a_1).L_0 \mod \qdaff^-(\mathfrak a_1).N_0$. We therefore make the following
\begin{defn}
\label{def:univqdaff}
We extend every simple (topological) $\ell$-weight $\qdaff^0(\mathfrak a_1)$-module $M_0$ into a $\qdaff^\geq(\mathfrak a_1)$-module by setting $\Xbsf+{1,r,s}.M_0=\{0\}$ for every $r, s\in \Z$. This being understood, we define the \emph{universal} highest $t$-weight $\qdaff'(\mathfrak a_1)$-module with highest $t$-weight space $M_0$ by setting
$$\mathcal M(M_0) = \widehat{\qdaff'(\mathfrak a_1)} \underset{\qdaff^\geq(\mathfrak a_1)}{\widehat\otimes} M_0$$
as $\qdaff'(\mathfrak a_1)$-modules. Denoting by $\mathcal N(M_0)$ the maximal (closed) $\qdaff(\mathfrak a_1)$-submodule of $\mathcal M(M_0)$ such that $M_0 \cap \mathcal N(M_0) = \{0\}$, we define the simple highest $t$-weight $\qdaff(\mathfrak a_1)$-module $\mathcal L(M_0)$ with highest $t$-weight space $M_0$ by setting $\mathcal L(M_0) \cong\mathcal M(M_0) / \mathcal N(M_0)$. It is unique up to isomorphisms.
\end{defn}
Classifying simple highest $t$-weight $\qdaff(\mathfrak a_1)$-modules therefore amounts to classifying those simple $\ell$-weight $\qdaff^0(\mathfrak a_1)$-modules $M_0$ that appear as their highest $t$-weight spaces. In the case of weight-finite $\qdaff(\mathfrak a_1)$-modules, this is achieved by the following

\begin{thm}
\label{thm:classification}
The following hold:
\begin{enumerate}[i.]
\item\label{thmi} Every weight-finite simple $\qdaff'(\mathfrak a_1)$-module $M$ is highest $t$-weight and can be obtained by twisting a type (1,0) weight-finite simple $\qdaff(\mathfrak a_1)$-module with an algebra automorphism from the subgroup of $\mathrm{Aut}(\qdaff'(\mathfrak a_1))$ generated by the algebra automorphisms $\tau$ and $\sigma$ of proposition \ref{prop:tausigma}.
\item\label{thmii} The type (1,0) simple highest $t$-weight  $\qdaff'(\mathfrak a_1)$-module $\mathcal L(M_0)$ is weight-finite if and only if its highest $t$-weight space $M_0$ is a simple $t$-dominant $\qdaff^0(\mathfrak a_1)$-module -- see proposition-definition \ref{def:dominant}.
\end{enumerate}
\end{thm}

\begin{proof}
Let $M$ be a weight-finite simple $t$-weight $\qdaff(\mathfrak a_1)$-module and assume for a contradiction that, for every $w\in M-\{0\}$, there exist $r, s\in \Z$ such that $\Xsf+{1,r,s}.w\neq0$. Then, there must exist two sequences $(r_n)_{n\in\N}, (s_n)_{n\in\N} \in\Z^\N$, such that 
$$0\notin \left\{ w_n = \Xsf+{r_1,s_1} \dots \Xsf+{r_n,s_n} .w : n\in\N\right\}\,.$$ 
Choosing $w\in M-\{0\}$ to be an eigenvector of $\Ksf+{1,0,0}$ with eigenvalue $\lambda\in\F^\times$ -- see definition-proposition \ref{defprop:CKeigen} for the existence of such a vector --, one easily sees from the relations that, for every $n\in\N$, $\Ksf+{1,0,0}.w_n = \lambda q^{2n} w_n$. It follows -- see definition-proposition \ref{defprop:CKeigen} -- that $\{\lambda q^{2n} : n\in\N\}\subseteq \Sp(M)$. A contradiction with the weight-finiteness of $M$. 
Thus, we conclude that there exists a highest $t$-weight vector $v_0\in M-\{0\}$ such that $\Ksf\pm{1,0,0}.v_0=\kappa_0^{\pm 1} v_0$ for some $\kappa_0\in\F^\times$. Obviously, $M\cong \qdaff(\mathfrak a_1). v_0$, for $\qdaff(\mathfrak a_1). v_0 \neq\{0\}$ is a submodule of the simple $\qdaff(\mathfrak a_1)$-module $M$. Thus $M$ is highest $t$-weight. Denote by $M_0 = \qdaff^0(\mathfrak a_1) . v_0$ its highest $t$-weight space. The latter is an $\ell$-weight $\qdaff^0(\mathfrak a_1)$-module. As such, it completely decomposes into countably many locally finite-dimensional indecomposable $\qdaff^{0,0}(\mathfrak a_1)$-modules that constitute its $\ell$-weight spaces. Over any of these, $\Csf^{1/2}$ must admit an eigenvector. But since $M$ is simple and $\Csf^{1/2}$ is central, the latter acts over $M$ by a scalar multiple of $\id$. 
It follows from definition-proposition \ref{defprop:CKeigen} that $\Csf$ acts over $M$ by $\id$ or $-\id$. In the former case, there is nothing to do; whereas in the latter, it is quite clear from proposition \ref{prop:tausigma} that, twisting the $\qdaff(\mathfrak a_1)$ action on $M$ by $\tau$, we can ensure that $\Csf$ acts by $\id$. It follows that $\Csf^{1/2}$ acts by $\id$ or $-\id$. Again, in the former case, there is nothing to do; whereas in the latter, twisting by $\sigma$, we can ensure that $\Csf^{1/2}$ acts by $\id$. Similarly, for every $m\in\N$, $\csf\pm{\pm m}$ must admit an eigenvector over any locally finite-dimensional $\ell$-weight space of $M_0$. But again, since $M$ is simple and $\csf\pm{\pm m}$ is central, the latter must act over $M$ by a scalar multiple of $\id$.

In any case, in view of (\ref{eq:K+K+}) and (\ref{eq:K+K-}), $\Ksf\pm{1,0,0}$ commutes with all the other generators of $\qdaff^0(\mathfrak a_1)$ and, since $M_0=\qdaff^0(\mathfrak a_1).v_0$, we have $\Ksf\pm{1,0,0}.w= \kappa_0^{\pm 1} w$ for every $w\in M_0$. Moreover, $M_0$ turns out to be a type 1 $\ell$-weight $\qdaff^0(\mathfrak a_1)$-module and, by definition-proposition \ref{defprop:CKeigen}, 
$$M_0 \subseteq \bigoplus_{\alpha\in A} \left \{v\in M : \Ksf\pm{1,0,0} . v = \kappa^{\pm 1}_0 v \quad \mbox{and} \quad \exists n\in\N^\times , \forall m\in\N^\times \quad \left (\Ksf\pm{1,0,\pm m} - \kappa^\pm_{\alpha, \pm m} \id \right )^n.v=0 \right \}$$
for some countable set of sequences $\left \{(\kappa^\pm_{\alpha, \pm m})_{m\in\N^\times}\in \F^{\N^\times}: \alpha\in A\right \}$. 
By proposition \ref{defprop:htwtsp}, 
\be\label{eq:X+M0}\Xsf+{1,r,s} .M_0 = \{0\}\,,\ee
for every $r,s\in\Z$. Pulling back with $\iota^{(0)}$ and $\iota^{(1)}$ respectively, we can simultaneously regard $M$ as a $\mathrm U_q(\mathrm L{\mathfrak a}_1)$-module for both of its Dynkin diagram subalgebras $\mathrm U_q(\mathrm L{\mathfrak a}_1)^{(0)}$ and $\mathrm U_q(\mathrm L{\mathfrak a}_1)^{(1)}$ -- see section \ref{Sec:qdaff}. Let $v\in M_0-\{0\}$ be a simultaneous eigenvector of the pairwise commuting linear operators in $\left \{\Ksf\pm{1,0,\pm m} : m\in\N \right \}$. Equation (\ref{eq:X+M0}) implies that $\x+1(z).v = \x-0(z).v=0$. Thus $v$ is a highest (resp. lowest) $\ell$-weight vector of $\qaff({\mathfrak a}_1)^{(1)}.v$ (resp. $\qaff({\mathfrak a}_1)^{(0)}.v$). The weight finiteness of $M$ now allows us to apply corollary \ref{cor:simplewffinite} to prove that the respective simple quotients of $\mathrm U_q(\mathrm L{\mathfrak a}_1)^{(0)}.v$ and $\mathrm U_q(\mathrm L{\mathfrak a}_1)^{(1)}.v$ containing $v$ are both finite-dimensional and isomorphic to a unique simple highest (resp. lowest) $\ell$-weight module $L(P_1)$ (resp. $\bar L(P_0)$). As a consequence of theorem \ref{thm:CP} and of proposition \ref{prop:lowestellweight}, we conclude that
$$\kk\pm0(z). v = q^{-\deg(P_0)} \left ( \frac{P_0(1/z)}{P_0(q^{-2}/z)} \right )_{|z|^{\mp 1}\ll1} v \qquad 
\mbox{and} \qquad \kk\pm1(z). v = q^{\deg(P_1)} \left ( \frac{P_1(q^{-2}/ z)}{P_1(1/z)} \right )_{|z|^{\mp 1}\ll1} v\,,$$
for some monic polynomials $P_0$ and $P_1$. On the other hand, pulling back with $\iota_m$ for every $m\in\Z$, we can regard $M$ as a $\mathrm U_q(\mathrm L{\mathfrak a}_1)$-module in infinitely many independent ways. Again, for every $m\in\Z$, $v$ turns out to be a highest $\ell$-weight vector for a unique simple weight finite, hence finite dimensional $\mathrm U_q(\mathrm L{\mathfrak a}_1)$-module. As such, it satisfies
$$\iota_m(\kk+1(z)).v = q^{\deg(Q_m)} \left(\frac{Q_m(q^{-2}/z)}{Q_m(1/z)}\right )_{|z|^{\mp1}\ll1}v\,,$$  
for some monic polynomial $Q_m$. Now since 
$$\iota_m(\kk\pm 1(z)) =  -\prod_{p=1}^{|m|} \cbsf\pm\left(q^{(1-2p)\sign(m)-1}z\right )^{\sign(m)} \Kbsf\mp{1,0}(\Csf^{-1/2}z)$$
and $\Psi(\kk\pm0(z)\kk\pm1(z))=\cbsf\pm(z)$, we must have
\bea q^{\deg(Q_m)} \left(\frac{Q_m(q^{-2}/z)}{Q_m(1/z)}\right )_{|z|^{\mp1}\ll1} &=& q^{\deg(P_1)+m(\deg(P_1)- \deg(P_0))}\left ( \frac{P_1(q^{-2} /z)}{P_1(1/z)} \right )_{|z|^{\mp 1}\ll1} \nn\\
&&\times \prod_{p=1}^{|m|}\left (\frac{P_1(q^{(2p-1)\sign(m) -1}/z)P_0(q^{(2p-1)\sign(m) +1}/z)}{P_1(q^{(2p-1)\sign(m) +1}/z)P_0(q^{(2p-1)\sign(m) -1}/z)}\right )_{|z|^{\mp 1}\ll1}
\label{eq:c=1}\eea
for every $m\in\Z^\times$. In the limit as $z^{-1}\to 0$, this implies $q^{\deg(Q_m)} = q^{\deg(P_1)+m(\deg(P_1)- \deg(P_0))}$ for every $m\in\Z$ and, consequently, $\deg(P_0) = \deg(P_1)=\deg(Q_m)$. After obvious simplifications, (\ref{eq:c=1}) becomes
\be \left(\frac{Q_m(q^{-2}/z)}{Q_m(1/z)}\right )_{|z|^{\mp1}\ll1} = \left ( \frac{P_1(q^{-2} /z) P_1(q^{-(1\pm 1)}/z)}{P_1(1/z) P_0(q^{-(1\pm 1)}/z)}\, \frac{P_0(q^{2m-(1\pm 1)}/z)}{P_1(q^{2m-(1\pm 1)}/z)}\right )_{|z|^{\mp 1}\ll1}\label{eq:c=1simp}
\ee
for every $m\in\Z^\times$. Now, $z^{-1}=0$ is not a root of $P(1/z)$ for any monic polynomial $P$. Moreover, $q$ being a formal parameter -- in case $q$ is regarded as a complex number, we shall assume that $1\notin q^{\Z^\times}$ --, it follows that the map $z^{-1}\mapsto q^{m}z^{-1}$ has no fixed points over the set of roots of a monic polynomial. Thus, for $|m|$ large enough, the respective sets of roots of $P_1(q^{-2} /z) P_1(q^{-(1\pm 1)}/z)$ and $P_1(q^{2m-(1\pm 1)}/z)$ are disjoint. Similarly, for $|m|$ large enough, the respective sets of roots of $P_1(1/z) P_0(q^{-(1\pm 1)}/z)$ and $P_0(q^{2m-(1\pm 1)}/z)$ are disjoint. It follows that, for $|m|$ large enough, on the r.h.s. of (\ref{eq:c=1simp}), cancellations can only occur between factors on opposite sides of the same fraction line. Now, either $P_0=P_1$ 
or $P_0\neq P_1$, in which case
$$\frac{P_1(1/z)}{P_0(1/z)} = \prod_{p=1}^n \frac{1-\alpha_p/z}{1-\beta_p/z}\,,$$
for some $n\in\N^\times$ such that $n\leq\deg(P_0) = \deg(P_1)$ and some $n$-tuples $(\alpha_p)_{p\in\rran{n}}$, $(\beta_p)_{p\in\rran{n}} \in\F^n$ such that
$$\left \{\alpha_p : p\in\rran{n} \right \} \cap \left \{\beta_p : p\in\rran{n} \right \} = \emptyset\,.$$
But then, we should have, for $|m|$ large enough,
$$\left(\frac{Q_m(q^{-2}/z)}{Q_m(1/z)}\right )_{|z|^{\mp1}\ll1} = \left ( \frac{P_1(q^{-2} /z)}{P_1(1/z) } \prod_{p=1}^n \frac{1-\alpha_pq^{-(1\pm 1)}/z}{1-\beta_pq^{-(1\pm 1)}/z} \prod_{p=1}^n \frac{1-\beta_pq^{2m-(1\pm 1)}/z}{1-\alpha_pq^{2m-(1\pm 1)}/z}\right )_{|z|^{\mp 1}\ll1}\,,$$
where, on the r.h.s., cancellations can only occur across the leftmost fraction line. A contradiction. \ref{thmi} 
follows. As for part of \ref{thmii}, we shall prove it in section \ref{Sec:ev}.
\end{proof}
Although we must postpone the proof of part \ref{thmii} of theorem \ref{thm:classification}, the proof above still makes it clear that
\begin{prop}\label{prop:highesttelldominant}
If a type (1,0) simple highest $t$-weight  $\qdaff(\mathfrak a_1)$-module $\mathcal L(M_0)$ is weight-finite, then its highest $t$-weight space $M_0$ is a simple $\ell$-dominant $\qdaff^0(\mathfrak a_1)$-module.
\end{prop}

\begin{prop}
\label{prop:ellweights}
Let $M$ be a $t$-weight $\qdaff(\mathfrak a_1)$-module and let $M_\alpha$ and $M_\beta$ be two $\ell$-weight spaces of $M$ such that, for some $m, n\in\Z$, $M_\alpha \cap \Xsf\pm{1,m,n}. M_\beta \neq \{0\}$. Then, there exists a unique $a\in\F^\times$ such that:
\begin{enumerate}
\item[i.] the respective $\ell$-weights $\kappa_\alpha^\varepsilon(z)$ and $\kappa_\beta^\varepsilon(z)$ of $M_\alpha$ and $M_\beta$ be related by
\be\label{eq:ellweightsellroots}\kappa_\alpha^\varepsilon(z) = \kappa_\beta^\varepsilon (z) A_{a}^\varepsilon(z)^{\pm 1} \,,\ee
where $\varepsilon \in \{-,+\}$ and 
$$A_{a}^\pm (z) = q^{2} \left (\frac{1-q^{-2} a/z}{1-q^{2} a/z}\right )_{|z|^{\pm 1}\ll 1}\,;$$
\item[ii.] $(z-a)^N M_\alpha \cap \Xbsf\pm{1,m}(z).M_\beta = \{0\}$ for some $N\in\N^\times$.
\end{enumerate}
\end{prop}

\begin{proof}
We keep the same notations as in the proof of proposition \ref{prop:Kellweight}. More specifically, we have two bases $\{v_i : i=1,\dots, \dim(M_\alpha)\}$ and $\{w_j : j=1,\dots, \dim(M_\beta)\}$ of $M_\alpha$ and $M_\beta$ respectively, in which
$$\forall i \in \rran{\dim M_\alpha}\,, \qquad\qquad \Kbsf\pm{1,0}(z)  . v_i = \kappa_\alpha^\pm(z) \sum_{k=i}^{\dim M_\alpha} \eta_{\alpha,i,k}^\pm(z) v_k \,,$$
$$\forall j \in \rran{\dim M_\beta}\,, \qquad\qquad \Kbsf\pm{1,0}(z) . w_j = \kappa_\beta^\pm(z) \sum_{l=j}^{\dim M_\beta} \eta_{\beta,j,l}^\pm(z) w_l \,,$$
for some $\eta_{\alpha,i,k}^\pm(z), \eta_{\beta,j,l}^\pm(z) \in \F[[z^{\pm 1}]]$, with $i,k\in\rran{\dim M_\alpha}$ and $j,l\in\rran{\dim M_\beta}$, such that $\eta_{\alpha,i,i}^\pm(z) =1$ for every $i\in\rran{\dim M_\alpha}$ and $\eta_{\beta,j,j}^\pm(z)=1$ for every $j\in\rran{\dim M_\beta}$.

Now, if $M_\alpha \cap \Xsf\pm{1,m,n}. M_\beta \neq \{0\}$, there must exist a largest nonempty subset $J\subseteq\rran{\dim M_\beta}$ such that, for every $j\in J$, $M_\alpha \cap \Xbsf\pm{1,m}(z).w_{j} \neq\{0\}$. Let $j_* = \max J$.  Obviously, for every $j\in J$, there must exist a largest nonempty subset $I(j)\subseteq \rran{\dim M_\alpha}$ such that, for every $j\in J$ and every $i\in I(j)$, $\F v_{i}\cap  \Xbsf\pm{1,m}(z).w_{j} \neq\{0\}$. Let $i_*(j) = \min I(j)$ and let for simplicity $i_*=i_*(j_*)$. Then, for every $j\in J$,
$$M_\alpha\cap\Xbsf\pm{1,m}(z).w_j = \sum_{i\in I(j)} \xi_{m,j,i}^\pm(z) v_i \,,$$ 
for some $\xi_{m, j, i}^\pm(z)\in \F[[z,z^{-1}]]-\{0\}$. When needed, we shall extend by zero the definition of $\xi_{m, j, i}^\pm(z)$ outside of the set of pairs $\{(j,i): j\in J, i\in I(j)\}$. Making use of the relations in $\qdaff(\mathfrak a_1)$ -- namely (\ref{eqbf:K+X+}) and (\ref{eq:K+X-}) --, we get, for every $j\in J$ and every $\varepsilon \in\{-,+\}$,
$$(z_1-q^{\pm 2}z_2)\Xbsf\pm{1,m}(z_1) \Kbsf\varepsilon{1, 0}(z_2)  . w_j  = (z_1q^{\pm 2} -z_2)\Kbsf\varepsilon{1, 0}(z_2)  \Xbsf\pm{1,m}(z_1) . w_j \,.$$
The latter easily implies that, for every $j\in J$ and every $i\in I(j)$,
\be\label{eq:xideltaeq}(z_1-q^{\pm 2} z_2) \kappa_\beta^\varepsilon(z_2) \sum_{\substack{l\in J\\l\geq j}} \eta_{\beta,j,l}^\varepsilon (z_2) \xi_{m,l,i}^\pm(z_1) = (z_1q^{\pm 2} -z_2) \kappa_\alpha^\pm(z_2) \sum_{\substack{k\in I(j)\\k\leq i}} \eta_{\alpha,k,i}^\varepsilon(z_2) \xi_{m,j,k}^\pm(z_1)\,. \ee
Taking $i=i_*$ and $j=j_*$ in the above equation immediately yields
$$\left [ (z_1-q^{\pm 2}z_2)\kappa_\beta^\varepsilon(z_2)-  (z_1q^{\pm 2} -z_2)\kappa_\alpha^\varepsilon(z_2)\right ] \xi_{m,j_*,i_*}^\pm(z_1)  =0\,.$$
The latter is equivalent to the fact that, for every $p\in\Z$,
\be\label{eq:Aa} \xi_{m,j_*,i_*, p}^\pm z \left (q^{\pm 2} \kappa_\beta^\varepsilon(z) - \kappa_\alpha^\varepsilon(z) \right ) = \xi_{m,j_*,i_*,p+1}^\pm \left (\kappa_\beta^\varepsilon(z) - q^{\pm 2} \kappa_\alpha^\varepsilon(z) \right )\,,\ee
where, as usual, we have set
$$\xi_{m,j_*,i_*, p}^\pm  = \res_z z^{p-1} \xi_{m,j_*,i_*}^\pm(z)\,.$$
Since $\xi_{m,j_*,i_*}^\pm(z) \neq 0$, there exists at least one -$p\in \Z$ such that $\xi_{m,j_*,i_*, p}^\pm \neq0$. Assuming that $\xi_{m,j_*,i_*, p+1}^\pm=0$, one easily derives a contradiction from (\ref{eq:Aa}) and, repeating the argument, one proves that $\xi_{m,j_*,i_*, p}^\pm\neq 0$ for every $p\in\Z$. Dividing (\ref{eq:Aa}) by $\xi_{m,j_*,i_*, p}^\pm$, one gets
$$ z \left (q^{\pm 2} \kappa_\beta^\varepsilon(z) - \kappa_\alpha^\varepsilon(z) \right ) = a\left (\kappa_\beta^\varepsilon(z) - q^{\pm 2} \kappa_\alpha^\varepsilon(z) \right )\,,$$
where we have set, for every $p\in\Z$, $\xi_{m,j_*,i_*, p+1}^\pm/\xi_{m,j_*,i_*, p}^\pm = a \in\F^\times$. \emph{i.} now follows. Moreover, we clearly have
$$\xi_{m,j_*,i_*}^\pm (z) = A_{m,j_*,i_*}^\pm \delta(z/a)\,,$$
for some $A_{m,j_*,i_*}^\pm \in \F^\times$. More generally, we claim that, 
\be\label{eq:claim} \forall j\in J\,,\forall i\in I(j)\,,\qquad\qquad  \xi_{m,j,i}^\pm (z) = \sum_{p=0}^{N(i,j)} A_{m,j,i,p}^\pm \delta^{(p)}(z/a) \,,\ee
for some $A_{m,j,i,p}^\pm \in\F$ and some $N(i,j)\in \N$. This is proven by a finite induction on $j$ and $i$. Indeed, making use of (\ref{eq:ellweightsellroots}), we can rewrite (\ref{eq:xideltaeq}) as
\be\label{eq:xideltasimp}(z_1-q^{\pm 2} z_2) (z_2-q^{\pm 2} a) \sum_{\substack{l\in J\\l\geq j}} \eta_{\beta,j,l}^\varepsilon (z_2) \xi_{m,l,i}^\pm(z_1) = (z_1q^{\pm 2} -z_2) (q^{\pm 2}z_2-a) \sum_{\substack{k\in I(j)\\k\leq i}} \eta_{\alpha,k,i}^\varepsilon(z_2) \xi_{m,j,k}^\pm(z_1)\,,\ee
for every $j\in J$ and every $i\in I(j)$. Now, assume that (\ref{eq:claim}) holds for every pair in 
$$\left \{(j,i): j\in J, \, i\in I(j), \quad j> j_0 \right \} \cup \left \{(j_0,i): i\in I(j_0),\quad  i\leq i_0 \right \}\,,$$
 for some $j_0\in J$ and some $i_0\in I(j_0)$ such that $i_0<\max I(j_0)$. Let $i_0'$ be the smallest element of $I(j_0)$ such that $i_0<i_0'$. It suffices to write (\ref{eq:xideltasimp}) for $j=j_0$ and $i=i_0'$, to get
\bea (z_1-a)z_2 (1-q^{\pm 4}) \xi_{m,j_0,i_0'}^\pm(z_1) 
&=& - (z_1-q^{\pm 2} z_2) (z_2-q^{\pm 2} a) \sum_{\substack{l\in J\\l> j_0}} \eta_{\beta,j_0,l}^\varepsilon (z_2) \xi_{m,l,i_0'}^\pm(z_1) \nn\\
&&+ (z_1q^{\pm 2} -z_2) (q^{\pm 2}z_2-a) \sum_{\substack{k\in I(j_0)\\k\leq i_0}} \eta_{\alpha,k,i_0'}^\varepsilon(z_2) \xi_{m,j_0,k}^\pm(z_1)\,.\eea
Combining the recursion hypothesis and lemma \ref{lem:deltaid} from the appendix, one easily concludes that (\ref{eq:claim}) holds for the pair $(j_0, i_0')$. Repeating the argument finitely many times, we get that it actually holds for all the pairs in $\left \{(j,i): j\in J, \, i\in I(j), \quad j\geq j_0 \right \}$. Now, either $j_0 = \min J$ and we are done; or $j_0>\min J$ and there exists a largest $j_0'\in J$ such that $j_0>j_0'$. Writing (\ref{eq:xideltasimp}) for $j=j_0'$ and $i=i_*(j_0')$, we get
\be (z_1-a)z_2 (1-q^{\pm 4})\xi_{m,j_0',i_*(j_0')}^\pm(z_1)= -(z_1-q^{\pm 2} z_2) (z_2-q^{\pm 2} a) \sum_{\substack{l\in J\\l\geq j_0}} \eta_{\beta,j_0',l}^\varepsilon (z_2) \xi_{m,l,i_*(j_0')}^\pm(z_1) \,.\nn\ee
Combining again the recursion hypothesis and lemma \ref{lem:deltaid}, we easily get that (\ref{eq:claim}) holds for $(j_0', i_*(j_0'))$. It is now clear that the claim holds for every $j\in J$ and every $i\in I(j)$. Letting $N=\max \{N(i,j): j\in J, \, i\in I(j)\}$, \emph{ii.} follows. Furthermore, for every $b\in\F-\{a\}$ and every $n\in\N$, we obviously have $(z-b)^n M_\alpha \cap \Xbsf\pm{1,m}(z).M_\beta \neq \{0\}$, thus making $a$ the unique element of $\F$ satisfying \emph{ii.}.
\end{proof}

\begin{cor}
The $\ell$-weights of any type $(1,0)$ weight-finite simple $\qdaff(\mathfrak a_1)$-module are all rational -- see definition \ref{def:dominant}.
\end{cor}
\begin{proof}
Let $M$ be a type $(1,0)$ weight-finite simple $\qdaff(\mathfrak a_1)$-module. By proposition \ref{prop:highesttelldominant}, its highest $t$-weight space $M_0$ is an $\ell$-dominant simple $\qdaff^0(\mathfrak a_1)$-module. Hence, $M\cong \mathcal L(M_0) \cong \qdaff^-(\mathfrak a_1).M_0$ and it easily follows by proposition \ref{prop:ellweights} that all the $\ell$-weights of $\mathcal L(M_0)$ are of the form
$$\kappa_\alpha^\pm(z)  \prod_{p=1}^N A_{a_p}^\pm(z)^{-1}\,,$$
for some $N\in\N$, some $a_1, \dots, a_N \in\F^\times$ and 
$$\kappa_\alpha^\pm(z)  = -q^{\deg P_\alpha} \left (\frac{P_\alpha(q^{-2}/z)}{P_\alpha(1/z)}\right )_{|z|^{\pm 1}\ll 1}\,,$$
for some monic polynomial $P_\alpha(1/z)\in\F[z^{-1}]$. Now, observe that
$$A_{a}^\pm (z)^{-1} =q^{-2} \left (\frac{1-q^{2} a/z}{1-q^{-2} a/z}\right )_{|z|^{\pm 1}\ll 1} = q^{-1} \left (\frac{1-q^{2}a/z}{1-a/z}\right )_{|z|^{\pm 1}\ll 1} q^{-1}\left ( \frac{1-a/z}{1-q^{-2}a/z}\right )_{|z|^{\pm 1}\ll 1}\,.$$
Hence, all the $\ell$-weights of $\mathcal L(M_0)$ are of the form
\be\label{eq:ellweights}\kappa_\beta^\pm(z) =- q^{\deg(P_\beta) - \deg(Q_\beta)} \left (\frac{P_\beta(q^{-2}/z) Q_\beta(1/z)}{P_\beta(1/z) Q_\beta(q^{-2}/z)}\right )_{|z|^{\pm 1} \ll 1}\,,\ee
for some relatively prime monic polynomials $P_\beta(1/z), Q_\beta(1/z) \in \F[z^{-1}]$, which concludes the proof.
\end{proof}
In view of remark \ref{rem:qchar}, we can therefore associate with any weigh-finite simple $\qdaff(\mathfrak a_1)$-module a $q$-character defined as the (formal) sum of the monomials corresponding to all its rational $\ell$-weights.

\begin{prop}
Let $M_0$ and $N_0$ be two $t$-dominant simple $\qdaff^0(\mathfrak a_1)$-modules such that $M_0\widehat\otimes N_0$ be simple. Then:
\begin{enumerate}[i.]
\item\label{it:M0NOtdom} $M_0\widehat\otimes N_0$ is a simple $t$-dominant $\qdaff^0(\mathfrak a_1)$-module of type $(1,0)$;
\item\label{it:ses} there exists a short exact sequence of $\qdaff(\mathfrak a_1)$-modules
$$\{0\}\to \mathcal N \to \mathcal L(M_0)\widehat{\otimes} \mathcal L(N_0) \to \mathcal L(M_0\widehat{\otimes} N_0)\to \{0\}\,;$$
\item\label{it:isom} if, in addition, $\mathcal L(M_0)\widehat{\otimes} \mathcal L(N_0)$ is simple, then
$$\mathcal L(M_0)\widehat\otimes \mathcal L(N_0) \cong \mathcal L(M_0\widehat\otimes N_0)\,.$$
\end{enumerate}
\end{prop}
\begin{proof}
$M_0$ and $N_0$ are both of type $(1,0)$ and (\ref{eq:coprodC}) and (\ref{eq:coprodcm}) respectively imply that so is $M_0\widehat\otimes N_0$. Similarly, they are both $\ell$-weight and $\ell$-dominant. Combining eqs. (\ref{eq:defppm}), (\ref{eq:deftp}), (\ref{eq:deftm}), (\ref{eq:coprodppm}), (\ref{eq:coprodtp}) and (\ref{eq:coprodtm}), we easily prove that
\be\label{eq:coprodKp}
\Delta^0(\Kbsf+{1,m}(z)) =-\sum_{k=0}^m \prod_{l=k+1}^m \cbsf-(zq^{-2l} \Csf_{(1)}^{1/2}) \Kbsf+{1,k}(z) \widehat\otimes \Kbsf+{1,m-k}(zq^{-2k}\Csf_{(1)}) \,,
\ee
\be\label{eq:coprodKm}
\Delta^0(\Kbsf-{1,-m}(z)) = -\sum_{k=0}^m \Kbsf-{1,-(m-k)} (zq^{-2k} \Csf_{(2)} ) \widehat\otimes \Kbsf-{1,-k}(z) \prod_{l=k+1}^m \cbsf+(zq^{-2l} \Csf_{(2)}^{1/2}) \,,
\ee
for every $m\in \N$. In particular, taking $m=0$, we have $\Delta^0(\Kbsf\pm{1,0}(z)) = -\Kbsf\pm{1,0} (z\Csf_{(2)}^{\frac{1\mp1}{2}}) \otimes \Kbsf\pm{1,0}(z\Csf_{(1)}^{\frac{1\pm 1}{2}})$. It follows that, if $\left\{M_{0,\alpha} : \alpha\in A\right \}$ and $\left\{N_{0,\beta} : \beta\in B\right \}$ are the countable sets of $\ell$-weights of $M_0$ and $N_0$ respectively, with respective Drinfel'd polynomials $\{P_\alpha:\alpha\in A\}$ and $\{P_\beta:\beta\in B\}$, then $\left\{M_{0,\alpha} \otimes N_{0,\beta}: \alpha\in A\,, \quad \beta\in B\right \}$ is the countable set of $\ell$-weights of $M_0\widehat\otimes N_0$. Moreover, the latter is obviously $\ell$-dominant since its Drinfel'd polynomials are in $\{P_\alpha P_\beta: \alpha\in A\, \quad \beta\in B\}$. Now let $\alpha, \alpha'\in A$, $\beta, \beta'\in B$ and let $P_\alpha$, $P_{\alpha'}$, $P_\beta$ and $P_{\beta'}$ be the Drinfel'd polynomials of $M_{0,\alpha}$, $M_{0,\alpha'}$, $N_{0,\beta}$ and $N_{0,\beta'}$ respectively and assume that 
\be\label{eq:tenstdom}\left (M_{0,\alpha} \otimes N_{0,\beta}\right ) \cap \Delta^0(\Kbsf\pm{1,\pm1}(z)) . \left (M_{0,\alpha'} \otimes N_{0,\beta'}\right )\neq\{0\}\,.\ee 
Then, writing (\ref{eq:coprodKp}) and (\ref{eq:coprodKm}) above with $m=1$, we get
$$\Delta^0(\Kbsf+{1,1}(z)) =- \cbsf-(zq^{-2} \Csf_{(1)}^{1/2}) \Kbsf+{1,0}(z) \widehat\otimes \Kbsf+{1,1}(z\Csf_{(1)})-\Kbsf+{1,1}(z) \widehat\otimes \Kbsf+{1,0}(zq^{-2}\Csf_{(1)}) \,,$$
$$\Delta^0(\Kbsf-{1,-1}(z)) = -\Kbsf-{1,-1} (z \Csf_{(2)} ) \widehat\otimes \Kbsf-{1,0}(z)  \cbsf+(zq^{-2} \Csf_{(2)}^{1/2}) -\Kbsf-{1,0} (zq^{-2} \Csf_{(2)} ) \widehat\otimes \Kbsf-{1,-1}(z)  \,.$$
Since both $M_{0,\alpha'}$ and $N_{0,\beta'}$ are $\ell$-weight spaces, it follows that
$$\Delta^0(\Kbsf\pm{1,\pm1}(z)) . \left (M_{0,\alpha'} \otimes N_{0,\beta'}\right ) \subseteq \left (\Kbsf\pm{1,\pm 1}(z). M_{0,\alpha'} \otimes N_{0, \beta'}\right ) \oplus \left ( M_{0,\alpha'} \otimes \Kbsf\pm{1,\pm 1}(z).N_{0, \beta'}\right ) \,,$$
Therefore, condition (\ref{eq:tenstdom}) holds only if the direct sum on the r.h.s. above has a non-vanishing intersection with $M_{0,\alpha} \otimes N_{0,\beta}$. But since the latter is an $\ell$-weight space, this happens only if either $M_{0,\alpha} \cap \Kbsf\pm{1,\pm 1}(z). M_{0,\alpha'} \neq\{0\}$ or $N_{0,\beta} \cap \Kbsf\pm{1,\pm 1}(z). N_{0,\beta'} \neq\{0\}$. The $t$-dominance of $M_0$ and $N_0$ implies that for the only $a\in \F^\times$ such that , either $P_{\alpha'}(1/a) = 0$ or $P_{\beta'}(1/a)=0$. In any case, $P_{\alpha'}(1/a) P_{\beta'} (1/a)=0$ and $M_0\widehat\otimes N_0$ is $t$-dominant. \ref{it:M0NOtdom} follows. By lemma \ref{lem:DeltaPsiXp}, it is clear that $\dot\Delta(\Xbsf+{1,r}(z)).\left (M_0\widehat\otimes N_0\right )=\{0\}$. Hence $M_0\widehat\otimes N_0$ is a highest $t$-weight space in $\mathcal L(M_0)\widehat\otimes \mathcal L(N_0)$. Let $\mathcal N$ denote the largest closed $\widehat{\qdaff'(\mathfrak a_1)}$-submodule of $\mathcal L(M_0)\widehat\otimes \mathcal L(N_0)$ such that $\mathcal N\cap \left (M_0\widehat\otimes N_0\right )=\{0\}$. \ref{it:ses} obviously follows. \ref{it:isom} is clear.
\end{proof}

\section{An evaluation homomorphism and evaluation modules}
\label{Sec:ev}
In this section, we construct an evaluation algebra $\widehat{\mathcal A}_t$ and an $F$-algebra homomorphism $\ev:\qdaff(\mathfrak{a}_1)\to\widehat{\mathcal A}_t$, that we shall refer to as the evaluation homomorphism.

\subsection{The quantum Heisenberg algebras $\mathcal H_t^+$ and $\mathcal H_t^-$}
\begin{defn}
The \emph{quantum Heisenberg algebra} $\mathcal H_t^\pm$ is the Hopf algebra generated over $\K(t)$ by
$$\left\{\gamma^{1/2}, \gamma^{-1/2}, \alpha_\pm, \alpha_\pm^{-1}, \alpha_{\pm,m} :  m\in\Z^\times \right \}\,, $$
subject to the relations,
$$\gamma^{1/2}, \gamma^{-1/2}, \alpha_{\pm}, \alpha_{\pm}^{-1}  \mbox{ are central,}$$
$$[\alpha_{\pm,-m} , \alpha_{\pm,n}] = -\frac{\delta_{m,n}}{m} [2m]_t \frac{\gamma^m-\gamma^{-m}}{t-t^{-1}}\,,$$
for every $m,n\in\Z^\times$, with comultiplication $\Delta$ defined by setting
$$\Delta(\gamma^{1/2}) =\gamma^{1/2}\otimes \gamma^{1/2}\,, \qquad \Delta(\gamma^{-1/2}) =\gamma^{-1/2}\otimes \gamma^{-1/2}\,,$$
$$\Delta(\alpha_\pm) = \alpha_\pm\otimes \alpha_\pm\,,\qquad \Delta(\alpha_\pm^{- 1}) = \alpha_\pm^{-1}\otimes \alpha_\pm^{-1}\,,$$ 
$$\Delta(\alpha_{\pm,m}) = \alpha_{\pm,m} \otimes \gamma^{|m|/2} + \gamma^{-|m|/2}\otimes \alpha_{\pm,m}\,,$$
for every $m,n\in\Z^\times$, antipode $S$ defined by setting
$$S(\gamma^{1/2}) = \gamma^{-1/2}\,,\qquad S(\gamma^{-1/2}) = \gamma^{1/2}\,,$$
$$S(\alpha_\pm) = \alpha_\pm^{-1}\,,\qquad S(\alpha_\pm^{-1}) = \alpha_\pm$$
$$S(\alpha_{\pm,m}) = -\alpha_{\pm,m}\,,$$
and counit $\varepsilon$ defined by setting
$$\varepsilon(\gamma^{1/2}) = \varepsilon(\gamma^{-1/2}) = \varepsilon(\alpha_{\pm}) =\varepsilon(\alpha_{\pm}^{-1}) = \varepsilon (1)=1\,,$$
$$\varepsilon(\alpha_{\pm,m})=0\,.$$
\end{defn}

\begin{defn}
In $\mathcal H_t^+$, we let
$$\Lbf+{} (z) = 1+ \sum_{m\in\N^\times} L^+_{-m} z^m =  \exp\left [-(t-t^{-1}) \sum_{m\in\N^\times} \alpha_{+,-m}(t^2z)^m \right ]\,,$$
$$\Rbf+{} (z) = \alpha_+ \left (1+ \sum_{m\in\N^\times} R^+_m z^{-m} \right ) = \alpha_+ \exp\left [(t-t^{-1}) \sum_{m\in\N^\times} \alpha_{+,m}(t^{-2}z)^{-m} \right ]\,.$$
Similarly, in $\mathcal H_t^-$, we let
$$\Lbf-{} (z) = \alpha_- \left (1+ \sum_{m\in\N^\times} L^-_{-m} z^m \right )= \alpha_- \exp\left [-(t-t^{-1}) \sum_{m\in\N^\times} \alpha_{-,-m}(t^{-2}z)^m \right ]\,, $$
$$\Rbf-{} (z) = 1+ \sum_{m\in\N^\times} R^-_m z^{-m}  =  \exp\left [(t-t^{-1}) \sum_{m\in\N^\times} \alpha_{-,m}(t^{2}z)^{-m} \right ]\,. $$
\end{defn}
\noi Then, we have the following equivalent presentation of $\mathcal H_t^\pm$.
\begin{prop}
$\mathcal H_t^\pm$ is the Hopf algebra generated over $\K(t)$ by $$\{\gamma^{1/2}, \gamma^{-1/2},L^\pm_{-m},  R^\pm_{m}: m\in\N\}$$ 
subject to the relations
$$ [\Lbf{\pm}{}(v) , \Lbf{\pm}{}(z) ] =[\Rbf{\pm}{}(v) , \Rbf{\pm}{}(z) ] = 0\,,$$
$$\Rbf\pm{}(v)\Lbf\pm{}(z) = \theta^\pm(z/v) \Lbf\pm{}(z) \Rbf\pm{}(v)\,,$$
where we have defined $\theta^\pm(z)\in\mathcal Z(\mathcal H_t)[[z]]$, by setting
$$\theta^\pm(z) = \left (\frac{(1-t^{2\pm 4}\gamma z)(1-t^{\pm4 -2}\gamma^{-1}z)}{(1-t^{\pm 4-2}\gamma z)(1-t^{2\pm4}\gamma^{-1} z)}\right )_{|z|\ll 1}\,.$$
Furthermore, we have
$$\Delta(\Lbf\pm{}(z))= \Lbf\pm{}(z\gamma^{1/2}_{(2)}) \otimes \Lbf\pm{}(z\gamma^{-1/2}_{(1)})\,,$$
$$\Delta(\Rbf\pm{}(z))= \Rbf\pm{}(z\gamma^{-1/2}_{(2)}) \otimes \Rbf\pm{}(z\gamma^{1/2}_{(1)})\,,$$
where, by definition,
$$\gamma^{1/2}_{(1)} = \gamma^{1/2} \otimes 1\,, \qquad \gamma^{-1/2}_{(1)} = \gamma^{-1/2} \otimes 1\,, \qquad \gamma^{1/2}_{(2)} = 1\otimes \gamma^{1/2}\,,\qquad \gamma^{-1/2}_{(2)} = 1\otimes \gamma^{-1/2}$$
and
$$S(\Lbf\pm{}(z)) = \Lbf\pm{}(z)^{-1}\,,\qquad S(\Rbf\pm{}(z)) = \Rbf\pm{}(z)^{-1}\,.$$
Finally, $\varepsilon(\Lbf\pm{}(z)) = \varepsilon(\Rbf\pm{}(z))=1$.
\end{prop}
\begin{proof}
This is an easy consequence of the definition of $\mathcal H_t^\pm$.
\end{proof}
\begin{rem}
Observe that $\theta^+(z)$ and $\theta^-(z)$ are not independent and that we actually have $\theta^-(z) = \theta^+(t^{-8}z)$. 
\end{rem}

\subsection{A PBW basis for $\mathcal H_t^\pm$}
\label{subsec:PBW}
For every $n\in\N^\times$, we let $\Lambda_n:=\{\lambda=(\lambda_1, \dots, \lambda_n) \in(\N^\times )^n:\lambda_1\geq\cdots\geq\lambda_n \}$ denote the set of $n$-partitions. We adopt the convention that $\Lambda_0 = \{\emptyset\}$ reduces to the empty partition and we let $\Lambda = \bigcup_{n\in\N} \Lambda_n$ be the set of all partitions.
\begin{prop}
Define, for every $\lambda\in\Lambda$,
\be L^\pm_\lambda = L^\pm_{-\lambda_1} \cdots L^\pm_{-\lambda_n}\,,\ee
\be R^\pm_\lambda = R^\pm_{\lambda_1} \cdots R^\pm_{\lambda_n}\,, \ee
with the convention that $L^\pm_\emptyset = R^\pm_\emptyset=1$. Then,
\be\label{eq:PBWbasis}\left \{\Phi_{\lambda, \mu}^\pm = L^\pm_{\lambda} R^\pm_{\mu} :  \lambda, \mu\in\Lambda \right\}\ee
is a $\K(t)[\gamma^{1/2}, \gamma^{-1/2}]$-basis for $\mathcal H_t^\pm$.
\end{prop}
\begin{proof}
The relations in $\mathcal H_t^\pm$ read, for every $m,n\in\N$,
$$[L^{\pm}_{-m}, L^{\pm}_{-n}] = [R^{\pm}_{m}, R^{\pm}_{n}] =0\,,$$
$$R^\pm_m L^\pm_{-n} = L^\pm_{-n} R^\pm_m + \sum_{p=1}^{\min(m,n)} \theta^\pm_p L^\pm_{p-n} R^\pm_{m-p}\,, $$
where, for every $p\in\N$, $\theta^\pm_p\in\K(t)[\gamma^{1/2}, \gamma^{-1/2}]$ can be obtained from
$$\theta^\pm(z) = 1+\sum_{p\in\N^\times} \theta^\pm_p z^p\,.$$
It is clear that any monomial in $\{L^\pm_{-m}, R^\pm_{m} :m\in\N\}$ can therefore be rewritten as a linear combination with coefficients in $\K(t)[\gamma^{1/2}, \gamma^{-1/2}]$ of elements in $\{\phi_{\lambda, \mu}^\pm :\lambda, \mu \in\Lambda\}$. The independence of the latter is clear. 
\end{proof}
A convenient way to encode the above basis elements is through $\mathcal H_t^\pm$-valued symmetric formal distributions. Let indeed, for every $n^+, n^-, m^+, m^-\in\N$, every $n^\pm$-tuple $\zbf^\pm = (z^\pm_1, \dots, z^\pm_{n^\pm})$ and every $m^\pm$-tuple $\zebf^\pm=(\zeta^\pm_1,\dots, \zeta^\pm_{m^\pm} )$ of formal variables,
$$\Phi^\pm(\zbf^\pm, \zebf^\pm) = \Lbf\pm{}(\zbf^\pm)  \Rbf\pm{}(\zebf^\pm) \,,$$
where we have set
$$\Lbf\pm{}(\zbf^\pm) = \prod_{p=1}^{n^\pm} \Lbf\pm{}(z^\pm_p)\,,$$ 
$$\Rbf\pm{}(\zebf^\pm) = \prod_{p=1}^{m^\pm} \Rbf\pm{}(\zeta^\pm_p) \,,$$
with the convention that if $n^\pm$ (resp. $m^\pm=0$), then $\Lbf\pm{}(\emptyset)=1$ (resp. $\Rbf\pm{}(\emptyset)=1$).
It turns out that 
$$\Phi^\pm(\zbf^\pm,\zebf^\pm)\in \mathcal H_t^\pm[[\zbf^\pm,(\zebf^\pm)^{-1}]]^{S_{n^\pm}\times  S_{m^\pm}}\,.$$ 
Indeed, owing to the commutation relations in $\mathcal H_t^\pm$, the formal distribution $\Phi^\pm(\zbf^\pm,\zebf^\pm)$ is symmetric in each of its argument tuples, $\zbf^\pm$ and $\zebf^\pm$ respectively; i.e. it is invariant under the natural action of $S_{n^\pm}\times S_{m^\pm}$ on its arguments. It is also clear that, for every $\lambda^\pm \in \Lambda_{n^\pm}$ and $\mu^\pm \in \Lambda_{m^\pm}$,
$$\Phi^\pm_{\lambda^\pm, \mu^\pm} = \res_{\zbf^\pm,\zebf^\pm} (\zbf^\pm)^{-1-\lambda^\pm}(\zebf^\pm)^{-1+\mu^\pm}\Phi^\pm(\zbf^\pm,\zebf^\pm) \,,$$
where we have set 
$$(\zbf^\pm)^{-1-\lambda^\pm} = \prod_{p=1}^{n^\pm} (z^\pm_p)^{-1-\lambda^\pm_p}\qquad \mbox{and} \qquad (\zebf^\pm)^{-1+\mu^\pm} = \prod_{p=1}^{m^\pm} (\zeta^\pm_p)^{-1+\mu^\pm_p}\,.$$

\subsection{The dressing factors $\Lbf\pm{m}(z)$ and $\Rbf\pm{m}(z)$}
\begin{defn}
\label{defn:LmRm}
For every $m\in\Z^\times$, we let
\be \Lbf\pm{m}(z) = \prod_{p=1}^{|m|} \Lbf\pm{}(zt^{\pm 2(1-2p)\sign(m)+2})^{\pm\sign(m)}\ee
\be \Rbf\pm{m}(z) = \prod_{p=1}^{|m|} \Rbf\pm{}(zt^{\pm 2(1-2p)\sign(m)+2})^{\pm\sign(m)}\ee
\end{defn}
\noi It easily follows that
\begin{prop}
\label{prop:RLrel}
In $\mathcal H_t^\pm$, for every $m,n\in \Z^\times$, we have
$$ [\Lbf{\pm}{m}(v) , \Lbf{\pm}{n}(z) ] =[\Rbf{\pm}{m}(v) , \Rbf{\pm}{n}(z) ] = 0\,,$$
$$\Rbf{\pm}{m}(v) \Lbf{\pm}{n}(z) = \theta^{\pm}_{m,n}(z/v)  \Lbf{\pm}{n}(z)\Rbf{\pm}{m}(v)\,,$$
where we have set
$$\theta^\pm_{m,n} (z) = \prod_{r=1}^{|m|}\prod_{s=1}^{|n|} \theta^\pm(zt^{\pm 2(1-2s)\sign(n) \mp 2(1-2r)\sign(m)})^{\sign(mn)}  \,. $$
Furthermore, we have, for every $m\in \Z^\times$,
$$\Delta(\Lbf\pm{m}(z))= \Lbf\pm{m}(z\gamma^{1/2}_{(2)}) \otimes \Lbf\pm{m}(z\gamma^{-1/2}_{(1)})\,,$$
$$\Delta(\Rbf\pm{m}(z))= \Rbf\pm{m}(z\gamma^{-1/2}_{(2)}) \otimes \Rbf\pm{m}(z\gamma^{1/2}_{(1)})\,.$$
\end{prop}
It is worth emphasizing that the $\Lbf\pm{m}(z)$ are not indepedent for all values of $m\in\Z^\times$ and that neither are the $\Rbf\pm{m}(z)$. Indeed, we have
\begin{lem}
\label{lem:LRidentities}
For every $m, n\in\Z^\times$,
\be \Lbf\pm{-m}(z)^{-1} = \Lbf\pm{m}(zt^{\pm 4m}) \ee
\be \Rbf\pm{-m}(z)^{-1} = \Rbf\pm{m}(zt^{\pm 4m}) \ee
\be \Lbf{\pm}{m}(zt^{\pm 4m})\Lbf{\pm}{n}(z) = \Lbf{\pm}{m+n}(zt^{\pm 4m})\ee
\be \Rbf{\pm}{m}(zt^{\pm 4m})\Rbf{\pm}{n}(z) = \Rbf{\pm}{m+n}(zt^{\pm 4m})\ee
\end{lem}

\subsection{The algebra $\mathcal B_t$}
\label{sec:At}
Remember the Hopf algebra $\breve{\mathrm U}_q(\mathrm L\mathfrak a_1)$ from definition \ref{def:defqaffdota1}. It has an invertible antipode and we denote by $\breve{\mathrm U}_q(\mathrm L\mathfrak a_1)^{\footnotesize \mathrm{cop}}$ its coopposite Hopf algebra.
\begin{prop}
\label{prop:uqactionHeisenberg}
The quantum Heisenberg algebra $\mathcal H_t^+$ (resp. $\mathcal H_t^-$) is a left $\breve{\mathrm U}_{t^2}(\mathrm L\mathfrak a_1)$-module algebra (resp. a left $\breve{\mathrm U}_{t^2}(\mathrm L\mathfrak a_1)^{\footnotesize \mathrm{cop}}$-module algebra) with
$$\kk{\varepsilon}1(v) \triangleright \gamma^{1/2}=\kk{\varepsilon}1(v) \triangleright \gamma^{-1/2}=0\,,$$
$$\kk{\varepsilon}1(v) \triangleright \Lbf\pm{}(z) = \lambda^{\varepsilon,\pm}(v,z) \Lbf\pm{}(z) \,,\quad  \kk{\varepsilon}1(v) \triangleright \Rbf\pm{}(z) = \rho^{\varepsilon,\pm}(v,z) \Rbf\pm{}(z) \,,$$
$$\x\varepsilon1(v) \triangleright \gamma^{1/2} =\x\varepsilon1(v) \triangleright \gamma^{-1/2} =\x\varepsilon1(v) \triangleright \Lbf\pm{}(z) =\x\varepsilon1(v) \triangleright \Rbf\pm{}(z) =0\,,$$
for $\varepsilon\in\{-,+\}$ and where we have set
$$\lambda^{\varepsilon,\pm}(v,z) = \left (\frac{t^{2\mp 2}v-t^{-2\pm2}z}{v-t^{\pm4} z}\right )_{|z/v|^{\varepsilon 1}\ll1}\qquad \mbox{and} \qquad \rho^{\varepsilon,\pm}(v,z) =\left (\frac{t^{\pm 4}v-z}{t^{2\pm 2}v-t^{-(2\pm 2)}z}\right )_{|z/v|^{\varepsilon 1}\ll 1}\,.$$
\end{prop}
\begin{proof}
One readily checks the compatibility with the defining relations of $\mathcal H_t^\pm$ and  $\breve{\mathrm U}_{t^2}(\mathrm L\mathfrak a_1)$.
\end{proof}
\begin{prop}
\label{prop:UqactiononH}
For every $m\in \Z^\times$ and every $\varepsilon\in\{-,+\}$, we have
$$\kk\varepsilon1(v)\triangleright \Lbf\pm m(z) = \lambda^{\varepsilon,\pm}_m(v,z) \Lbf\pm m(z)\,, \qquad \kk\varepsilon1(v)\triangleright \Rbf\pm m(z) = \rho^{\varepsilon,\pm}_m(v,z) \Rbf\pm m(z)\,,$$
$$\x\varepsilon1(v)\triangleright \Lbf\pm m(z) =\x\varepsilon1(v)\triangleright \Rbf\pm m(z) =0\,,$$
where we have set
$$\lambda^{\varepsilon,\pm}_m(v,z) = \left (\frac{t^{-2(1\mp 1)m}v-t^{\pm4-2(1\pm 1)m}z}{v-t^{\pm 4}z}\right)_{|z/v|^{\varepsilon 1}\ll 1}$$
and
$$\rho^{\varepsilon, \pm}_m(v,z) =\left (\frac{t^{\pm 4}v-z}{t^{\pm 4 -2(1\mp 1)m}v-t^{-2(1\pm 1)m}z}\right )_{|z/v|^{\varepsilon 1} \ll 1} \,.$$
\end{prop}
\begin{proof}
This is readily checked making use of definition \ref{defn:LmRm}, of the Hopf algebraic structures of $\breve{\mathrm U}_{t^2}(\mathrm L\mathfrak a_1)$ and $\breve{\mathrm U}_{t^2}(\mathrm L\mathfrak a_1)^{\footnotesize \mathrm{cop}}$, of the $\breve{\mathrm U}_{t^2}(\mathrm L\mathfrak a_1)$-module algebra structures of $\mathcal H_t^+$ and of the $\breve{\mathrm U}_{t^2}(\mathrm L\mathfrak a_1)^{\footnotesize \mathrm{cop}}$-module algebra structure of $\mathcal H_t^-$.
\end{proof}

\begin{defprop}
We denote by $\mathcal H_t^+\rtimes \breve{\mathrm U}_{t^2}(\mathrm L\mathfrak a_1)\stackrel{{\footnotesize \mathrm{cop}}}{\ltimes} \mathcal H_t^-$ the associative $\F$-algebra obtained by endowing $\mathcal H_t^+\otimes \breve{\mathrm U}_{t^2}(\mathrm L\mathfrak a_1)\otimes \mathcal H_t^-$ with the multiplication given by setting, for every $h_+, h_+'\in \mathcal H_t^+$, every $h_-, h_-'\in\mathcal H_t^-$ and every $x, x'\in \breve{\mathrm U}_{t^2}(\mathrm L\mathfrak a_1)$,
$$\left (h_+\otimes x\otimes h_-\right ).\left (h_+'\otimes x'\otimes h_-'\right ) = \sum h_+\left (x_{(1)}\triangleright h_+' \right )\otimes x_{(2)}x' \otimes h_-\left ( x_{(3)}\triangleright h_-'\right )\,,$$
-- see proposition \ref{prop:UqactiononH} for the definition of the $\breve{\mathrm U}_{t^2}(\mathrm L\mathfrak a_1)$-module algebra structure of $\mathcal H_t^+$ and of the $\breve{\mathrm U}_{t^2}(\mathrm L\mathfrak a_1)^{\footnotesize \mathrm{cop}}$-module algebra structure of $\mathcal H_t^-$.
In that algebra, $\{\gamma^{1/2}-t, \gamma^{-1/2}-t^{-1}\}$ generates a left ideal. The latter is actually a two-sided ideal since $\gamma^{\pm 1/2}$ is central and, denoting it by $(\gamma^{1/2}-t)$, we can set $\breve{\mathcal B}_t =\mathcal H_t^+\rtimes \breve{\mathrm U}_{t^2}(\mathrm L\mathfrak a_1)\stackrel{{\footnotesize \mathrm{cop}}}\ltimes \mathcal H_t^-/(\gamma^{1/2}-t)$.
\end{defprop}
\begin{proof}
Making use of the coassociativity of the comultiplication $\Delta$, it is very easy to prove that, with the above defined multiplication, $\mathcal H_t^+\rtimes \breve{\mathrm U}_{t^2}(\mathrm L\mathfrak a_1)\stackrel{{\footnotesize \mathrm{cop}}}{\ltimes} \mathcal H_t^-$ is actually an associative $\F$-algebra. 
\end{proof}
\begin{prop}\label{prop:loopinB}
Setting $x\mapsto 1\otimes x\otimes 1$, for every $x\in \breve{\mathrm{U}}_{t^2}(\mathrm L\mathfrak a_1)$, defines a unique injective $\K(t)$-algebra homomorphism $\breve{\mathrm{U}}_{t^2}(\mathrm L\mathfrak a_1) \hookrightarrow \breve{\mathcal B}_t$. Similarly, $h\mapsto h\otimes 1\otimes 1$ and $h\mapsto 1\otimes 1\otimes h$ define unique injective $\K(t)$-algebra homomorphisms $\mathcal H_t^+\hookrightarrow \breve{\mathcal B}_t$ and $\mathcal H_t^-\hookrightarrow \breve{\mathcal B}_t$ respectively.
\end{prop}
\begin{rem}
We shall subsequently identify $\breve{\mathrm{U}}_{t^2}(\mathrm L\mathfrak a_1)$, $\mathcal H_t^+$ and $\mathcal H_t^-$ with their respective images in $\breve{\mathcal B}_t$ under the injective algebra homomorphisms of the above proposition.
\end{rem}

\begin{prop}
\label{prop:crossrel}
In $\breve{\mathcal B}_t$, for every $m\in\Z^\times$ and every $\varepsilon\in\{-,+\}$, we have the following relations
\be\label{eq:RL} (v-t^{\pm 4} z)(v-t^{\pm 4(1+m-n)}z) \Rbf{\pm}{m}(v) \Lbf{\pm}{n}(z) =  (v-t^{\pm 4(1-n)}z)(v-t^{\pm 4(1+m)}z)\Lbf{\pm}{n}(z)\Rbf{\pm}{m}(v)\,,\ee
\be\label{eq:Lk}  (zt^{\pm 4}-v)\kk{\varepsilon}{1}(v)\Lbf\pm{m}(z) =(zt^{\pm 4 - 2(1\pm 1)m} -vt^{-2(1\mp 1)m}) \Lbf\pm{m}(z) \kk{\varepsilon}{1}(v)  \,,\ee
\be  (zt^{\pm 4}-v)  \x{\pm}{1}(v) \Lbf\pm{m}(z) = (zt^{\pm 4 - 2(1\pm 1)m} -vt^{-2(1\mp 1)m}) \Lbf\pm{m}(z) \x{\pm}{1}(v) \,,\ee
\be  \x{\pm}{1}(v) \Lbf\mp{m}(z) =\Lbf\mp{m}(z) \x{\pm}{1}(v) \,,  \ee
\be\label{eq:Rk}  (zt^{-2(1\pm 1)m} -vt^{\pm 4 - 2(1\mp 1)m}) \kk{\varepsilon}{1}(v)\Rbf\pm{m}(z)= (z-vt^{\pm 4})   \Rbf\pm{m}(z) \kk{\varepsilon}{1}(v)\,, \ee
\be  (zt^{-2(1\pm 1)m} -vt^{\pm 4 - 2(1\mp 1)m})  \x{\pm}{1}(v) \Rbf\pm{m}(z)=(z-vt^{\pm 4})  \Rbf\pm{m}(z) \x{\pm}{1}(v)  \,,\ee
\be\label{eq:R+x-} \x{\pm}{1}(v) \Rbf\mp{m}(z) =\Rbf\mp{m}(z) \x{\pm}{1}(v) \,,  \ee
\end{prop}
\begin{proof}
In order to prove (\ref{eq:RL}), it suffices to check that
$$\theta^\pm(z) = \left (\frac{(1-z)(1-t^{\pm8}z)}{(1-t^{\pm4}z)^2}\right )_{|z|\ll 1} \mod\mathcal I$$
and that subsequently, for every $m,n\in\Z^\times$,
$$\theta^\pm_{m,n} (z) =\left ( \frac{(1-t^{\pm 4(1-n)}z)(1-t^{\pm 4(1+m)}z)}{(1-t^{\pm 4} z)(1-t^{\pm 4(1+m-n)}z)}\right )_{|z|\ll 1}\mod\mathcal I\,. $$

As for the equations (\ref{eq:Lk} -- \ref{eq:R+x-}), they immediately follow from the definitions of $\mathcal H_t^+\rtimes  \breve{\mathrm{U}}_{t^2}(\mathrm L\mathfrak a_1)\stackrel{{\footnotesize \mathrm{cop}}}{\ltimes}  \mathcal H_t ^-$ and of the actions $\triangleright$ of $\breve{\mathrm{U}}_{t^2}(\mathrm L\mathfrak a_1)$ on $\mathcal H_t^+$ and $\mathcal H_t^-$  -- see proposiiton \ref{prop:UqactiononH}. E.g., we have, by definition,
\bea\x+{1}(v) \Lbf+{m}(z) &=& \left (1 \otimes \x+{1}(v)\otimes 1\right)\left (\Lbf+{m}(z)\otimes 1\otimes 1\right) = \sum \left (\x+{1}(v)_{(1)}\triangleright\Lbf+{m}(z)\right ) \otimes \x+{1}(v)_{(2)} \otimes \left(\x+{1}(v)_{(3)}\triangleright 1\right )
 \nn\\
&=&\left (\x+{1}(v)\triangleright\Lbf+{m}(z)\right ) \otimes 1\otimes 1+  \left (\kk-{1}(v)\triangleright\Lbf+{m}(z)\right ) \otimes \x+{1}(v) \otimes 1 \nn\\ &&+ \left (\kk-{1}(v)\triangleright\Lbf+{m}(z)\right ) \otimes \kk-{1}(v) \otimes \varepsilon(\x+1(v))1 \nn\\
&=& \lambda^+_m(v,z)\Lbf+{m}(z)\x+{1}(v)   \,,\nn\eea
and
\bea\x+{1}(v) \Lbf-{m}(z) &=& \left (1 \otimes \x+{1}(v) \otimes 1\right)\left (1\otimes 1\otimes \Lbf-{m}(z)\right) = \sum \left(\x+{1}(v)_{(1)}\triangleright 1\right )\otimes \x+{1}(v)_{(2)} \otimes \left (\x+{1}(v)_{(3)}\triangleright\Lbf-{m}(z)\right )
 \nn\\
&=&   \varepsilon(\x+{1}(v))1\otimes 1 \otimes \left (1 \triangleright\Lbf-{m}(z)\right )+ 1\otimes \x+{1}(v) \otimes \left (1 \triangleright\Lbf-{m}(z)\right )+1 \otimes \left (\x+{1}(v)\triangleright\Lbf-{m}(z)\right ) \otimes \kk-{1}(v)  \nn\\
&=& \Lbf-{m}(z)\x+{1}(v)   \,,\nn\eea
as claimed. 
\end{proof}
\begin{rem}
In addition to the above, we obviously have in $\breve{\mathcal B}_t$, all the relations of its subalgebra $\breve{\mathrm{U}}_{t^2}(\mathrm L\mathfrak a_1)$ and all the relations of its subalgebras $\mathcal H_t^+$ and $\mathcal H_t^-$ modulo $(\gamma^{1/2}-t)$. 
\end{rem}

\begin{defprop}
Let $\mathcal I$ be the left ideal of $\breve{\mathcal B}_t$ generated by
$$\left \{\res_{z_1, z_2} z_1^{-1+m}z_2^{-1+n}\left ( [\x + 1(z_1), \x - 1(z_2)] - \frac{1}{q-q^{-1}}\delta \left ( \frac{z_1}{z_2} \right )  \left [ \kk + 1(z_1) -  \kk - 1(z_1) \right ]\right ) : m,n\in \Z \right \}\,.$$
Then $\mathcal I . \breve{\mathcal B}_t \subseteq \mathcal I$ and $\mathcal I$ is a two-sided ideal of $\breve{\mathcal B}_t$. Set $\mathcal B_t = \breve{\mathcal B}_t/\mathcal I$.
\end{defprop}

\begin{proof}
In order to prove that $\mathcal I . \breve{\mathcal B}_t \subseteq \mathcal I$, it suffices to prove that, for any $x\in\breve{\mathcal B}_t$,
$$\left ( [\x + 1(z_1), \x - 1(z_2)] - \frac{1}{q-q^{-1}}\delta \left ( \frac{z_1}{z_2} \right )  \left [ \kk + 1(z_1) -  \kk - 1(z_1) \right ]\right ) x \in \mathcal I\,.$$
The latter easily follows by inspection, making use of the relevant relations in $\breve{\mathcal B}_t$ and $\breve{\mathrm{U}}_{t^2}(\mathrm L\mathfrak a_1)$, namely (\ref{eq:Lk} - \ref{eq:R+x-}) and (\ref{eq:relkpmkpm} - \ref{eq:relxpmxpm}).
\end{proof}
\begin{rem}
Thus, in addition to the relations in $\breve{\mathcal B}_t$, we have, in $\mathcal B_t$, 
$$[\x + 1(z_1), \x - 1(z_2)] = \frac{1}{q-q^{-1}}\delta \left ( \frac{z_1}{z_2} \right )  \left [ \kk + 1(z_1) -  \kk - 1(z_1) \right ]\,.$$
\end{rem}

\subsection{The completion $\widehat{\mathcal B}_t$ of $\mathcal B_t$}
Making use of its natural $\Z$-grading, we endow $\mathcal B_t$ with a topology, in the same way as 
we endowed $\qdaff(\mathfrak a_1)$ with its topology in section \ref{Sec:qdaff}. We denote by $\widehat{\mathcal B}_t$ the corresponding completion. Consequently, its subalgebra $\mathcal H_t^\pm$ inherits a topology and we denote by $\widehat{\mathcal H}_t^\pm$ its corresponding completion in that topology.

\subsection{The shift factors}

\begin{defn}
\label{def:Hm}
In $\widehat{\mathcal H}_t^\pm$, we define, 
$$\Hbf\pm{} (z) = \Lbf\pm{}(z) \Rbf\pm{}(z)\,.$$
Similarly, for every $m\in\Z^\times$, we let
$$\Hbf\pm m(z) = \prod_{p=1}^{|m|} \Hbf\pm{}(zt^{\pm 2(1-2p)\sign(m)+2})^{\pm\sign(m)}\,.$$
\end{defn}
\begin{lem}
\label{lem:Hidentities}
For every $m, n\in\Z^\times$,
\be \Hbf\pm{-m}(z)^{-1} = \Hbf\pm{m}(zt^{\pm 4m}) \ee
\be \Hbf{\pm}{m}(zt^{\pm 4m})\Hbf{\pm}{n}(z) = \Hbf{\pm}{m+n}(zt^{\pm 4m})\ee
\end{lem} 
\begin{proof}
Follows directly from the definition in the same way as lemma \ref{lem:LRidentities}.
\end{proof}

\begin{prop}
In $\widehat{\mathcal H}_t^\pm$, we have, for every $m,n\in\Z^\times$,
$$\Hbf\pm{m}(z) \Hbf\pm{n}(v) = \Theta_{m,n}^\pm(z,v) \Hbf\pm{n}(v) \Hbf\pm{m}(z)\,,$$
where
$$\Theta_{m,n}^\pm(z,v) = \frac{(v-t^{\pm 4}z)(v-t^{\pm 4(1+n-m)}z)(t^{\pm 4(1-n)}v-z)(t^{\pm 4(1+m)}v-z) }{(z-t^{\pm 4}v)(z-t^{\pm 4(1+m-n)}v)(t^{\pm 4(1-m)}z-v)(t^{\pm 4(1+n)}z-v)}\,.$$
\end{prop}
\begin{proof}
In view of definition \ref{def:Hm} and of the relations in proposition \ref{prop:RLrel}, it is clear that commuting $\Hbf\pm m(z)$ and $\Hbf\pm n(v)$ amounts to commuting, on one hand $\Lbf\pm m(z)$ and $\Rbf\pm n(v)$ and, on the other hand, $\Rbf\pm m(z)$ and $\Lbf\pm n(v)$. The result follows.
\end{proof}

\begin{prop}
For every $m\in \Z^\times$ and every $\varepsilon\in\{-,+\}$, we have
\be\label{eq:konHm}\kk\varepsilon1(v)\triangleright \Hbf\pm{\pm m}(z) = H_{m,z}^\varepsilon(v)^{\pm 1} \Hbf\pm{\pm m}(z)\,,\ee
\be\label{eq:xonHm} \x\varepsilon1(v)\triangleright \Hbf\pm m(z) =0\,.\ee
\end{prop}
\begin{proof}
The left $\breve{\mathrm U}_{t^2}(\mathrm L\mathfrak a_1)$-module algebra (resp. a left $\breve{\mathrm U}_{t^2}(\mathrm L\mathfrak a_1)^{\footnotesize \mathrm{cop}}$-module algebra) structure of $\mathcal H_t^+$ (resp. $\mathcal H_t^-$) -- see proposition \ref{prop:uqactionHeisenberg} -- is extended by continuity to $\widehat{\mathcal H}_t^+$ (resp. $\widehat{\mathcal H}_t^-$) Then, it suffices to check that, for every $m\in\Z^\times$ and every $\varepsilon\in\{-,+\}$,
$$\kk\varepsilon1(v)\triangleright \Hbf\pm{\pm m}(z) = \lambda^{\varepsilon,\pm}_{\pm m}(v,z)   \rho^{\varepsilon,\pm}_{\pm m}(v,z) \Hbf\pm{\pm m}(z)\,,$$
and that
$$H_{m,z}^\varepsilon(v)^{\pm 1} =  \lambda^{\varepsilon,\pm}_{\pm m}(v,z)   \rho^{\varepsilon,\pm}_{\pm m}(v,z)\,.$$
\end{proof}

\begin{cor}
\label{cor:kdpHm}
For every $m\in\Z$, every $p\in\N$ and every $\varepsilon\in\{-,+\}$, we have
$$\prod_{k=1}^{p+1}\left [ \kk\varepsilon1(v_k) - H_{m,z}^\varepsilon(v_k)^{\pm 1}\right ]  \triangleright \partial^p\Hbf\pm{\pm m}(z)=0 \,,$$
\end{cor}
\begin{proof}
It suffices to differentiate (\ref{eq:konHm}) $p$ times with respect to $z$ to obtain
$$\left [\kk\varepsilon1(v) - H_{m,z}^\varepsilon(v)^{\pm 1} \, \id \right ] \triangleright \partial^p\Hbf\pm{\pm m}(z) = \sum_{k=0}^{p-1} {p\choose{k+1}} \frac{\partial^{k+1}}{\partial z^{k+1}}\left [H_{m,z}^\varepsilon(v)^{\pm 1}\right ] \partial^{p-k-1} \Hbf\pm{\pm m}(z)\,.$$
The claim immediately follows.
\end{proof}

\begin{prop}
\label{prop:Hrel}
In $\widehat{\mathcal B}_t$, we have, for every $m,n\in\Z^\times$,
$$\Hbf+m(z)\Hbf-n(v) = \Hbf-n(v)\Hbf+m(z)\,,$$
$$(zt^{\pm 4}-v)(zt^{-2(1\pm1)m}-vt^{\pm 4 - 2(1\mp 1)m})\kk\varepsilon1(v) \Hbf\pm m(z) = (z-vt^{\pm 4})(zt^{\pm 4 -2(1\pm1)m}-vt^{- 2(1\mp 1)m})\Hbf\pm m(z)\kk\varepsilon1(v) \,, $$
$$(zt^{\pm 4}-v)(zt^{-2(1\pm1)m}-vt^{\pm 4 - 2(1\mp 1)m})\x\pm1(v) \Hbf\pm m(z) = (z-vt^{\pm 4})(zt^{\pm 4 -2(1\pm1)m}-vt^{- 2(1\mp 1)m})\Hbf\pm m(z)\x\pm1(v) \,, $$
$$\x\pm1(v) \Hbf\mp m(z) = \Hbf\mp m(z)\x\pm1(v) \,. $$
\end{prop}
\begin{proof}
This follows immediately from $[\Lbf\pm{}(z), \Lbf\mp{}(v)] =[\Lbf\pm{}(z), \Rbf\mp{}(v)]=[\Rbf\pm{}(z), \Rbf\mp{}(v)]=0$. 
\end{proof}

\subsection{The evaluation algebra $\widehat{\mathcal A}_t$}
\begin{defprop}
Let $\mathcal J$ denote the closed left ideal of $\widehat{\mathcal B}_t$ generated by
\be \left \{\res_z z^{m} \left [ \Hbf-{}(z) \left (\kk{+}{1}(zt^{-4}) - \kk{-}{1}(zt^{-4})\right ) -  \Hbf+{}(z)^{-1} \left (\kk{+}{1}(z) - \kk{-}{1}(z)\right )\right ] : m\in\Z\right \}\,.\label{eq:ideal}\ee
Then, $\mathcal J .\widehat{\mathcal B}_t \subseteq \mathcal J$, making $\mathcal J$ a closed two-sided ideal of $\widehat{\mathcal B}_t$, and we let $\widehat{\mathcal A}_t = \widehat{\mathcal B}_t/\mathcal J$.
\end{defprop}
\begin{proof}
In order to prove that $\mathcal J .\widehat{\mathcal B}_t \subseteq \mathcal J$, it suffices to check that, for every $x\in \widehat{\mathcal B}_t$,
 $$\left [ \Hbf-{}(z)\left (\kk{+}{1}(zt^{-4}) - \kk{-}{1}(zt^{-4})\right ) - \Hbf+{}(z)^{-1} \left (\kk{+}{1}(z) - \kk{-}{1}(z)\right )\right ] x \in \mathcal J\,.$$
The latter easily follows by inspection, making use of the relevant relations in $\widehat{\mathcal B}_t$, namely (\ref{eq:RL}--\ref{eq:R+x-}) in proposition \ref{prop:crossrel}.
\end{proof}
\begin{prop}
For every $m\in\Z$, the following relation holds in $\widehat{\mathcal A}_t$,
\be\label{eq:center} \Hbf-{-m}(z) \left [\kk{+}{1}(zt^{-4m}) -\kk{-}{1}(zt^{-4m})\right ]  = \Hbf+m(z)^{-1} \left [ \kk{+}{1}(z) -\kk{-}{1}(z) \right ]\,.\ee
\end{prop}
\begin{proof}
We prove (\ref{eq:center}) for $m\in \N^\times$ by induction. The case $m=1$ corresponds to the vanishing of the generators of the ideal $\mathcal J$, see (\ref{eq:ideal}). Assuming the result holds for some $m\in\N^\times$, we have
\bea
\Hbf-{-(m+1)}(z) \left [\kk{+}{1}(zt^{-4(m+1)}) -\kk{-}{1}(zt^{-4(m+1)})\right ]  &=& \Hbf-{}(z) \Hbf-{-m}(zt^{-4}) \left [\kk{+}{1}(zt^{-4(m+1)}) -\kk{-}{1}(zt^{-4(m+1)})\right ]  \nn\\
&=&\Hbf-{}(z) \Hbf+{m}(zt^{-4})^{-1} \left [\kk{+}{1}(zt^{-4}) -\kk{-}{1}(zt^{-4})\right ]  \nn\\
&=& \Hbf+{m}(zt^{-4})^{-1} \Hbf-{}(z) \left [\kk{+}{1}(zt^{-4}) -\kk{-}{1}(zt^{-4})\right ]  \nn\\
&=& \Hbf+{m}(zt^{-4})^{-1} \Hbf+{}(z)^{-1} \left [\kk{+}{1}(z) -\kk{-}{1}(z)\right ]  \nn\\
&=& \Hbf+{m+1}(z)^{-1}\left [\kk{+}{1}(z) -\kk{-}{1}(z)\right ]  \nn
\eea
The cases with $m\in-\N^\times$ follow by rewriting the above equation for $m\in\N^\times$ as
\be   \Hbf+m(z)\left [\kk{+}{1}(zt^{-4m}) -\kk{-}{1}(zt^{-4m})\right ]  = \Hbf-{-m}(z)^{-1} \left [ \kk{+}{1}(z) -\kk{-}{1}(z) \right ]\nn\ee
and making use of lemma \ref{lem:LRidentities}.
\end{proof}
\begin{rem}
In addition to the above relation, $\widehat{\mathcal A}_t$ obviously inherits the relations in $\widehat{\mathcal B}_t$ modulo $\mathcal J$. In particular, all the relations in proposition \ref{prop:crossrel} hold in $\widehat{\mathcal A}_t$.
\end{rem}

\subsection{The evaluation homomorphism}
\begin{prop}
There exists a unique continuous $\K$-algebra homomorphism $\ev:\qdaff(\mathfrak{a}_1)^{(-1)}\rightarrow\widehat{\mathcal A}_t$ such that, for every $m\in\N^\times$ and every $n\in\Z$,
\be\ev(q) = t^2\,, \ee
\be\label{eq:evKpm0} \ev(\Kbsf\pm{1,0}(z)) = -\kk{\mp}{1}(z)\,,\ee
\be\label{eq:evK} \ev(\Kbsf\pm{1,\pm m}(z)) = \Hbf{\pm}{\pm m}(z)\left [\kk{\pm}{1}(zt^{-4m}) - \kk{\mp}{1}(zt^{-4m})\right ] \,,\ee
\be \ev(\Xbsf\pm{1,n}(z)) = \Hbf\pm{n}(z) \x\pm{1}(zt^{\mp 4n})\,.\ee
We shall refer to $\ev$ as the \emph{evaluation homomorphism}. It is such that $\ev \circ \iota_0 = \id$ over $\mathrm{U}_{t^2}(\mathrm L\mathfrak a_1)$.
\end{prop}
\begin{proof}
It suffices to check all the defining relations of $\qdaff(\mathfrak a_1)$. E.g. we have, for every $m,n\in\Z$,
\be\label{eq:evX+X-}\left [\ev(\Xbsf+m(v)), \ev(\Xbsf-n(z)) \right ] = \frac{1}{t^2-t^{-2}}\delta\left (\frac{v}{zt^{4(m+n)}}\right)\Hbf{+}{m}(v)\Hbf{-}{n}(z)\left[\kk+1(vt^{-4m}) - \kk-1(zt^{4n}) \right ] \,.\ee
If $m+n=0$, making use of (\ref{eq:center}), we are done. Assuming that $m+n>0$, lemma \ref{lem:LRidentities} allows us to write
\bea\Hbf{+}{m}(zt^{4(m+n)})\Hbf{-}{n}(z)\left[\kk+1(zt^{4n}) - \kk-1(zt^{4n}) \right ] &=& \Hbf{+}{m}(zt^{4(m+n)})\Hbf{+}{-n}(z)^{-1}\left[\kk+1(z) - \kk-1(z) \right ]\nn\\
&=& \Hbf{+}{m}(zt^{4(m+n)})\Hbf{+}{n}(zt^{4n})\left[\kk+1(z) - \kk-1(z) \right ] \nn\\
&=&\Hbf{+}{m+n}(zt^{4(m+n)})\left[\kk+1(z) - \kk-1(z) \right ] \nn
\eea
so that, eventually,
\be\left [\ev (\Xbsf+m(v)), \ev(\Xbsf-n(z)) \right ] = \frac{1}{t^2-t^{-2}}\delta\left (\frac{v}{zt^{4(m+n)}}\right) \ev(\Kbsf+{m+n}(zt^{4(m+n)}))\,.\nn
\ee
A similar argument proves the case $m+n<0$.
\end{proof}
\noi We have the following obvious
\begin{cor}
For every $N\in\N$ there exists an algebra homomorphism $\ev_{(N)}:\qdaff(\mathfrak{a}_1)^{(N)}\rightarrow\widehat{\mathcal A}_t$ making the following diagram commutative.
$$\begin{tikzcd}
\cdots \arrow{r} & \qdaff(\mathfrak{a}_1)^{(N)}\arrow{r} \arrow[rrrd, "\ev_{(N)}"']& \qdaff(\mathfrak{a}_1)^{(N-1)}\arrow{r} \arrow[rrd, end anchor = north west, "\ev_{(N-1)}"]& \cdots \arrow[r] & \qdaff(\mathfrak{a}_1)^{(-1)} \arrow[d, "\ev"] \\
&&&& \widehat{\mathcal A}_t
\end{tikzcd}$$
We can furthermore define the algebra homomorphism $\ev_{(\infty)}:\qdaff(\mathfrak{a}_1)\rightarrow\widehat{\mathcal A}_t$ by
$$\ev_{(\infty)} = \lim_{\longleftarrow} \ev_{(N)}\,.$$
\end{cor}

\subsection{Evaluation modules}
Remember the surjective algebra homomorphism $\breve{\mathrm U}_q(\mathrm L\mathfrak a_1)\to {\mathrm U}_q(\mathrm L\mathfrak a_1)$ from proposition \ref{prop:Ubreve}. It allows us to pull back any  simple ${\mathrm U}_q(\mathrm L\mathfrak a_1)$-module $M$ into a simple $\breve{\mathrm U}_q(\mathrm L\mathfrak a_1)$-module. With that construction in mind, we have
\begin{prop}
Let $M$ be a simple finite dimensional ${\mathrm U}_q(\mathrm L\mathfrak a_1)$-module. Then,
\begin{enumerate}[i.]
\item\label{it:indmod} $\widehat{\mathcal H}_t^+ \otimes M \otimes \widehat{\mathcal H}_t^-$ is a $\mathcal H_t^+\rtimes \breve{\mathrm U}_{t^2}(\mathrm L\mathfrak a_1)\stackrel{{\footnotesize \mathrm{cop}}}{\ltimes} \mathcal H_t^-$-module with the action defined by setting, for every $h_+,h_+'\in\mathcal H_t^+$, every $h_-, h_-'\in\mathcal H_t^-$, every $x\in\breve{\mathrm U}_{t^2}(\mathrm L\mathfrak a_1)$ and every $v\in M$,
$$(h_+\otimes x\otimes h_-).(h_+'\otimes v\otimes h_-') = \sum h_+ \left (x_{(1)} \triangleright h_+'\right ) \otimes x_{(2)}.v \otimes h_- \left (x_{(3)}\triangleright h_-'\right )$$
and extending by continuity.
\item\label{it:Btmod} $\widehat{\mathcal H}_t^+ \otimes M \otimes \widehat{\mathcal H}_t^-$ descends to a $\mathcal B_t$-module.
\item\label{it:Atmod} 
$\left (\widehat{\mathcal H}_t^+ \otimes M \otimes \widehat{\mathcal H}_t^-\right )/\mathcal J.\left (\widehat{\mathcal H}_t^+ \otimes M \otimes \widehat{\mathcal H}_t^-\right )$ is an $\widehat{\mathcal A}_t$-module. It pulls back along $\ev$ to a $\qdaff'(\mathfrak a_1)$-module that we denote by $\ev^*(M)$.
\item\label{it:WF} As a $\qdaff'(\mathfrak a_1)$-module, $\ev^*(M)$ is weight-finite. 
\item\label{it:highesttweight} For any highest $\ell$-weight vector $v\in M-\{0\}$, the $\qdaff^0(\mathfrak a_1)$-module
$$\tilde M_0\cong \left (\widehat{\mathcal H}_t^+ \otimes \F v \otimes \widehat{\mathcal H}_t^-\right )/\mathcal J.\left (\widehat{\mathcal H}_t^+ \otimes \F v \otimes \widehat{\mathcal H}_t^-\right )\,,$$
is a highest $t$-weight space of $\ev(M)$. We denote by $M_0$ the simple quotient of $\tilde M_0$ containing $v$ and we let $\ev^*(M_0) = \qdaff'(\mathfrak a_1).M_0$.
\item\label{it:M0tdom} $M_0$ is $t$-dominant.
\end{enumerate}
\end{prop}
\begin{proof}
\ref{it:indmod} is readily checked. As for \ref{it:Btmod}, it suffices to check that $\mathcal I.\left (\widehat{\mathcal H}_t^+ \otimes M \otimes \widehat{\mathcal H}_t^-\right ) = \{0\}$. But the latter is clear when $M$ is obtained by pulling back a ${\mathrm U}_q(\mathrm L\mathfrak a_1)$-module over which the relation generating $\mathcal I$ is automatically satisfied. \ref{it:Atmod} is obvious. It easily follows from proposition \ref{prop:Hrel} that, for every $m\in\Z^\times$, $[k^\varepsilon_{1,0}, \Hbf\pm m(z)] =0$. Hence, $\Sp(\ev^*(M))=\Sp(M)$ and the weight finiteness of $\ev^*(M)$ follows from that of $M$, which proves \ref{it:WF}. It is clear that, for every $r\in\Z$, we have
\bea \ev(\Xbsf+{1,r}(z)).\left (\mathcal H_t^+\otimes v\otimes \mathcal H_t^-\right ) &=& \Hbf+ r(z) \x+1(zt^{-4r}) . \left (\widehat{\mathcal H}_t^+\otimes v\otimes \widehat{\mathcal H}_t^-\right ) \nn\\
&=& \sum \Hbf+ r(z)\left (\x+1(zt^{-4r})_{(1)} \triangleright \widehat{\mathcal H}_t^+ \right )\otimes \x+1(zt^{-4r})_{(2)}.v \otimes  \left (\x+1(zt^{-4r})_{(3)} \triangleright \widehat{\mathcal H}_t^- \right )\nn\\
&=&0\,.\nn\eea
\ref{it:highesttweight} follows. Denote by $P(1/z)\in\F[z^{-1}]$ the Drinfel'd polynomial associated with $v$ and let $\nu\in\N^{\F^\times}_f$ denote the multiset of its roots. Then,
\be\label{eq:konv}\kk\pm1(z) .v = -\kappa^\mp_0(z) v \,, \qquad \mbox{where} \qquad \kappa_0^\mp(z) = -t^{2\deg(P)} \left (\frac{P(t^{-4}/z)}{P(1/z)}\right )_{|z|^{\mp 1}\ll 1}\,.\ee
Moreover, the partial fraction decomposition
 $$\frac{P(t^{-4}/z)}{P(1/z)} =\prod_{a\in\F^\times} \frac{1}{(1-a/z)^{\nu(a)-\nu(at^4)}} =C_0 + \sum_{a\in\F^\times} \sum_{p=1}^{\nu(a)-\nu(at^4)}\frac{C_{p}(a)}{(1-a/z)^{p}}\,,$$
in which $C_0, C_{p}(a)\in\F$ and the product and sum over $a\in\F^\times$ are always finite since $P$ only has finitely many roots, allows us to write
\be\left [\kk+1(z) - \kk-1(z)\right ].v = t^{2\deg(P)}  \sum_{a\in\F^\times} \sum_{p=0}^{\nu(a)-\nu(at^4)-1}\frac{(-1)^{p+1} C_{p+1}(a)}{p! \, a^{p+1}}\delta^{(p)}\left (\frac{z}{a}\right ) v\,.\nn\ee
Letting $\tilde C_p(a)=(-1)^{p+1}t^{2\deg(P)} C_{p+1}(a) a^{-p-1}/p!$ for every $a\in\F^\times$ and every $p\in \rran{0, \nu(a)-\nu(at^{4})-1}$, it follows that, for every $m\in\N^\times$,
\be\label{ev:K+} \ev(\Kbsf+{1,m}(z)).\left (1\otimes v\otimes 1\right ) = t^{2\deg(P)}  \sum_{a\in\F^\times} \sum_{p=0}^{\nu(at^{-4m})-\nu(at^{4(1-m)})-1}\tilde{C}_{p}(at^{-4m})\delta^{(p)}\left (\frac{z}{a}\right ) \left (\Hbf+{m}(z) \otimes v\otimes 1 \right )\,,
\ee
\be\label{eq:evK-} \ev(\Kbsf-{1,-m}(z)).\left (1\otimes v\otimes 1\right ) =- t^{2\deg(P)}  \sum_{a\in\F^\times} \sum_{p=0}^{\nu(at^{-4m})-\nu(at^{4(1-m)})-1}\tilde{C}_{p}(at^{-4m})\delta^{(p)}\left (\frac{z}{a}\right ) \left (1\otimes v\otimes \Hbf-{-m}(z)  \right )\,.
\ee
Now, making use of (\ref{eq:evKpm0}), (\ref{eq:konv}) and of corollary \ref{cor:kdpHm}, one easily shows that, for every $p\in \N$ and every $a\in \F^\times$,
$$\prod_{k=1}^{p+1} \left [\ev (\Kbsf\pm{1,0}(z_k)) - H^\pm_{m,a}(z_k)\kappa_0^\pm(z_k)\, \id\right ].\left (\partial^p\Hbf+{m}(a) \otimes v\otimes 1 \right ) =0\,,$$
$$\prod_{k=1}^{p+1} \left [\ev (\Kbsf\pm{1,0}(z_k)) - H^\pm_{m,a}(z_k)^{-1}\kappa_0^\pm(z_k)\, \id\right ].\left (1 \otimes v\otimes \partial^p\Hbf-{-m}(a) \right ) =0\,,$$
thus proving that $\partial^p\Hbf+{m}(a) \otimes v\otimes 1$ (resp. $1 \otimes v\otimes \partial^p\Hbf-{-m}(a)$) is an $\ell$-weight vector in the $\ell$-weight space $\ev^*(M)_{\kappa_{(+,m,a)}}$ (resp. $\ev^*(M)_{\kappa_{(-,m,a)}}$) of $\ev^*(M)$ with $\ell$-weight $\kappa_{(+,m,a)}^\pm(z) = \kappa_0^\pm(z)H^\pm_{m,a}(z)$ (resp. $\kappa_{(-,m,a)}^\pm(z)= \kappa_0^\pm(z)H^\pm_{m,a}(z)^{-1}$), as expected from proposition \ref{prop:Kellweight}.

On the other hand, 
\bea &&\left \{ \Hbf-{}(z)\left[\kk+1(zt^{-4}) - \kk-1(zt^{-4})\right ]- \Hbf+{}(z)^{-1} \left[\kk+1(z) - \kk-1(z)\right ]\right \}.\left (1\otimes v\otimes 1\right ) \nn\\ 
&&
\qquad \qquad\qquad\qquad 
= \sum_{a\in\F^\times} \left \{ \sum_{p=0}^{\nu(at^{-4}) - \nu(a)-1} \tilde C_p(at^{-4}) \delta^{(p)}\left (\frac{z}{a}\right )  \left (1\otimes v\otimes \Hbf-{}(z)\right )  \right .
\nn\\&&
\qquad\qquad\qquad\qquad\qquad\qquad\qquad\qquad\qquad\qquad 
\left . - \sum_{p=0}^{\nu(a) - \nu(at^4)-1} \tilde C_{p}(a)  \delta^{(p)}\left (\frac{z}{a}\right )  \left (\Hbf+{}(z)^{-1}\otimes v\otimes 1\right ) \right \} \,.\nn
\eea
Thus, modulo $\mathcal J$, we have, for every $a\in\F^\times$,
$$\sum_{p=0}^{\nu(at^{-4}) - \nu(a)-1} \tilde C_p(at^{-4}) \delta^{(p)}\left (\frac{z}{a}\right )  \left (1\otimes v\otimes \Hbf-{}(z)\right )  = \sum_{p=0}^{\nu(a) - \nu(at^4)-1} \tilde C_p(a) \delta^{(p)}\left (\frac{z}{a}\right )  \left (\Hbf+{}(z)^{-1}\otimes v\otimes 1\right )\,.$$
The above equation makes it clear that every $a\in \F^\times$ such that $\nu(at^{-4}) > \nu(a)$ is a zero of order at least $\nu(at^{-4}) - 2\nu(a) + \nu(at^4)$ of $1\otimes v\otimes \Hbf-{}(z)$, unless $\nu(at^{-4}) - \nu(a) \leq \nu(a)-  \nu(at^4)$. Hence, in view of (\ref{eq:evK-}), we have $\ev^*(M)_{\kappa_{(-, 1,a)}} \cap \ev(\Kbsf-{1,-1}(z)) . (1\otimes v\otimes 1) =0$ unless $a\in \Dsf_1(\nu)=\{x\in\F^\times : \nu(xt^{-4}) >\nu(x) >\nu(xt^4)\}$. But the latter implies that $P(1/a)=0$. A similar reasoning applies to any $\ell$-weight vector in $\tilde M_0$ and $\tilde M_0$ is $t$-dominant by lemma \ref{lem:tdom}. Taking the quotient of $\tilde M_0$ to $M_0$ clearly preserves $t$-dominance and \ref{it:M0tdom} follows.
\end{proof}
By the universality of $\mathcal M(M_0)$ -- see definition \ref{def:univqdaff} -- and the above proposition, there must exist a surjective $\qdaff'(\mathfrak a_1)$-module homomorphism $\pi:\mathcal M(M_0)\twoheadrightarrow \ev^*(M_0)$. Restricting the latter to the (closed) $\qdaff'(\mathfrak a_1)$-submodule $\mathcal N(M_0)$ of $\mathcal M(M_0)$, we get the surjective $\qdaff'(\mathfrak a_1)$-module homomorphism $\pi_{|\mathcal N(M_0)}$, whose image naturally injects as a $\qdaff'(\mathfrak a_1)$-submodule in $\ev^*(M_0)$. The canonical short exact sequence involving $\mathcal N(M_0)$, $\mathcal M(M_0)$ and the simple quotient $\mathcal L(M_0)$ -- see definition \ref{def:univqdaff} -- allows us to define a surjective $\qdaff'(\mathfrak a_1)$-module homomorphism  $\tilde\pi$ to get the following commutative diagram,
$$\begin{tikzcd}
\{0\}\arrow{r}  & \mathcal N(M_0) \arrow{r} \arrow[d, "\pi_{|\mathcal N(M_0)}"] & \mathcal M(M_0) \arrow{r} \arrow[d,"\pi"] & \mathcal L(M_0) \arrow{r} \arrow[d,"\tilde\pi"] & \{0\}\\
\{0\}\arrow{r}  & \pi(\mathcal N(M_0)) \arrow{r} \arrow{d} & \ev^*(M_0) \arrow{r} \arrow{d} & \ev^*(M_0)/ \pi(\mathcal N(M_0)) \arrow{r}\arrow{d}& \{0\}\\
&\{0\}&\{0\}&\{0\}&
\end{tikzcd}$$
where columns and rows are exact. It is obvious that $\tilde\pi$ is not identically zero and, by the simplicity of $\mathcal L(M_0)$, we must have $\ker(\tilde\pi)=\{0\}$. Hence, $\tilde\pi$ is a $\qdaff'(\mathfrak a_1)$-module isomorphism and we have constructed the simple weight-finite $\qdaff'(\mathfrak a_1)$-modules $\mathcal L(M_0)$ as a quotient of the evaluation module $\ev^*(M_0)$. To see that all the simple weight-finite $\qdaff'(\mathfrak a_1)$-modules $\mathcal L(M_0)$ can be obtained in this way, it suffices to observe that, by proposition \ref{prop:simpleelldomuq0mods}, all the simple $\ell$-dominant $\qdaff^0(\mathfrak a_1)$-modules are of the form $L^0(P)$ for some monic polynomial $P$ and that, in the construction above, one can choose any $P$, simply by choosing the corresponding simple finite-dimensional ${\mathrm U}_q(\mathrm L\mathfrak a_1)$-module $M$. Therefore, as a consequence of the above proposition, the highest $t$-weight space of any simple weight-finite $\qdaff'(\mathfrak a_1)$-modules $\mathcal L(M_0)$ is $t$-dominant. This concludes the proof of part \ref{thmii} of theorem \ref{thm:classification} as well as that of theorem \ref{thm:elldomtdom}.

\newpage
\appendix
\section{A lemma about formal distributions}
In this short appendix, we prove the following
\begin{lem}
\label{lem:deltaid}
Let $m\in\{0,1\}$ and $n\in\N$, let $A(v)\in \F[[v]]-\{0\}$ be a non-zero formal power series and let $F(z) \in \F[[z,z^{-1}]]$ be a formal distribution such that
\be\label{eq:formdisteq}(z-a)(z-v)^m A(v)F(z) + \sum_{p=0}^n B_p(v) \delta^{(p)}(z/a)=0\,,\ee
for some non-zero scalar $a\in\F^\times$ and some formal power series $B_0(v), \dots, B_n(v) \in \F[[v]]$. Then,
$$F(z) = \sum_{p=0}^{n+1} f_p\delta^{(p)}(z/a)\,,$$
for some scalars $f_0, \dots, f_{n+1}\in\F$.
\end{lem}
\begin{proof}
Consider first the case where $m=0$. Then, multiplying (\ref{eq:formdisteq}) by $(z-a)^{n+1}$, we get 
$$(z-a)^{n+2} A(v) F(z)=0\,.$$ 
Since $A(v)\neq 0$, we can specialize at a non-zero $v$-mode and it follows that
$$F(z) = \sum_{p=0}^{n+1} f_p \delta^{(p)}(z/a)$$
for some scalars $f_0, \dots, f_{n+1}\in\F$.  Now consider the case where $m=1$. It follows from (\ref{eq:formdisteq}) that
$$(z-a) A(v) F(z) +\left (\frac{1}{z-v}\right )_{|v/z|\ll 1} \sum_{p=0}^n B_p(v) \delta^{(p)}(z/a) = C(z) \delta (z/v )\,,$$
for some formal distribution $C(z)\in \F [[z,z^{-1}]]$. But specializing the l.h.s. of the above equation to any $v$-mode of the form $v^{-p}$ with $p\in\N^\times$, we immediately get that $C(z)=0$. We are thus back to the previous case.
\end{proof}

\bibliography{wfta1}    

\bibliographystyle{amsalpha}    

\end{document}